\documentclass[11pt]{amsart}

\usepackage[english]{babel}
\usepackage[utf8]{inputenc}
\usepackage{fullpage}

\usepackage{amsmath}
\usepackage{amssymb}
\usepackage{amsthm}
\usepackage{mathtools} 
\usepackage{bm} 
\usepackage{physics} 
\usepackage{enumitem}
\usepackage{bbm} 

\usepackage[font=small,labelfont=bf]{caption}
\usepackage{graphicx}
\usepackage[section]{placeins} 
\usepackage{float} 
\usepackage{subfig}
\usepackage{fancyhdr}
\usepackage{tikz}
\usetikzlibrary{arrows,calc}

\tikzset{
>=stealth',
help lines/.style={dashed, thick},
axis/.style={<->},
important line/.style={thick},
connection/.style={thick, dotted},
}

\usepackage{hyperref}
\hypersetup{
    colorlinks,
    citecolor=black,
    filecolor=black,
    linkcolor=black,
    urlcolor=black
}
\usepackage{algorithm}
\usepackage{algpseudocode}

\newcommand{\N}{\mathbb{N}}

\newcommand{\R}{\mathbb{R}}
\newcommand{\C}{\mathbb{C}}
\def\HH{\mathcal{H}}
\def\TT{\mathcal{T}}

\newcommand{\vertiii}[1]{{\left\vert\kern-0.25ex\left\vert\kern-0.25ex\left\vert #1 
    \right\vert\kern-0.25ex\right\vert\kern-0.25ex\right\vert}}
   \newcommand\restr[2]{{
  \left.\kern-\nulldelimiterspace 
  #1 
  \vphantom{\big|} 
  \right|_{#2} 
  }} 

\let\hat\widehat

\makeatletter 
\renewcommand*\env@matrix[1][\arraystretch]{
  \edef\arraystretch{#1}%
  \hskip -\arraycolsep
  \let\@ifnextchar\new@ifnextchar
  \array{*\c@MaxMatrixCols c}}
\makeatother

\DeclareMathOperator*{\argmax}{arg\,max}

\newtheorem{theorem}{Theorem}[section]
\newtheorem{proposition}[theorem]{Proposition}
\newtheorem{lemma}[theorem]{Lemma}

\newtheorem{remark}[theorem]{Remark}
\newtheorem{definition}[theorem]{Definition}
\newtheorem{example}[theorem]{Example}

\def\revision#1{{#1}} 
\title{Convergence of adaptive stochastic collocation with finite elements}
\date{}
\author{Michael Feischl and Andrea Scaglioni}

\thanks{Funded by the Deutsche Forschungsgemeinschaft (DFG, German Research Foundation) – Project-ID 258734477 – SFB 1173 and the Austrian Science Fund (FWF) the SFB Taming complexity in
partial differential systems (grant SFB F65)\\
Institute of Analysis and Scientific Computing
TU Wien, Wiedner Hauptstra\ss e 8-10, 1040 Vienna}
\begin{document}

\begin{abstract}
 We consider an elliptic partial differential equation with a random diffusion parameter discretized by a stochastic collocation method in the parameter domain and a finite element method in the spatial domain. We prove for the first time convergence of a stochastic collocation algorithm which adaptively enriches the parameter space as well as refines the finite element meshes.
\end{abstract}

\maketitle

\section{Introduction} 
Partial differential equations with random data are a ubiquitous tool in the modeling of real life phenomena such as structural vibrations~\cite{ap2}, groundwater flow~\cite{flow},
and composite material behavior~\cite{ap1}. The efficient approximation of solutions of those equations is a challenging 
problem as it requires the approximation of high-dimensional functions in a parameter domains as well as low-dimensional but in general non-regular functions in the spatial domain.
While effective ways to generate the random data have been studied in~\cite{h2randfield, schwabrand}, we focus on the numerical approximation of the resulting solution of the PDE.

To that end, we consider an adaptive stochastic collocation algorithm for a random diffusion problem proposed in~\cite{guignard2018posteriori} and extend it to include spatial mesh refinement for a finite element method. We give the first proof of convergence of the adaptive algorithm to the exact solution and even derive some convergence rates as well as optimality statements. The main difficulty to overcome is the interplay of parametric enrichment and finite element refinement to ensure overall convergence.

Stochastic collocation is a so-called non-intrusive method, which has the big advantage that it does not require new solver algorithms, but reuses deterministic solvers only. Roughly speaking, the exact solution depends on a parametric variable (the random input) and a spatial variable. While the spatial dependence is resolved by standard finite element approximation, the parametric dependence is discretized by collocation. For each collocation point, we only need to solve a deterministic problem and therefore can reuse well tested finite element codes.

Problems of this kind have been considered in many prior works. See, e.g.,~\cite{scc2} for the (apparent) first appearance of the term \emph{stochastic collocation} in the field of computational fluid dynamics. The authors combine a polynomial chaos expansion with a collocation solver. This is generalized and formalized in~\cite{scc3}, where different strategies and collocation points are discussed, and error analysis is provided.  In~\cite{babuvska2007stochastic}, rigorous error analysis is developed for the discretization of a random PDE with a combination of standard finite elements and Gaussian collocation points (and also sparse grid points). In~\cite{scc5}, dimension dependent anisotropy of the solution is addressed by use of anisotropic sparse grids and in~\cite{scc4}, the optimal choice of the sparse grid parameters (the multi-index set) is discussed.

More recent works, which technically do not use stochastic collocation  but still deal with similar mathematical techniques and problems, include~\cite{hoqmc,qmcsoa} for quasi-Monte Carlo sampling approaches,~\cite{ml1,ml2} for multi-level methods, and~\cite{multiindex} for a multi-index method. Those methods have in common that they do not recover the full probability distribution of the exact solution but only compute a certain quantity of interest of it, e.g., the expectation or higher moments.

Another important and mathematically well-understood branch of methods for PDEs with random parameters are so-called stochastic Galerkin methods. Belonging to the class of intrusive methods, they employ a Galerkin method in both the spatial and the parametric domain. Algorithms have been proposed in~\cite{sgc1,sgc2,sgc3,sgc4,sgc5}.

The question of whether intrusive or non-intrusive algorithms are more advantageous is a fundamental one and depends on the specific problem setting and goals of the user. An overview is given in~\cite{Giraldi14}.

\bigskip

When dealing with PDEs with random data, adaptivity comes into mind for two reasons: First, spatial adaptivity is necessary to resolve singularities originating from geometric features (e.g., concave corners) and from irregular coefficients induced by the random input. Uniform meshes suffer from drastic reduction of convergence rate in the presence of such singularities, see, e.g.,~\cite{axioms} for an exhaustive overview on $h$-adaptive methods. 

Second, parametric adaptivity is necessary to resolve anisotropies in the random coefficient. The random input is often  parameterized on high-dimensional parameter domains, and
usually not all directions of that domain are equally important. Therefore, a straightforward tensor approximation approach would suffer dramatically from the curse of dimensionality. Here, an adaptive approach can outperform uniform methods significantly, see~\cite{aa1,aa2} for an overview.

Adaptive approximation of high-dimensional parameter domains has a long history. A starting point is often the work by Kolmogorov~\cite{kolmogorov}, in which the decomposition of a high-dimensional function into a sum of lower dimensional terms is discussed. Similar ideas have been pursued in statistics, we mention, e.g., so-called additive models~\cite{additive}, \emph{Multivariate adaptive regression splines}~\cite{mars}, or the \emph{ANOVA} decomposition~\cite{anova1,anova2}.  One of the first dimension adaptive algorithms for sparse grids can be found in~\cite{Gerstner03}. The method uses an error estimator that, in practice, agrees very well with the actual error but does not constitute a rigorous upper bound (it may underestimate the error by an arbitrarily large factor in exotic cases).  

\bigskip

For PDEs with random data, adaptive stochastic Galerkin algorithms have been investigated in~\cite{bespalov2019convergence,schwabadapt} with convergence and even optimality proofs.
Even low-rank tensor formats have been used in~\cite{Espig14, Dolgov15} to speed up the computation of stochastic Galerkin matrix and~\cite{Dolgov15} shows that in certain cases, the low-rank approximation of the full tensor product approximation can be stored and manipulated faster than some adaptive sparse approximations.

An adaptive sparse grid collocation algorithm based on a reliable error estimator was proposed in~\cite{guignard2018posteriori}. The work uses a sparse grid interpolation operator to discretize the parametric domain and proposes an error estimator which consists of a parametric estimator as well as a finite element estimator.   

A couple of recent works deal with similar approaches. A non-adaptive but true multi-level collocation method is proposed in~\cite{teckentrup}. The multi-level aspect allows the method to treat high-fidelity finite element approximations with low-fidelity parametric approximation and vice versa. This further reduces the impact of the curse of dimensionality and results in a very efficient algorithm. Such an approach could be built on top of the method proposed in the present manuscript in order to reduce the cost. An adaptive version of the multi-level algorithm was recently proposed in~\cite{scheichl}. While the algorithm is adaptive in both the parameter and spatial domain, the authors use an error estimator from~\cite{Gerstner03} that is not an upper bound for the error and hence can not prove convergence without extra assumptions. Finally, the very recent but independent work~\cite{eigel} analyzes the same parameter adaptive algorithm as in this work and proves convergence. However, they do not consider spatial adaptive refinement of the finite element meshes. As we show in numerical experiments in Section~\ref{sec:numerics}, spatial adaptivity clearly improves overall performance and hence should be included in any adaptive algorithm \revision{(however, a careful choice of the adaptive strategy is important as minor variations can lead to significant performance differences).}

The remainder of this work is organized as follows: We present the model problem in Section~\ref{sec:model} and describe the adaptive algorithm in Section~\ref{sec:def_adaptive_algo}.
In Section~\ref{sec:paramconv}, we prove convergence of the adaptive algorithm for the pure parameter enrichment problem (i.e., the problem considered in~\cite{guignard2018posteriori}), and Section~\ref{sec:conv_fully_discr_algo} proves the convergence of the full adaptive algorithm including spatial adaptivity \revision{(with one adaptive mesh per collocation point or one global adaptive mesh)}. Section~\ref{sec:numerics} presents some numerical experiments and in the final section, we draw some conclusions.

\subsection{Problem statement}\label{sec:model}
{ Consider an integer $d\geq2$ and an open bounded domain $D\subset \R^d$ with Lipschitz continuous boundary $\partial D$. Let $ \left(\Omega, \mathcal{F}, \mathcal{P}\right) $ be a complete probability space.}
{Let $Y_n:\Omega\rightarrow \R$ be independent random variables with ranges $\Gamma_n\coloneqq Y_n(\Omega)$ and densities $\rho_n: \Gamma_n\rightarrow \R_{\geq 0}$  for all $n\in 1,\ldots, N$. In the present work, we assume that the ranges $\Gamma_n$ are \emph{bounded} subsets of $\R$.}
Let $\Gamma \coloneqq \bigotimes_{n=1}^N \Gamma_n \subset \R^N$ and $\rho \coloneqq \bigotimes_{n=1}^N \rho_n $.
The triple ($\Gamma, \mathcal{B}(\Gamma), \rho(\bm{y}) \textrm{d}\bm{y}$) ($\mathcal{B}(\Gamma)$ the Borel $\sigma$-algebra on $\Gamma$) is a probability space.
Consider $f\in L^2(D)$ and $a:\Gamma \times D \rightarrow \R$ with the following properties:  
\begin{itemize}
 \item uniform boundedness
\begin{equation*}
\exists a_{\rm min}, a_{\rm max}\in \R_{>0} : a_{\rm min} \leq a(\bm{y}, x) \leq a_{\rm max}\qquad \rho\textrm{-a.e. } \bm{y}\in\Gamma, \forall x\in D
\end{equation*}
\item affine dependence on $\bm{y}\in\Gamma$
\begin{equation*}
\forall n\in 0,\ldots, N\ \exists\ a_n:D\rightarrow \R : a(\bm{y}, x) = a_0(x) + \sum_{n=1}^N a_n(x) y_n
\end{equation*}
\item {regularity in space $\nabla a(\bm{y},\cdot)|_T\in L^\infty(T)$ for all elements $T$ of a coarse \emph{initial} mesh $\TT_{\rm init}$ of $D$.}
\end{itemize}
{We consider the parametric weak formulation of the Poisson problem:}
Find $u:\Gamma\rightarrow V$ such that
\begin{equation} \label{pb:analyitic_problem}
\int_D a(x, \bm{y}) \nabla u(x, \bm{y}) \cdot \nabla v(x) \textrm{d} x = \int_D f(x) v(x) \textrm{d}x \qquad \forall v\in V, \ \rho\textrm{-a.e.}\ \bm{y}\in \Gamma.
\end{equation}
Here, $V$ denotes the Sobolev space $H^1_0(D)$ with the norm $\norm{v}_V \coloneqq \norm{\nabla v}_{L^2(D)}$.

Due to uniform ellipticity of the problem the exact solution is unique and (see also, e.g.,~\cite[Lemma~3.1]{babuvska2007stochastic}) there exists $\bm{\tau} \subset \R_{>0}^N$ such that $u:\Gamma\rightarrow V$ can be extended to a bounded holomorphic function on the set
\begin{equation}\label{def_anal_ext_domain}
\Sigma(\Gamma, \bm{\tau}) \coloneqq \left\{ \bm{z} \in \C^N : {\rm dist}(z_n, \Gamma_n) \leq \tau_n \ \forall n = 1,\ldots ,N\right\}.
\end{equation}
\def\HH{\mathcal{H}}
\subsection{The sparse grid stochastic collocation interpolant}
We aim at building a discretization of the solution $u$ of (\ref{pb:analyitic_problem}) in the space 
\begin{equation} \label{discr_space_abstact}
\mathbb{P}(\Gamma, W) \cong \mathbb{P}(\Gamma) \otimes W,
\end{equation}
where $\mathbb{P}(\Gamma)$ is a finite-dimensional polynomial space on $\Gamma$ and $W$ is a finite-dimensional subspace of $V$.
In order to do so, we fix a set $\mathcal{H}$ of distinct collocation points in $\Gamma$ and denote by $\left\{L_{\bm{y}}\right\}_{\bm{y}\in\mathcal{H}}$ the related set of Lagrange basis functions (i.e. the unique set of  polynomials of minimal degree over $\Gamma$ such that $L_{\bm{z}}(\bm{y}) = \delta_{\bm{y},\bm{z}}$ for any $\bm{y},\bm{z}\in\mathcal{H}$). 
By $\mathbb{P}(\Gamma)$, we denote the polynomial space spanned by $\left\{L_{\bm{y}}\right\}_{\bm{y}\in\mathcal{H}}$.
For any $\bm{y}\in \mathcal{H}$, we consider 
$\mathcal{T}_{\bm{y}}$, a shape-regular triangulation on $D$ depending on $\bm{y}$, 
and $V_{\bm{y}} \coloneqq S^1_{0}(\mathcal{T}_{\bm{y}})$, the classical finite elements space of piecewise-linear functions over $\mathcal{T}_{\bm{y}}$ with zero boundary conditions.
We denote by $U_{\bm{y}} \in V_{\bm{y}}$ the finite element solution of the problem for the parameter $\bm{y}$, i.e.,
\begin{subequations} \label{abstract_form_discrete_sol}
\begin{equation}
\int_D a(x, \bm{y}) \nabla U_{\bm{y}}(x) \cdot \nabla v(x) \textrm{d} x = \int_D f(x) v(x) \textrm{d}x \qquad \forall v\in V_{\bm{y}}.
\end{equation}
Finally, the discretization of $u$ takes the form
\begin{equation}
u_{\HH} (x, \bm{z}) = \sum_{\bm{y}\in\HH} U_{\bm{y}}(x) L_{\bm{y}} (\bm{z}).
\end{equation}
\end{subequations}
The number of degrees of freedom of $u_{\HH}$ is $\sum_{{\bm{y}}\in\HH} \dim\left(V_{\bm{y}}\right)$. The space $W$ from (\ref{discr_space_abstact}) will be  {the smallest subspace of $V$ that contains each} of the finite element spaces $\left\{V_{\bm{y}}\right\}_{{\bm{y}}\in\HH}$. 
The set of collocation nodes and the polynomial space are defined following the sparse grid construction, which we now describe briefly.
We start by considering a family of 1D nodes, i.e. a set $\mathcal{Y}^n \coloneqq \left\{ y^{(n)}_j \right\}_{j=1}^n\subset \R$ defined for any positive integer $n$. We require the family of $\mathcal{Y}^n$ to be nested, i.e. $\mathcal{Y}^n \subset \mathcal{Y}^{n+1}$ for any $n\in\N$.
The particular number of the quadrature nodes used in the algorithm is encoded in the function $m(\cdot)\colon \N\to \N$.
Finally, let $I\subset \N^N$ be a \emph{downward-closed} multi-index set, i.e.,
\begin{equation*}
\forall \bm{i}\in I, \quad \bm{i}-\bm{e}_n\in I \quad \forall n = 1,\ldots , N \textrm{ such that } i_n >1.
\end{equation*}
with $\bm{e}_n$ the $n$-th unit vector in $\mathbb{N}^N$.
The \emph{sparse grid interpolant} of a function $v\in C^0(\Gamma, V)$ is
\begin{equation} \label{definition_SG_interpol}
S_I[v](\bm{y}) := \sum_{\bm{i}\in I} \Delta^{m(\bm{i})}(v) (\bm{y}),
\end{equation}
where the \emph{hierarchical surplus} operator is defined as $\Delta^{m(\bm{i})} \coloneqq \bigotimes_{n=1}^N \Delta^{m(i_n)}$, 
the \emph{detail operator} is defined as $\Delta^{m(i_n)}  \coloneqq \mathcal{U}_n^{m(i_n)} - \mathcal{U}_n^{m(i_n-1)}$ and $\mathcal{U}_n^{m(i_n)} \colon C^0(\Gamma_n) \rightarrow \mathbb{P}_{m(i_n)-1}(\Gamma_n)$ is the Lagrange interpolant with respect to the nodes $\mathcal{Y}^{m(i_n)}\subset \Gamma_n$. Finally, we set $\mathcal{U}_n^0 \equiv 0 $ for all $n\in 1,\ldots, N$.

The polynomial space $\mathbb{P}(\Gamma)$ introduced in (\ref{discr_space_abstact}) corresponds to
\[\mathbb{P}_I(\Gamma) \coloneqq \sum_{\bm{i}\in I} \mathbb{P}_{m(\bm{i})-\bm{1}}(\Gamma) \quad\textrm{where}\quad\mathbb{P}_{m(\bm{i})-\bm{1}}(\Gamma) \coloneqq \bigotimes_{n=1}^N{\mathbb{P}_{m(i_n)-1}}(\Gamma_n).\]
The sparse grid stochastic collocation interpolant can be written as a linear combination of tensor product Lagrange interpolants (see, for instance, \cite{wasilkowski1995explicit}):
\begin{equation}\label{inclusion_exclusion_formula}
S_I[u](\bm{y}) = \sum_{\bm{i}\in I} c_{\bm{i}} \bigotimes_{n=1}^N \mathcal{U}_n^{m(i_n)}(u)(\bm{y}), \quad c_{\bm{i}} \coloneqq \sum_{\substack{{\bm{j}\in \left\{0,1\right\}^N} \\ {\bm{i}+\bm{j}\in I}}} (-1) ^{\vert\bm{j}\vert_1}.
\end{equation} 
The set of collocation points $\HH$ in~\eqref{inclusion_exclusion_formula} and also in~\eqref{abstract_form_discrete_sol} is referred to as \emph{sparse grid} and we will also denote it by $\mathcal{H}_I$ in order to make the dependence on $I$ explicit. 
The \emph{nestedness} of the family of 1D nodes $\mathcal{Y}^n$ makes $S_I[\cdot]$ interpolatory in the collocation nodes (see \cite[proposition 6]{barthelmann2000high})
\begin{equation*}
S_I[u](\bm{y}) = u(\bm{y})\qquad \forall \bm{y}\in\HH_I.
\end{equation*}
Due to this fact, (\ref{abstract_form_discrete_sol}) can be rewritten as 
\begin{equation}
{u_{\mathcal{H}}(x,\bm{z})=}u_I (x, \bm{z}) = S_I[u] (x, \bm{z}) = \sum_{\bm{y}\in \mathcal{H}_I} U_{\bm{y}}(x) L_{\bm{y}} (\bm{z}) \qquad x\in D, \bm{z}\in \Gamma.
\end{equation}

The nestedness is satisfied, e.g., by choosing  \emph{Clenshaw-Curtis (CC)} nodes to construct the sparse grid, i.e.
\begin{equation*}
y^{(m)}_j \coloneqq -\cos{\frac{\pi (j-1)}{m-1}} \qquad \forall j = 1,\ldots ,m,
\end{equation*}
with the \emph{doubling rule}
\begin{equation}\label{doubling_rule}
m(i) \coloneqq 
\begin{cases}
0 \qquad i=0,\\
1 \qquad i=1,\\
2^{i-1} + 1 \qquad i > 1.
\end{cases}
\end{equation}
We will stick with this particular choice for the remainder of this work for simplicity, remark however that other choices are possible (see, e.g.,~\cite{guignard2018posteriori}). { The essential properties of $\mathcal{Y}^n$ used in the proofs below are nestedness $\mathcal{Y}^n\subseteq \mathcal{Y}^{n+1}$ and the fact that the Lebesgue constants of the associated interpolation operators grow sub-exponentially.}

The requirement on the multi-index set $I$ to be downward-closed is needed to ensure that the sum (\ref{definition_SG_interpol}) is actually telescopic. 

Since $u$ is analytic in $\bm{y}$, we may consider the expansion (see again \cite{guignard2018posteriori})
\begin{equation}\label{approx_series_interpol}
u(\bm{y}) = \sum_{\bm{i}\in \N^N} \Delta^{m(\bm{i})}u (\bm{y}) \qquad \textrm{a.{e}. } \bm{y}\in\Gamma
\end{equation}
converging absolutely in $V$. As it will be central in the following discussion, we recall the definition of the \emph{margin of a multi-index set $I$}:
\[\mathcal{M}_I \coloneqq \left\{ \bm{i}\in \N^N : \bm{i}-\bm{e}_n\in I\ \textrm{for some } n\in 1,\ldots, N \ \textrm{such that } i_n > 1\right\}.\] 

\subsection{The adaptive stochastic collocation finite element algorithm} \label{sec:def_adaptive_algo}
The adaptive algorithm employs the error estimator proposed in \cite[Proposition 4.3]{guignard2018posteriori}. 
We recall that $u$  denotes the analytic solution of the problem (\ref{pb:analyitic_problem})
while the discrete solution is $S_I[U] = \sum_{\bm{y}\in \mathcal{H}_I} U_{\bm{y}} L_{\bm{y}}$.
By $U:\Gamma \rightarrow W$, we denote a function that takes the value $U_{\bm{y}}$ on the collocation point $\bm{y}\in \mathcal{H}_I$ (sometimes we will also use the notation $U(\bm{y}) = U_{\bm{y}}$).

The total estimator is composed of a \emph{parametric estimator}
\begin{equation} \label{eq:def_param_estim}
\zeta_{SC, I} \coloneqq \sum_{\bm{i}\in \mathcal{M}_I} \zeta_{\bm{i},I}, \qquad
\zeta_{\bm{i},I} \coloneqq \norm{\Delta^{m(\bm{i})} \left(a\nabla S_I[U]\right)}_{L^{\infty}(\Gamma, {L^2(D)})}
\end{equation}
(the gradient $\nabla$ here acts exclusively on the space variable $x\in D$) as well as a \emph{finite element estimator}
\begin{align} \label{eq:def_FE_estim}
\begin{split}
\eta_{ {\rm FE}, I} & \coloneqq \sum_{\bm{y}\in\HH_I} \eta_{\bm{y}} \norm{L_{\bm{y}}}_{{L^{\infty}(\Gamma)}}, \qquad
\eta_{\bm{y}} \coloneqq \left( \sum_{T\in\mathcal{T}_{\bm{y}}} \eta_{{\bm{y}}, T}^2\right)^{\frac{1}{2}},\\
\eta_{{\bm{y}},T}^2  &\coloneqq h_T^2 \norm{f+\nabla\cdot\left(a(\bm{y}_k) \nabla U_{\bm{y}}\right)}_{L^2(T)}^2
 +  \sum_{e\subset\partial T} h_e\norm{\frac{1}{2} \left[a(\bm{y})\nabla U_{\bm{y}}\cdot \bm{n}_e\right]_{\bm{n}_e}}^2_{L^2(e)},
\end{split}
\end{align}
where $[\cdot]_{\bm{n}_e}$ denotes the jump over the edge (face) in normal direction $\bm{n}_e$.
The combination of both yields a reliable upper bound, i.e.,
\begin{equation} \label{eq:fully_discrete_a_post_estimate}
\norm{u-S_{I}[U]}_{{L^\infty}(\Gamma, V)}  \leq \frac{1}{a_{\rm min}} \left({C}\eta_{ {\rm FE}, I} + \zeta_{SC, I}\right),
\end{equation}
where $a_{\rm min}>0$ appears in the equivalence relation between $H^1_0(D)$ and energy norm
\begin{align*}
a_{\rm min}^{1/2} \norm{v}_{H^1_0(D)} \leq \norm{a(\bm{y})^{\frac{1}{2}} \nabla v}_{L^2(D)} \leq a_{\rm max}^{1/2} \norm{v}_{H^1_0(D)} \quad \text{a.e. }\bm{y}\in\Gamma\text{ for all }v\in H^1_0(D)
\end{align*} 
{and $C>0$ depends only on the shape regularity of $\TT_{\rm init}$.}
We consider the following adaptive algorithm.
\begin{algorithm}[H]
\caption{$u_{\epsilon} \leftarrow$ \texttt{SCFE}($\epsilon, \theta, \alpha, \TT_{\rm init}$)}
\label{algo:SCFE}
\hspace*{\algorithmicindent} 
\begin{algorithmic}[1] 
\State $I_{-1} \coloneqq \emptyset$
\State $I_{0} \coloneqq \left\{\bm{1}\right\}$
\State compute finite element solution $U_{0,\bm{y}}$ on $\TT_{\rm init}$ for all $\bm{y}\in\mathcal{H}_{I_0}$
\For{$\ell=0,1,2,...$}
	\State $U_\ell \leftarrow$ \texttt{Refine\_FE\_spaces} $(I_\ell, U_\ell, \alpha, \theta)$ \label{algo:line:FE_ref}
	\State compute parametric estimators $\left(\zeta_{\bm{i}, I_\ell}\right)_{\bm{i}\in\mathcal{M}_{I_\ell}}$, $\zeta_{SC,I_\ell}$
	\State compute finite element estimator ${\eta}_{ {\rm FE}, I_\ell}$
	\If{{$a_{\rm min}^{-1}$}$\left(\zeta_{SC,I_\ell} + {C}\eta_{ {\rm FE}, I_\ell} \right)< \epsilon$} \label{algo:line:termination}
		\State \Return{$u_{\epsilon} \leftarrow S_{I_\ell} [U_\ell]$}
	\EndIf
	\State $(U_{\ell+1}, I_{\ell+1}) \leftarrow $ \texttt{Refine\_parameter\_space}($I_\ell, U_\ell,\left(\zeta_{\bm{i}, I_\ell}\right)_{\bm{i}\in\mathcal{M}_{I_\ell}}, \TT_{\rm init}$) \label{algo:line:param_ref}
\EndFor
\end{algorithmic}
\end{algorithm}
The algorithm consists of alternating between enriching the polynomial space $\mathbb{P}_I$ {in Line \ref{algo:line:param_ref} (Alg.~\ref{algo:SCFE})} and refining the finite element spaces corresponding to each collocation point \emph{independently from each other} {in Line \ref{algo:line:FE_ref} (Alg.~\ref{algo:SCFE})}. 
The intuitive idea behind this choice is the following: In order for the parameter enrichment routine to make a meaningful choice, the finite element solution in the collocation points has to be "close enough" to the exact solution. 
The algorithm terminates when the a-posteriori estimator falls below a given tolerance $\epsilon>0$ in Line \ref{algo:line:termination}~(Alg.~\ref{algo:SCFE}). {The reliable upper bound (\ref{eq:fully_discrete_a_post_estimate}) guarantees that the error of the discrete solution is also bounded by $\epsilon$.}

The sub-routine \texttt{Refine\_FE\_spaces} reads:
{\begin{algorithm}[H]
\caption{$U \leftarrow$\texttt{Refine\_FE\_spaces} ($I, U, \alpha, \theta$)}
\label{algo:refine_FE}
\hspace*{\algorithmicindent} 
\begin{algorithmic}[1]
\State compute finite element estimator $\left(\eta_{\bm{y}}\right)_{\bm{y}\in\mathcal{H}_I}$, $\eta_{ {\rm FE}, I}$
\State compute parametric estimator $\zeta_{SC,I}$
\State ${\rm Tol} \coloneqq \alpha \frac{1}{{\sum_{\bm{i}\in\mathcal{M}_I}\prod_{n=1}^N i_n }} \zeta_{SC, I}$  \label{algo_FE_ref:line:tolerance}
\While{$\eta_{ {\rm FE}, I} > {\rm Tol}$} \label{algo_FE_ref:line:loop}
	\State find minimal $\mathcal{K} \subseteq \bigsqcup_{\bm{y}\in \HH} \TT_{\bm{y}}:= \bigcup_{\bm{y}\in\mathcal{H}} \bigcup_{T_{\bm{y}}\in \TT_{\bm{y}}} (\bm{y},T_{\bm{y}})$ such that 
	\begin{align*}
	 \sum_{(\bm{y},T_{\bm{y}})\in\mathcal{K}} \eta_{\bm{y},T_{\bm{y}}}^2\norm{L_{\bm{y}}}_{L^\infty(\Gamma)} \geq \theta\eta_{ {\rm FE}, I}^2
	 \end{align*}
	 \label{algo_FE_ref:line:dorfler_marking}
	 \For{$\bm{y}\in\HH$} \label{algo_FE_reg:NVB}
	 \State refine $\mathcal{T}_{\bm{y}}$ with $\mathcal{K}_{\bm{y}}:=\{T\in\mathcal{T}_{\bm{y}}\,:\,(\bm{y},T)\in\mathcal{K}\}$ as marked elements
	 \EndFor
 	\State compute $U_{\bm{y}}$ over $\mathcal{T}_{\bm{y}}$
	\State compute finite element estimator $\left(\eta_{\bm{y}}\right)_{\bm{y}\in\mathcal{H}_I}$, $\eta_{ {\rm FE}, I}$ \label{algo_FE_ref:line:estim}
	\State compute parametric estimator $\zeta_{SC,I}$
	\State ${\rm Tol} \leftarrow \alpha \frac{1}{{\sum_{\bm{i}\in\mathcal{M}_I}\prod_{n=1}^N i_n}} \zeta_{SC, I}$ \label{algo_FE_ref:line:tolerance_inloop}
\EndWhile
\end{algorithmic}
\end{algorithm}}

The aim of this sub-routine is to refine the finite element solutions in the collocation points until the finite element estimator falls below the tolerance defined in Line \ref{algo_FE_ref:line:tolerance} { (Alg.~\ref{algo:refine_FE})}.
We use newest-vertex-bisection with mesh closure for mesh refinement (see, e.g.,~\cite{stevenson08} for further details). 
Observe that, since the tolerance depends on the parametric estimator $\zeta_{SC, I}$, which in turn depends on the discrete solution, the tolerance needs to be re-computed at every finite-element refinement. {This is necessary as one parametric refinement might uncover new features of the solution which need to be resolved in the finite-element refinement step. In practical computations, this rarely happens after the first few iterations of the algorithm. Note also that the linear convergence result in Proposition~\ref{prop:linconv} shows that starting from the initial mesh does not significantly increase the computational cost (at worst case, it contributes logarithmically). It might be possible to prove convergence without the scaling $1/\sum_{\bm{i}\in\mathcal{M}_I}\prod_{n=1}^N i_n$ of the tolerance in Line~\ref{algo_FE_ref:line:tolerance} of Algorithm~\ref{algo:SCFE} (the experiments in Section~\ref{sec:numerics} indicate that it is not necessary) however with the present techniques we didn't find a way to do that.}

In Section \ref{sec:conv_fully_discr_algo} we will prove that the sub-routine terminates (i.e. that the finite element estimator eventually falls below the tolerance) and that the choice of tolerance made in Line \ref{algo_FE_ref:line:tolerance} { (Alg.~\ref{algo:refine_FE})} is a sufficient condition for convergence.

Finally, the sub-routine \texttt{Refine\_parameter\_space} reads as follows:
\begin{algorithm}[H]
\caption{($U' ,I') \leftarrow$\texttt{Refine\_parameter\_space} ($I$, $U$, $\left(\zeta_{\bm{i}, I}\right)_{\bm{i}\in\mathcal{M}_{I}}$, $\TT_{\rm init}$)}
\label{algo:Refine_param_space}
\hspace*{\algorithmicindent} 
\begin{algorithmic}[1]
\State $\bm{i} \coloneqq \argmax_{\bm{i}\in\mathcal{M}_I} \mathcal{P}_{\bm{i}, I}$ \label{algo_param_enrich:line:prof_max}
\State $I' \coloneqq I\cup A_{\bm{i}, I}$ \label{algo_param_enrich:line:enlarge_I}
\State $U'\leftarrow $ update $U$ by computing finite element solution $U_{\bm{y}}$ on $\TT_{\rm init}$ for all $\bm{y}\in\mathcal{H}_{I'}\setminus \mathcal{H}_{I}$ \label{algo_param_enrich:line:comp_new_sols}
\end{algorithmic}
\end{algorithm}

The aim here is to enrich the polynomial space $\mathbb{P}_I$ as done in \cite[Algorithm 1]{guignard2018posteriori}. 
At each iteration, the algorithm enlarges the multi-index set $I$ by adding multi-indices from the margin of $I$ depending on the values of the pointwise error estimators $\left(\zeta_{\bm{i}, I}\right)_{\bm{i}\in I}$.
More precisely, in Line \ref{algo_param_enrich:line:prof_max} { (Alg.~\ref{algo:Refine_param_space})} we select a \emph{profit maximizer}, i.e. a multi index in the margin that maximizes a given \emph{profit function} $\mathcal{P}_{\bm{i}, I}$ (see below for some examples):
\begin{equation} \label{profit_maximizer}
\bm{i} = \argmax_{\bm{i}\in\mathcal{M}_I} \mathcal{P}_{\bm{i}, I}
\end{equation}
(in case more than one multi-index maximizes the profit, we pick the one that comes first in the lexicographic ordering).\\
Then, in Line \ref{algo_param_enrich:line:enlarge_I} { (Alg.~\ref{algo:Refine_param_space})} $I$ is enlarged by adding $A_{\bm{i}, I}$, the smallest subset of $\mathcal{M}_{I}$ containing $\bm{i}$ such that $I \cup A_{\bm{i}, I}$ is downward-closed. 
Finally, in Line \ref{algo_param_enrich:line:comp_new_sols} { (Alg.~\ref{algo:Refine_param_space})} we compute the finite element solution over the default mesh $\TT_{\rm init}$ corresponding to each new collocation point, while preserving the old ones.

We analyze two possible choices of profit:
\begin{itemize}
\item Workless profit: 
\begin{equation} \label{def:workless_profit}
\mathcal{P}_{\bm{i}, I} \coloneqq \sum_{\bm{j}\in A_{\bm{i},I}} \zeta_{\bm{j}, I}
\end{equation}
\item Profit with work: 
\begin{equation} \label{def:profit_work}
\mathcal{P}_{\bm{i}, I} \coloneqq \frac{\sum_{\bm{j}\in A_{\bm{i}, I}} \zeta_{\bm{j}, I}}{\sum_{\bm{j}\in A_{\bm{i}, I}} W_{\bm{j}}}
\end{equation}
where the \emph{work} is defined as $W_{\bm{j}} \coloneqq \prod_{n=1}^N \left( m(j_n) - m(j_{n-1}) \right)$.
\end{itemize} 

\section{Convergence of the parametric enrichment algorithm}\label{sec:paramconv}
We examine the convergence properties of a simplified version of Algorithm \ref{algo:SCFE}, also discussed in \cite{guignard2018posteriori}. In the present case, we suppose to be able to sample the function $u:\Gamma\rightarrow V$ for any fixed parameter $\bm{y}\in\Gamma$. Thus, a discrete solution is given by the sparse grid interpolant $S_I[u]\in\mathbb{P}_I(\Gamma, V)$, for a downward-close multi-index set $I\subset \N^N$.
Moreover, the a-posteriori estimator,
{which in the fully discrete setting was the sum of parametric estimator (\ref{eq:def_param_estim}) and finite element estimator (\ref{eq:def_FE_estim}),}
simplifies to $\zeta_{SC, I} \coloneqq \sum_{\bm{i}\in\mathcal{M}_I} \zeta_{\bm{i}, I}$ (no additional term accounting for the finite element discretization) where the pointwise estimator is
\begin{equation*} 
\zeta_{\bm{i}, I} \coloneqq \norm{\Delta^{m(\bm{i})} \left(a\nabla S_I[u]\right)}_{L^\infty(\Gamma, L^2(D))}.
\end{equation*}
In this setting, the
{reliable upper bound (\ref{eq:fully_discrete_a_post_estimate}) simplifies to}: $\norm{u - S_{I}[u]}_{L^{\infty}(\Gamma, V)}  \lesssim \zeta_{SC, I}.$
Workless-profit  and profit with work are defined analogously to (\ref{def:workless_profit}) and (\ref{def:profit_work}) respectively.
This  simplified version of the algorithm reads:
\begin{algorithm}[H]
\caption{$u_{\epsilon} \leftarrow SC(\epsilon)$}
\label{algo:SC_adaptive_algo}
\hspace*{\algorithmicindent} 
\begin{algorithmic}[1] 
\State $I \coloneqq \left\{\bm{1}\right\}$
\State $u_{\epsilon} \coloneqq S_I[u]$
\State compute $\zeta_{SC, I}$
\While{$\zeta_{SC, I} \geq \epsilon$}
	\State $\bm{i} \coloneqq  \argmax_{\bm{i}\in\mathcal{M}_I} \mathcal{P}_{\bm{i}, I}$
	\State $I \leftarrow I\cup A_{\bm{i}, I}$
	\State $u_{\epsilon} \leftarrow S_I[u]$
	\State compute new a-posteriori estimator $\zeta_{SC, I}$
\EndWhile
\end{algorithmic}
\end{algorithm}
\subsection{Preliminary results}
\subsubsection{Stability and convergence of the hierarchical surplus $\Delta^{m(\textbf{i})}$}
In this section we recall basic results on the hierarchical surplus operator $\Delta^{m(\bm{i})}$ (see for instance \cite{nobile2008sparse}). The analysis is carried out in the $L^\infty(\Gamma, V)$ norm as it is the most "stringent" among the $L^p(\Gamma, V)$ norms for $p\in [1,\infty]$. {We note that all the arguments below work with other choices of $p$.

We will first state 1D results (corresponding to the case $N=1$).
For $i\in\N$, the Lebesgue constant  $\lambda_{m(i)}$ of the interpolant $\mathcal{U}^{m(i)}$ satisfies the relation
\begin{equation} \label{stab_abstract}
\norm{\mathcal{U}^{m(i)}v}_{L^\infty(\Gamma, V)} \leq \lambda_{m(i)} \norm{v}_{L^\infty(\Gamma, V)} \qquad \forall v\in C^0(\Gamma, V).
\end{equation}
Moreover, since CC nodes and the \emph{doubling rule} (\ref{doubling_rule}) are used, it can be estimated as (see \cite{dzjadyk1983asymptotics})
\begin{equation}\label{estim_lebesgue_const}
\lambda_{m(i)}\leq 2i.
\end{equation}
Therefore, the relation (\ref{stab_abstract}) can be rewritten explicitly with respect to $i$ as
\begin{equation}\label{stab_explicit_i}
\norm{\mathcal{U}^{m(i)}v}_{L^\infty(\Gamma, V)} \lesssim i \norm{v}_{L^\infty(\Gamma, V)} \qquad \forall v\in C^0(\Gamma, V).
\end{equation}
The estimate~\eqref{stab_explicit_i} can be used to derive a stability estimate for the detail operator
\begin{equation*}
\norm{\left(\mathcal{U}^{m(i)} - \mathcal{U}^{m(i-1)}\right) v}_{L^\infty(\Gamma, V)}
\lesssim i \norm{v}_{L^\infty(\Gamma, V)}.
\end{equation*}
Moving to the general case $N\in\N$, we can now exploit the tensor product structure of $\Gamma\subset \R^N$ to obtain a stability estimate for the hierarchical surplus operator
\begin{equation}\label{stab_hierarchical_surplus}
\norm{\Delta^{m(\bm{i})}v}_{L^{\infty}(\Gamma, V)} \lesssim \left(\prod_{n=1}^N i_n\right) \norm{v}_{L^\infty(\Gamma, V)}.
\end{equation}
Since this estimate will be employed several times in the rest of the paper, we denote this bound on the norm of $\Delta^{m(\bm{i})}$ by
\begin{equation} \label{def:bound_stab_const}
\Lambda_{\bm{i}} \coloneqq \prod_{n=1}^N i_n.
\end{equation}

We derive another estimate of $\norm{\Delta^{m(\bm{i})} u}_{L^{\infty}(\Gamma, V)}$ that relies on the fact that $u:\Gamma\rightarrow V$ is analytic with respect to $\bm{y}$.
The tensor product structure of $\Gamma$ allows us again to start from a 1D results and then generalize to $N$ dimensions. 
We state a result that relates the best approximation error in $\mathbb{P}_{m}(\Gamma, V)$ to the size of the domain of the holomorphic extension of $u$ \eqref{def_anal_ext_domain}.
\begin{lemma}[{\cite[Lemma~4.4]{babuvska2007stochastic}}]
If $v\in C^0(\Gamma, V)$ and it exists $\tau>0$ such that $v$ admits an analytic extension to $\Sigma(\Gamma, \tau)$ (defined in (\ref{def_anal_ext_domain})), 
then for $m\in \N$ 
\begin{equation}
E_{m}(v) \coloneqq 
\min_{w\in \mathbb{P}_{m} (\Gamma, V)} \norm{v-w}_{L^\infty(\Gamma, V)} \leq 
\frac{2}{e^{\sigma}-1} e^{-\sigma m} \max_{z\in \Sigma(\Gamma, \tau)}\norm{v(z)}_{V}
\end{equation}
where $\sigma \coloneqq \log{\left( \frac{2\tau}{\vert \Gamma\vert} + \sqrt{1+\frac{4\tau^2}{\vert\Gamma\vert^2}}\right)} >0$. \hfill \qed
\end{lemma}

Since $\mathcal{U}^{m(i)}$ is exact on $\mathbb{P}_{m(i)-1}(\Gamma, V)$, its error can be expressed as (see \cite{barthelmann2000high})
\begin{equation*}
\norm{u - \mathcal{U}^{m(i)}u}_{L^\infty(\Gamma, V)} \leq \left( 1 + \lambda_{m(i)}\right) E_{m(i)-1}(u).
\end{equation*}
Remembering (\ref{estim_lebesgue_const}) and the previous lemma, the error estimate for $\mathcal{U}^{m(i)}$ can be simplified as 
\begin{equation*}
\norm{u - \mathcal{U}^{m(i)}u}_{L^\infty(\Gamma, V)} \lesssim ie^{-\sigma m(i)} \max_{z\in \Sigma(\Gamma, \tau)}\norm{u(z)}_{V}. 
\end{equation*}
This estimate can be applied to the detail operator after a triangle inequality to obtain
\begin{equation}\label{estim_detail}
\norm{\Delta^{m(i)} u}_{L^\infty(\Gamma, V)}  \lesssim i e^{-\sigma m(i-1)} \max_{z\in \Sigma(\Gamma, \tau)}\norm{u(z)}_{V}.
\end{equation}
{Applying (\ref{estim_detail}) to the multidimensional case (by considering one component at a time) leads to the following estimate:}
\begin{lemma} \label{lemma:estim_hierarchical_surplus}
For $\bm{i}\in \N^N$, the hierarchical surplus {of an analytic function $u:\Gamma\rightarrow V$} satisfies
\begin{equation}\label{estim_hierarchical_surplus}
\norm{\Delta^{m(\bm{i})}(u)}_{L^{\infty}(\Gamma, V)} \lesssim
\Lambda_{\bm{i}} e^{-\sigma \vert m(\bm{i}-\bm{1})\vert_1 }.
\end{equation}
{where
\begin{equation*}
\sigma \coloneqq \min_{n\in 1,\ldots, N} \sigma_n, \qquad \sigma_n \coloneqq \log{\left( \frac{2\tau_n}{\vert \Gamma_n\vert} + \sqrt{1+\frac{4\tau_n^2}{\vert\Gamma_n\vert^2}}\right)},
\end{equation*}}
{where the hidden constant depends on $u$.}
\end{lemma}

\subsubsection{A simplified formula for $\zeta_{\textbf{i}, I}$}
In the present section we highlight elementary facts on the zeros of $\Delta^{m(\bm{j})}u$ and the kernel of $\Delta^{m(\bm{j})}$.
These facts are combined to show that the operator $\Delta^{m(\textbf{i})} \left(a \nabla \Delta^{m(\textbf{j})}\right)$ is identically zero unless the multi-indices $\bm{i}$, $\bm{j}\in \N^N$ are ``close to each other" 
(see also~\cite[Proposition 4.3]{guignard2018posteriori} for partial results in this direction). 

We will denote by $\mathcal{R}_{\bm{i}}\subset \N^N$ the axis-aligned rectangle with opposite vertices $\bm{1}$ and $\bm{i}$:
\begin{align}\label{def:rectangle}
\mathcal{R}_{\bm{i}} \coloneqq \left\{\bm{j}\in\N^N : j_n \leq i_n\ \forall n\in 1,\ldots,\N \right\}.
\end{align}

{\begin{theorem} \label{thm:indices_condition_for_zero_error_estimator}
Given, $\bm{i}$, $\bm{j}\in \N^N$, if one of the following two conditions
\begin{align}\label{eq:cond1}
\exists n\in1,\ldots, N : i_n < j_n
\end{align}
or
\begin{align}\label{eq:cond2}
\forall n\in1,\ldots, N : \bm{j} + \bm{e}_n < \bm{i},
\end{align}
is satisfied, then
\begin{align}\label{eq:zero}
\Delta^{m(\bm{i})}\left(a\nabla\Delta^{m(\bm{j})} u \right)\equiv 0\qquad \forall u\in C^0(\Gamma, V).
\end{align}
\end{theorem}}
{
\begin{proof}
Fix $\bm{y}\in \mathcal{Y}^{m(\bm{i})}$. By~\eqref{eq:cond1} and the nestedness of CC nodes, it exists $n\in 1,\ldots,N$ such that $y_n\in\mathcal{Y}^{m(j_n-1)}$.
This implies that $\Delta^{m(\bm{j})}u(\bm{y}) = 0$, as $y_n$ is an interpolation point for both $\mathcal{U}_n^{m(j_n)}$ and $\mathcal{U}_n^{m(j_n-1)}$.
Thus, recalling that $\nabla$ acts on the space variable $x$ only,
\begin{align}\label{eq:zeros_arg_hiersurp}
a(\bm{y}) \nabla \Delta^{m(\bm{j})}u(\bm{y}) = 0\qquad \forall \bm{y}\in \mathcal{Y}^{m(\bm{i})}.
\end{align}
Next, observe that a hierarchical surplus can be written as a linear combination of Lagrange interpolants
\[ \Delta^{m(\bm{i})} = \sum_{\bm{\alpha} \in \left\{0, 1\right\}^N} \left(-1\right)^{\vert\bm{\alpha}\vert}\mathcal{U}^{m(\bm{i}-\bm{\alpha})}.\]
By the nestedness of CC nodes, \eqref{eq:zeros_arg_hiersurp} implies that $a\nabla \Delta^{m(\bm{j})}u$ is in the kernel of each of the interpolants $\mathcal{U}^{m(\bm{i}-\bm{\alpha})}$, $\bm{\alpha}\in\{0,1\}^N\setminus \{0\}$, which in turn implies \eqref{eq:zero}.
To show that \eqref{eq:cond1} also implies \eqref{eq:zero}, first observe that
\begin{align}\label{eq:argument_is_polynomial}
a\nabla \Delta^{m(\bm{j})} u \in \sum_{n=1}^N \mathbb{P}_{m(\bm{j})-\bm{1}+\bm{e}_n} 
= \mathbb{P}_{\left\{\bm{j}\right\}\cup \mathcal{M}_{\left\{\bm{j}\right\}}}
\subseteq \mathcal{P}_{\mathcal{R}_{\bm{i}} \setminus \left\{\bm{i}\right\}},
\end{align}
where the last inclusion is due to assumption \eqref{eq:cond2}.
Next, observe that a hierarchical surplus can be written as a difference of sparse grid interpolants:
$\Delta^{m(\bm{i})} = S_{\mathcal{R}_{\bm{i}}} - S_{\mathcal{R}_{\bm{i}}\setminus \left\{\bm{i}\right\}}$. 
This implies that $\mathbb{P}_{\mathcal{R}_{\bm{i}} \setminus \left\{\bm{i}\right\}}$ is a subset of the kernel of $\Delta^{m(\bm{i})}$, as both $S_{\mathcal{R}_{\bm{i}}}$ and $S_{\mathcal{R}_{\bm{i}}\setminus \left\{\bm{i}\right\}}$ are exact on this space. Together with \eqref{eq:argument_is_polynomial}, this concludes the proof.
\end{proof}}

\begin{remark} \label{rk:simplified_ptwise_error_estim_operator}
The previous theorem can be used to simplify the computation of $\zeta_{\bm{i}, I}$. Consider a multi-index set $I\subset \N^N$ and $\bm{i}\in \mathcal{M}_I$. Define
\[J_{\bm{i}, I} \coloneqq \left\{ \bm{j}\in I : \exists n\in1,\ldots, N : \bm{j} = \bm{i} - \bm{e}_n \right\}.\]
Then, thanks to the previous theorem:
\begin{equation*} 
\Delta^{m(\bm{i})}\left(a\nabla S_I \left[u\right]\right) 
= \Delta^{m(\bm{i})}\left(a\nabla \sum_{\bm{j}\in I} \Delta^{m(\bm{j})} u \right) 
= \Delta^{m(\bm{i})}\left(a\nabla \sum_{\bm{j}\in J_{\bm{i}, I}} \Delta^{m(\bm{j})} u \right),
\end{equation*}
so
\begin{equation} \label{formula:simplified_ptwise_error_estim_operator}
\zeta_{\bm{i}, I} = \norm{\Delta^{m(\bm{i})}\left(a\nabla\sum_{\bm{j}\in J_{\bm{i}, I}} \Delta^{m(\bm{j})} u \right)}_{L^{\infty}(\Gamma, {L^2(D)})}.
\end{equation}
See Figure \ref{fig:simplified_estim} for a graphical representation.
\end{remark}

\begin{figure}[H]
\centering
\begin{tikzpicture}
\coordinate (y) at (0,5);
\coordinate (x) at (4,0);
\draw[axis] (y) -- (0,0) --  (x);
\foreach \x in {1,2,3} {%
    \draw ($(\x,0) + (0,-2pt)$) -- ($(\x,0) + (0,2pt)$)
        node [below] {$\x$};
}
\foreach \y in {1,2,3,4} {%
    \draw ($(0,\y) + (-2pt,0)$) -- ($(0,\y) + (2pt,0)$)
        node [left] {$\y$};
}
\node[] at (4, -0.3) {$j_1$};
\node[] at (-0.3, 5) {$j_2$};

\draw [fill] (1, 1) circle (0.1);
\draw [fill] (1, 2) circle (0.1);
\draw [fill] (1, 3) circle (0.1);
\draw [fill] (2, 1) circle (0.1);
\draw [fill] (2, 2) circle (0.1);
\draw [fill] (2, 3) circle (0.1);
\draw [fill] (1, 4) circle (0.1);
\draw [fill] (2, 4) circle (0.1);
\draw [fill] (3, 1) circle (0.1);
\draw [fill] (3, 1) circle (0.1);
\draw [fill] (3, 2) circle (0.1);
\draw [red, ultra thick] (3, 3) circle (0.1);

\coordinate (A) at (2.5, 2.5);
\coordinate (B) at (2.5, 1.6);
\coordinate (C) at (3.3, 1.6);
\coordinate (D) at (3.3, 2.7);
\coordinate (E) at (2.7, 2.7);
\coordinate (F) at (2.7, 3.3);
\coordinate (G) at (1.6, 3.3);
\coordinate (H) at (1.6, 2.5);
\draw[blue, dashed, ultra thick] (A) -- (B) -- (C) -- (D) -- (E) -- (F) -- (G) -- (H) -- (A);
\node[above right] at (3, 3) {{\color{red} $\bm{i}$}};
\node[] at (3.7, 1.5) {{\color{blue} $J_{\bm{i}, I}$}};
\end{tikzpicture}
\caption{Graphical representation of the simplified computation of $\zeta_{\bm{i}, I}$ from (\ref{formula:simplified_ptwise_error_estim_operator}).
{We consider $N=2$ parameters, each point in the plot corresponds to an element $\bm{j}=(j_1,j_2)\in \N^2$, where on the x-axis we represent $j_1$ and on the y-axis $j_2$.}
Filled dots represent $I$, the red hollow one is $\bm{i}\in\mathcal{M}_I$.
The blue dashed line encircles the {multi-indices} in $J_{\bm{i}, I}$, i.e. the only relevant ones in $I$ for the computation of $\zeta_{\bm{i}, I}${, as explained in Remark \ref{rk:simplified_ptwise_error_estim_operator}}.}
\label{fig:simplified_estim}
\end{figure}

\subsubsection{A priori estimates for estimators and index sets}
\begin{proposition} \label{prop:estim_ptwise_estimator}
Given $u:\Gamma\rightarrow V$ analytic, a multi-index set $I\subset \N^N$ and $\bm{i}\in \mathcal{M}_I$, the {pointwise} error estimator can be bounded as 
\[\zeta_{\bm{i}, I} \lesssim  N \Lambda_{\bm{i}}^2 e^{-\sigma \vert m(\bm{i}-\bm{1})\vert_1},\]
where $\Lambda_{\bm{i}}$ is defined in (\ref{def:bound_stab_const}).
\end{proposition}
\begin{proof}
Observe that $S_I[u]$ is analytic but not uniformly with respect to $I$, so one cannot apply directly the convergence result for the hierarchical surplus.
Recalling Remark~\ref{rk:simplified_ptwise_error_estim_operator}, we can simplify the expression of $\zeta_{\bm{i}, I}$ as 
\begin{align*}
\zeta_{\bm{i}, I} 
& = \norm{\Delta^{m(\bm{i})} \left( a\nabla S_I[u]\right)}_{L^\infty(\Gamma, L^2(D))}\\
& = \norm{\Delta^{m(\bm{i})} \left( a\nabla \sum_{\substack{{n\in 1,\ldots, N}\\{\bm{i}-\bm{e}_n\in I}}}\Delta^{m(\bm{i}-\bm{e}_n)}u\right)}_{L^\infty(\Gamma, L^2(D))}.
\end{align*}
Applying the stability of $\Delta^{m(\bm{i})}$, boundedness of $a$, and the triangle inequality, we obtain
\begin{align*}
\zeta_{\bm{i}, I} 
\lesssim \Lambda_{\bm{i}} \sum_{\substack{{n\in 1,\ldots, N}\\{\bm{i}-\bm{e}_n\in I}}}\norm{\Delta^{m(\bm{i}-\bm{e}_n)}\nabla u}_{L^\infty(\Gamma, L^2(D))}.
\end{align*}
Observe finally that, since $u$ is analytic, we can {apply Lemma \ref{lemma:estim_hierarchical_surplus}} to obtain
\begin{align*}
\zeta_{\bm{i}, I} 
& \leq \Lambda_{\bm{i}} \sum_{\substack{{n\in 1,\ldots, N}\\{\bm{i}-\bm{e}_n\in I}}}  \Lambda_{\bm{i}-\bm{e}_n} e^{-\sigma \vert m(\bm{i}-\bm{e}_n-\bm{1})\vert_1} \lesssim N \Lambda_{\bm{i}}^2 e^{-\sigma \vert m(\bm{i}-\bm{1})\vert_1}.
\end{align*}
\end{proof}
\begin{remark} \label{rk:boundedness_param_estim}
A direct consequence of the previous proposition is the uniform boundedness of the sequence of a-posteriori estimators $\left(\zeta_{SC, I_\ell}\right)_\ell$. Indeed, we have the following bound independently of of the iteration number~$\ell$
\begin{align*}
\zeta_{SC, I_\ell} 
= \sum_{\bm{i}\in\mathcal{M}_{I_\ell}} \zeta_{\bm{i}, I_\ell}
\lesssim N \sum_{\bm{i}\in\mathcal{M}_{I_\ell}} \Lambda_{\bm{i}} e^{-\sigma \vert m(\bm{i-1})\vert_1}
\leq N \sum_{\bm{i}\in\N^N} \Lambda_{\bm{i}} e^{-\sigma \vert m(\bm{i-1})\vert_1} 
< \infty .
\end{align*}
\end{remark}

{\begin{lemma}\label{lemma:bould_lambda_i_l_wrt_l}
The profit maximizer $\bm{i}_\ell\in\N^N$ at iteration $\ell$ of Algorithm~\ref{algo:Refine_param_space} satisfies 
\[\Lambda_{\bm{i}_\ell} = \prod_{n=1}^N \langle\bm{i}_\ell, \bm{e}_n \rangle \leq \left(1+\frac{\ell}{N}\right)^N. \]
Moreover, there holds
\begin{subequations}\label{eq:bound_card_A_i}
\begin{align}
\#A_{\bm{i}_\ell, I_\ell} \leq \left(1+\frac{\ell}{N}\right)^N 
 \end{align}
as well as
\begin{align}
\#\mathcal{M}_{I_\ell} \leq N \left(1 + (\ell-1) \left(1+\frac{\ell-1}{N}\right)^N \right).
\end{align}
\end{subequations}
\end{lemma}
\begin{proof}
First observe that due to the arithmetic-geometric inequality,
\[ \prod_{n=1}^N j_n \leq \left(\frac{\sum_{n=1}^N j_n}{N}\right)^N = \left(\frac{\vert\bm{j}\vert_1}{N}\right)^N \qquad \forall \bm{j}\in \R^N.\]
Then, it can be easily proved by induction that $\vert\bm{i}_\ell\vert_1 = N+\ell$.
To prove~\eqref{eq:bound_card_A_i}, first observe that $A_{\bm{i}} = \mathcal{R}_{\bm{i}}\setminus I$, where $\mathcal{R}_{\bm{i}}$ is the axis-aligned rectangle in $\N^N$ as defined in (\ref{def:rectangle}). 
Thus, $\#A_{\bm{i}_\ell, I_\ell} \leq \#\mathcal{R}_{\bm{i}_\ell, I_\ell} = \Lambda_{\bm{i}_\ell}$ and due to the previous lemma we obtained the desired bound. 
As for the second estimate, first observe that $\#\mathcal{M}_{I_\ell} \leq N \# I_\ell$. Then, an estimate on $\# I_\ell$ comes from the partition $I_\ell = \left\{\bm{1}\right\} \cup \bigcup_{m=1}^{\ell-1} A_{\bm{i}_m}$ and the estimate on $\#A_{\bm{i}_\ell}$.
\end{proof}}

\subsubsection{Remarks on the algorithm driven by workless profit} \label{sec:rks_workless_profit}
In this section, we point out some elementary facts on the behavior of the algorithm when the workless profit defined in (\ref{def:workless_profit}) is used.\\
Inspired by \cite{chkifa2014high}, we give the following definition:
\begin{definition} \label{def:maximal_points_in_margin}
Given a downward closed multi-index set $I\subset \N^N$, $\bm{i}\in \mathcal{M}_I$ is \emph{maximal in $\mathcal{M}_I$} if and only if
\[ \forall \bm{j}\in \mathcal{M}_I\setminus \left\{\bm{i}\right\}, \ \exists n\in 1,\ldots, N : i_n > j_n.\]
The \emph{set of maximal points in $\mathcal{M}_I$} is denoted by $\mu_I$.
\end{definition}
\begin{example}\label{ex:maximal_set_margin_rectangle}
If $\bm{i}\in \N^N$ and $I = \mathcal{R}_{\bm{i}}$ is an axis-aligned rectangle as defined in \eqref{def:rectangle}, then 
\[\mu_I = \left\{ \bm{i} + \bm{e}_n, n\in 1,\ldots, N\right\}.\]
\end{example}
\begin{lemma}
For the workless profit (\ref{def:workless_profit}), the selected point $\bm{i}_\ell$ is maximal in $\mathcal{M}_{I_\ell}$
\begin{align}\label{eq:maximal_mids_rect}
\bm{i}_\ell\in \mu_{I_\ell}.
\end{align}
Therefore, $I_\ell$ is an axis-aligned rectangle in $\N^N$, i.e.
\begin{align} \label{eq:I_l_is_rectangle}
I_\ell = \mathcal{R}_{\bm{i}_{\ell-1}}.
\end{align}
\end{lemma}
\begin{proof}
We prove \eqref{eq:maximal_mids_rect} by contradiction. If $\bm{i}_\ell$ is not maximal, there exists $\bm{j}\in\mathcal{M}_{I_\ell}\setminus\left\{\bm{i}_\ell \right\}$ such that for all $n\in 1,\ldots, N$ $\langle \bm{i}_\ell, \bm{e}_n \rangle \leq j_n$, which implies $\bm{i}_\ell\in \mathcal{R}_{\bm{j}}$. 
Thus, $\bm{i}_\ell\in A_{\bm{j}, I_\ell} = \mathcal{R}_{\bm{j}}\setminus I_\ell$ and by definition of the workless profit, we have the contradiction $\mathcal{P}_{\bm{i}_\ell, I_\ell} < \mathcal{P}_{\bm{j}, I_\ell}$.

The second fact \eqref{eq:I_l_is_rectangle} can be proved by induction. For $\ell=1$, $I_1 = \mathcal{R}_{\bm{1}} = \left\{\bm{1}\right\}$.
Assume that for fixed $\ell\in\N$, $I_\ell = \mathcal{R}_{\bm{i}_{\ell-1}}$.
With~\eqref{eq:maximal_mids_rect} and Example~\ref{ex:maximal_set_margin_rectangle}, we know
\[\bm{i}_\ell\in \mu_{I_\ell} = \mu_{\mathcal{R}_{\bm{i}_{\ell-1}}} = \left\{ \bm{i}_{\ell-1} + \bm{e}_n, n\in 1,\ldots, N\right\}.\]
Thus $I_{\ell+1} = I_{\ell } \cup A_{\bm{i}_\ell, I_\ell} = \mathcal{R}_{\bm{i}_\ell}.$

\end{proof}
To summarize, the use of the workless profit (\ref{def:workless_profit}) implies that, for all $\ell>0$, 
\begin{itemize}
\item it exists \emph{a unique} number $n(\ell)\in 1,\ldots, N$ such that
\begin{equation}\label{i_l+1_at_distance_one_from_i_l}
\bm{i}_{\ell+1} = \bm{i}_\ell + \bm{e}_{n(\ell)}.
\end{equation}
\item as a consequence, the norm of $\bm{i}_\ell$ is given by:
\begin{equation}\label{norm_i_l}
\vert\bm{i}_{\ell+1}\vert_1 = \vert \bm{i}_\ell\vert_1 + 1 = N+\ell. 
\end{equation}
\item $I_\ell$ is a rectangle:
\begin{equation}\label{I_l_is_rectangle}
I_{\ell+1} = \mathcal{R}_{\bm{i}_{\ell }}.
\end{equation}
Therefore, the sparse grid stochastic collocation interpolant is actually a full tensor product Lagrange interpolant:
\[ S_{I_{\ell+1}} = \bigotimes_{n=1}^N \mathcal{U}_n^{m(\langle \bm{i}_\ell, \bm{e}_n\rangle)}.\]
\item the multi-indices added at iteration $\ell$ are
\begin{equation}\label{added_points_simpl_profit}
A_{i_\ell, I_\ell} = I_{\ell+1} \setminus I_\ell = \left\{\bm{j}\in \mathcal{R}_{\bm{i}_\ell} : j_{n(\ell)} =  \langle \bm{i}_\ell, \bm{e}_{n(\ell)} \rangle\right\}.
\end{equation}
\end{itemize}
In other words, the evolution of the approximation space is determined by the sequence of integers $\left(n(\ell)\right)_{\ell }$.
This allows us to simplify the notation as follows
\begin{align*}
A_{n, I_\ell} \coloneqq A_{\bm{i}_{\ell-1} + \bm{e}_n, I_\ell}\\
\mathcal{P}_{n, I_\ell} \coloneqq \sum_{\bm{j}\in A_{n, I_\ell}} \zeta_{\bm{j}, I_\ell}
\end{align*}
Let us moreover denote the maximal $n$-th dimension of $I_\ell$ as
\begin{equation}\label{def:rectangle_width}
r_{n,\ell} \coloneqq \max_{\bm{j}\in I_\ell} j_n.
\end{equation} 
See Figure \ref{fig:simplified_profit} for a graphical representation.
\begin{figure}
\centering
\begin{tikzpicture}
\coordinate (y) at (0,5);
\coordinate (x) at (4,0);
\draw[axis] (y) -- (0,0) --  (x);
\foreach \x in {1,2,3} {%
    \draw ($(\x,0) + (0,-2pt)$) -- ($(\x,0) + (0,2pt)$)
        node [below] {$\x$};
}
\foreach \y in {1,2,3,4} {%
    \draw ($(0,\y) + (-2pt,0)$) -- ($(0,\y) + (2pt,0)$)
        node [left] {$\y$};
}
\node[] at (4, -0.3) {$j_1$};
\node[] at (-0.3, 5) {$j_2$};
\draw [fill] (1, 1) circle (0.1);
\draw [fill] (1, 2) circle (0.1);
\draw [fill] (1, 3) circle (0.1);
\draw [fill] (2, 1) circle (0.1);
\draw [fill] (2, 2) circle (0.1);
\draw [fill] (2, 3) circle (0.1);
\draw [ultra thick] (3, 1) circle (0.1);
\draw [ultra thick] (3, 2) circle (0.1);
\draw [red, ultra thick] (3, 3) circle (0.1);
\draw [ultra thick] (1, 4) circle (0.1);
\draw [ultra thick] (2, 4) circle (0.1);
\draw[draw=blue] (2.8,0.6) rectangle ++(0.7, 2.7);
\node[below right] at (2, 3) {$\bm{i}_{\ell-1}$};
\node[below right] at (3, 3) {{\color{red} $\bm{i}_{\ell }$}};
\node[] at (4.0, 0.5) {{\color{blue} $A_{1, I_\ell}$}};
\end{tikzpicture}
\caption{Example of approximation parameters at a generic step $\ell$ of the algorithm when the workless profit (\ref{def:workless_profit}) is used. 
{We consider $N=2$ parameters, each point in the plot corresponds to an element $\bm{j}=(j_1,j_2)\in \N^2$, where on the x-axis we represent $j_1$ and on the y-axis $j_2$.}
Filled dots represent $I_\ell$, hollow ones $\mathcal{M}_{I_\ell}$. 
The multi-index selected by the algorithm at current step, $\bm{i}_\ell$, is in red (so in this case $n(\ell)=1$).
The blue rectangle encircles multi-indices in $A_{n(\ell), I_\ell}$.}
\label{fig:simplified_profit}
\end{figure}

The estimate for the pointwise error estimator from Proposition \ref{prop:estim_ptwise_estimator} can be improved as follows.
First observe that, due to (\ref{i_l+1_at_distance_one_from_i_l}) and (\ref{I_l_is_rectangle}), 
\begin{equation*}
J_{\bm{i}, I_\ell} 
{=} \left\{ \bm{j}\in I_\ell : \exists n\in 1,\ldots, N\ {:}\ \bm{j} = \bm{i}-\bm{e}_n\right\}
= \left\{ \bm{i} - \bm{e}_{n(\ell)}\right\}.
\end{equation*}

Thus, $\# J_{\bm{i}, I_\ell} = 1$ and we may reduce the dependence on $N$ by
\begin{equation}\label{estim_zeta_i_workless_profit}
\zeta_{\bm{i}, I_\ell} 
\lesssim \Lambda_{\bm{i}}^2 e^{-\sigma\vert m(\bm{i}-\bm{1})\vert}. 
\end{equation}

\subsection{Convergence of the parametric estimator}
In the following two lemmata, we prove that Algorithm~\ref{algo:SC_adaptive_algo} driven by workless profit and profit with work respectively forces the maximum profit over the margin  to zero.
\begin{proposition} \label{prop:limit_profit_without_work}
If the workless profit (\ref{def:workless_profit}) is used, then
\[\lim_{\ell \rightarrow\infty} \mathcal{P}_{n(\ell), I_\ell} = 0.\]
\end{proposition}
\begin{proof}
For fixed $n\in 1,\ldots, N$, we estimate each pointwise error estimator appearing in $\mathcal{P}_{n, I_\ell}$ by (\ref{estim_zeta_i_workless_profit}) and the fact that for any $\bm{i}$ in $A_{n, I_\ell}$, $i_n = r_{n,\ell}+1$.
\begin{align*}
\mathcal{P}_{n, I_\ell} 
& = \sum_{\bm{j}\in A_{n, I_\ell}} \zeta_{\bm{j}, I_\ell}
  \lesssim  \sum_{\bm{i}\in A_{n, I_\ell}} \Lambda_{\bm{i}}^2 e^{-\sigma \vert m(\bm{i-1})\vert_1}\\
& =   \sum_{\bm{i}\in A_{n, I_\ell}} \prod_{k=1}^N \left(i_k^2 e^{-\sigma \vert m(i_k-1)\vert}\right)\\
& \leq  \left(r_{n,\ell}+1\right)^2 e^{-\frac{\sigma}{2} m(r_{n,\ell})} 
\sum_{\bm{i}\in A_{n, I_\ell}} \left(i_n^2 e^{-\frac{\sigma}{2} m(i_n+1)} \prod_{k=1, k\neq n}^N \left(i_k^2 e^{-\sigma m(i_k-1)}\right) \right)\\
& \leq  \left(r_{n,\ell}+1\right)^2 e^{-\frac{\sigma}{2} m(r_{n,\ell})} 
\sum_{\bm{i}\in A_{n, I_\ell}} \Lambda_{\bm{i}}^2 e^{-\frac{\sigma}{2} \vert m(\bm{i-1})\vert_1}.
\end{align*}
The last factor is uniformly bounded with respect to $\ell$ (but this bound depends on the number of dimensions $N$)
\[\sum_{\bm{i}\in A_{n, I_\ell}} \Lambda_{\bm{i}}^2 e^{-\frac{\sigma}{2} \vert m(\bm{i-1})\vert_1} 
\leq \sum_{\bm{i}\in \N^N} \Lambda_{\bm{i}}^2 e^{-\frac{\sigma}{2} \vert m(\bm{i-1})\vert_1} < \infty.
\]
We are left with:
\[\mathcal{P}_{n, I_\ell}  \lesssim \left(r_{n,\ell}+1\right)^2 e^{-\frac{\sigma}{2} m(r_{n,\ell})}.\]
The proof is completed by observing that $\lim_{\ell \rightarrow \infty} r_{n(\ell), l} = \infty.$
\end{proof}

For the profit with work, we can even show convergence to zero of the profit without using the analyticity assumption on $u$. This is not relevant for the problem at hand, as the analyticity follows immediately, but may be relevant for more complicated and less regular random coefficients.
\begin{proposition}\label{prop:limit_profit_with_work}
There holds
$
\lim_{\ell \rightarrow \infty} \mathcal{P}_{\bm{i}_\ell, I_\ell} = 0
$.
\end{proposition}
\begin{proof}
As in the proof of Proposition~\ref{prop:estim_ptwise_estimator}, but without using any analyticity of $u$, we obtain with~\eqref{stab_hierarchical_surplus} that
\begin{align*}
 \zeta_{\bm{i},I}\lesssim \Lambda_{\bm{i}}^2N 
 \norm{\nabla u}_{L^\infty(\Gamma,L^2(D))}.
\end{align*}
We observe that the doubling rule (\ref{doubling_rule}) implies
\begin{equation} \label{bounds_work_i}
2^{\vert\bm{i}\vert_1 - 2N} \leq W_{\bm{i}} \leq 2^{\vert\bm{i}\vert_1 -N}.
\end{equation}
Thus, the profit can be estimated as:
\begin{align*}
\mathcal{P}_{\bm{i}_\ell, I_\ell} 
\lesssim \frac{\sum_{\bm{j}\in A_{\bm{i}_\ell, I_\ell}}\zeta_{\bm{{j}},I_\ell}}{\sum_{\bm{j}\in A_{\bm{i}_\ell, I_\ell}} W_{\bm{j}}} 
\lesssim \frac{\# A_{\bm{i}_\ell,I_\ell} \Lambda_{\bm{i}_\ell}^2 N}{W_{\bm{i}_\ell}}
\leq N(1+\ell/N)^N \Lambda_{\bm{i}_\ell}^22^{2N-\vert \bm{i}_\ell\vert_1}.
\end{align*}
Since $2^{|\bm{i}_\ell|_1}$ grows much faster than $\Lambda_{\bm{i}_\ell}^2=\prod_{n=1}^N i_{\ell,n}^2$, we conclude the proof.
\end{proof}

The following result shows that, if a multi-index $\bm{i}\in \N^N$ stays in the margin indefinitely, then it's pointwise estimator vanishes. This result is valid for both workless profit and profit with work.
\begin{proposition} \label{prop:multi_index_in_margin_has_0_ptwise_estim}
Let $\hat{\bm{i}}\in \N^N$ and suppose the index remains in the margin indefinitely, i.e.,
\begin{align*}
\exists \ell_0 \in \N : \forall \ell\geq \ell_0,\ \hat{\bm{i}}\in \mathcal{M}_{I_{\ell }}.
\end{align*} 
Then, the pointwise error estimator corresponding to $\hat{\bm{i}}$ vanishes
\begin{align*}
\lim_{\ell \rightarrow \infty} \zeta_{\hat{\bm{i}}, I_\ell} = 0.
\end{align*}
\end{proposition}
\begin{proof}
Let $\hat{\bm{i}}\in \N^N$ such that $\hat{\bm{i}}\in \mathcal{M}_{I_\ell}$ for all $\ell>\ell_0$. Thus, $\hat{\bm{i}}\neq \bm{i}_\ell$ for any $\ell > \ell_0$, which means that
\begin{equation*}
 \mathcal{P}_{\hat{\bm{i}}, I_\ell} \leq \mathcal{P}_{\bm{i}_\ell, I_\ell}\qquad \forall \ell>\ell_0.
\end{equation*}
In case the profit with work (\ref{def:profit_work}) is used, since $\lim_{\ell \rightarrow \infty} \mathcal{P}_{\bm{i}_\ell, I_\ell} = 0$ as proved in Proposition \ref{prop:limit_profit_with_work}, we have that $\lim_{\ell \rightarrow \infty} \mathcal{P}_{\hat{\bm{i}}, I_\ell} = 0$ (otherwise $\widehat{\bm{i}}$ would be selected at some point).
Moreover, since $\sum_{\bm{j}\in A_{\hat{\bm{i}}, I_\ell}} W_{\bm{j}}$ (i.e. the denominator in the profit $\mathcal{P}_{\hat{\bm{i}}, I_\ell}$) is eventually constant with respect to $\ell$, we have that $\lim_{\ell \rightarrow \infty} \sum_{\bm{j}\in A_{\hat{\bm{i}}, I_\ell}} \zeta_{\hat{\bm{i}}, I_\ell} = 0$, and in particular we obtain the statement.
The same holds if the profit without work (\ref{def:workless_profit}) is employed, as in Proposition \ref{prop:limit_profit_without_work} we have proved that also in this case $\lim_{\ell \rightarrow \infty} \mathcal{P}_{\bm{i}_\ell, I_\ell} = 0$.
\end{proof}

\begin{remark}
Recall the simplified formula~\eqref{formula:simplified_ptwise_error_estim_operator} for $\zeta_{\hat{\bm{i}}, I_\ell} $ 
with $J_{\hat{\bm{i}}, I_\ell} \coloneqq \left\{ \hat{\bm{i}} - \bm{e}_n \ : n\in 1,\ldots ,N \right\}$.
Observe that $\left(J_{\hat{\bm{i}}, I_\ell}\right)_{\ell }$ is eventually constant, i.e. it exists $\ell_2 > \ell_0$ {(as defined in the previous proposition)} such that for all $\ell>\ell_2$ $J_{\hat{\bm{i}}, I_\ell} = J_{\hat{\bm{i}}, I_{\ell _2}}$. 
Thus, $\left( \zeta_{\hat{\bm{i}}, I_\ell} \right)_\ell$ is also eventually constant.
Therefore, $\left(\zeta_{\hat{\bm{i}}, I_\ell}\right)_\ell$ does not only vanish in the limit, but is actually eventually zero:
\begin{align*}
\forall \ell>\ell_2, \zeta_{\hat{\bm{i}}, I_\ell} = 0.
\end{align*}
\end{remark}

We can finally prove the convergence of the parameter-enrichment algorithm with a technique inspired by \cite[Proposition 10]{bespalov2019convergence}.
\begin{theorem}[Convergence of the parameter-enrichment algorithm]\label{th:convergence_param_only_algo}
The adaptive stochastic collocation Algorithm \ref{algo:SC_adaptive_algo} driven by either workless profit or profit with work, leads to a vanishing sequence of a-posteriori error estimators, thus also leading to a convergent sequence of discrete solutions
\[\lim_{\ell \rightarrow \infty} \zeta_{SC,I_ l} = 0 = \lim_{\ell \rightarrow \infty} \norm{u-S_{I_\ell}[u]}_{L^{\infty}(\Gamma, V)}\]
\end{theorem}
\begin{proof}
The a-posteriori error estimator at step $\ell\in\N$ can be written as 
\[\zeta_{SC, I_\ell} = \sum_{\bm{i}\in\N^N} \zeta_{\bm{i}, I_\ell} \mathbbm{1}_{\mathcal{M}_{I_\ell}}(\bm{i}),\]
where $\mathbbm{1}_{\mathcal{M}_{I_\ell}}$ is the indicator function of the margin $\mathcal{M}_{I_\ell}$.
In order to prove that the sequence vanishes by dominated convergence, it is sufficient to prove that
$(i)$ for any $\bm{i}\in\N^N$, $\lim_{\ell \rightarrow\infty} \zeta_{\bm{i}, I_\ell} \mathbbm{1}_{\mathcal{M}_{I_\ell}}=0$
and $(ii)$ that the sequence $\left( \zeta_{SC, I_\ell} \right)_\ell$ is bounded.
The uniform boundedness $(ii)$ was proved in Remark \ref{rk:boundedness_param_estim}.
As for $(i)$, observe that at least one of the following cases applies:
\begin{itemize}
\item  $\bm{i}$ is eventually added to $I_\ell$, thus $\mathbbm{1}_{\mathcal{M}_{I_\ell}} (\bm{i})$ is eventually zero;
\item $\bm{i}$ is never added to the margin (for all $\ell\in\N$, $\bm{i}\in \N^N \setminus \mathcal{M}_{I_\ell}$), thus $\zeta_{\bm{i}, I_\ell}$ is constantly zero;
\item it exists $\bar{\ell}\in\N$ such that for any $\ell\geq \bar{\ell}$, $\bm{i}\in\mathcal{M}_{I_\ell}$. In this case, due to Proposition \ref{prop:multi_index_in_margin_has_0_ptwise_estim}, $\lim_{\ell \rightarrow \infty} \zeta_{\bm{i}, I_\ell} = 0.$
\end{itemize}
This concludes the proof.
\end{proof}
\subsection{Convergence of the parametric error}\label{sec:errconv}
{In the present section we denote by $\mathcal{L}(L^{\infty}(\Gamma, V))$ the space of linear bounded operators $T: L^{\infty}(\Gamma, V)\rightarrow L^{\infty}(\Gamma, V)$. It is well known that this is a Banach space when equipped with the usual operator norm 
\begin{align*}
\norm{T}_{\mathcal{L}(L^{\infty}(\Gamma, V))} \coloneqq \sup_{u\in L^{\infty}(\Gamma, V), u\neq 0} \frac{\norm{Tu}_{L^{\infty}(\Gamma, V)}}{\norm{u}_{L^{\infty}(\Gamma, V)}}.
\end{align*}
}
We have the following monotonicity property of the approximation error of $S_I[\cdot]$ with respect to $I$:
\begin{lemma}\label{lemma:monotonicity_error}
Let $u\in C^0(\Gamma, V)$ and $I, J\subset \N^N$ downward-closed multi-index sets such that $J\subset I$. Then
\begin{equation*}
\norm{u - S_{I}[u]}_{L^{\infty}(\Gamma, V)}  
\leq \left( 1 + \norm{S_I}_{\mathcal{L}(L^{\infty}(\Gamma, V))}\right)
\norm{u - S_{J}[u]}_{L^{\infty}(\Gamma, V)}.
\end{equation*}
\end{lemma}
\begin{proof}
With the identity operator $\bm{1}$ on $C^0(\Gamma, V)$, observe that
\begin{equation*}
u - S_{I}[u] = \left( \bm{1} - S_{I}\right) u = \left( \bm{1} - S_{I}\right) \left( \bm{1} - S_{J}\right) u
\end{equation*}
since $J\subset I$ implies $S_{I} \left[S_{J} [u]\right] = S_{J} [u]$. The triangle inequality concludes the proof.
\end{proof}

In the present section we provide error estimates for $S_{I_\ell}$ with respect to the number of iterations $\ell$. We consider both the possible definitions of profit (\ref{def:workless_profit}) and (\ref{def:profit_work}).

\begin{remark} \label{rk:bound_oeprator_norm_S_I}
The quantity $\norm{S_{I_\ell}}_{\mathcal{L}(L^\infty(\Gamma, V))}$ from Lemma~\ref{lemma:monotonicity_error} satisfies
\begin{itemize}
\item Workless profit: 
${I_\ell} = \mathcal{R}_{\bm{i}_{\ell-1}}$, i.e. $S_{I_\ell}$ is actually a tensor-product Lagrange interpolant (see Section \ref{sec:rks_workless_profit}). Therefore, we can estimate
\begin{equation} \label{estim_norm_S_I_workless_profit}
\norm{S_{I_\ell}}_{\mathcal{L}(L^\infty(\Gamma, V))}
= \norm{\bigotimes_{n=1}^N \mathcal{U}_n^{m(\langle \bm{i}_{\ell-1}, \bm{e}_n\rangle )}}_{\mathcal{L}(L^\infty(\Gamma, V))} 
\leq \prod_{n=1}^N \langle \bm{i}_{\ell-1}, \bm{e}_n\rangle
\leq \left(1+\frac{\ell-1}{N}\right)^N{,}
\end{equation}
where in the first inequality we used the stability bound for the Lagrange interpolant (\ref{stab_explicit_i}) and Lemma \ref{lemma:bould_lambda_i_l_wrt_l} for the second inequality.
\item Profit with work: 
Partitioning $I_\ell$ with the sequence $\left( A_{\bm{i}_m, I_m}\right)_{m=1}^{\ell-1}$ and using Lemma~\ref{lemma:bould_lambda_i_l_wrt_l}.
\begin{align} \label{estim_norm_S_I_work_profit}
\begin{split} 
\norm{S_{I_\ell}}_{\mathcal{L}(L^\infty(\Gamma, V))} 
& \leq \sum_{\bm{i}\in I_\ell} \norm{\Delta^{m(\bm{i})}}_{\mathcal{L}(L^\infty(\Gamma, V))} 
\leq \sum_{m=1}^{\ell-1} \# A_{\bm{i}_m, I_m} \Lambda_{\bm{i}_m}\\
& \leq (\ell-1) \left(1+\frac{\ell-1}{N}\right)^{2N}{.}
\end{split}
\end{align}
\end{itemize}
\end{remark}

We finally prove the parametric error estimates, first with workless profit, then with profit with work.
\begin{theorem} \label{th:rate_conv_workless_profit}
Consider Algorithm \ref{algo:SC_adaptive_algo} with workless profit defined in (\ref{def:workless_profit}). 
Denote by $I_\ell$ the downward-closed multi-index sets chosen by the algorithm at step $\ell>0$ and
by $S_{I_\ell}[u]$ the corresponding sparse grid stochastic collocation approximation of the analytic function $u:\Gamma\rightarrow V$.
Then,
\begin{equation}\label{estim_error_SGSC_interpol_with_rate}
\norm{u - S_{I_\ell}[u]}_{L^\infty(\Gamma, V)}  \lesssim
\left( 1+ \left(1+\frac{\ell-1}{N}\right)^N\right)
N
 \ell^2 e^{-\frac{\sigma}{2} m(1+\frac{\ell}{N})}\qquad \forall \ell>0.
\end{equation}
\end{theorem}
\begin{proof}
Fix $\ell>0$.
Recall the definition of $r_{n,\ell}$ from~\eqref{def:rectangle_width} and
consider the direction 
$\bar{n}\in\{1,\ldots,N\}$ which maximizes $r_{n,\ell}$.
With $n(\ell)$ from~\eqref{i_l+1_at_distance_one_from_i_l}, define
\[\ell' \coloneqq \max \left\{\ell'\in 1,\ldots, \ell : n(\ell') = \bar{n}\right\}\]
 and observe that with each iteration, at least one side of the axis aligned rectangle $I_\ell$ is increased by one, i.e.,
\begin{equation} \label{bound_below_r}
r_{n(\ell'), \ell'} = r_{\bar{n}, {\ell}} \geq 1 + \frac{\ell}{N}.
\end{equation} 
Applying estimate (\ref{estim_norm_S_I_workless_profit}) form the previous remark, we can bound 
\begin{equation*}
\norm{u - S_{I_\ell}[u]}_{L^\infty(\Gamma, V)}  \leq \left( 1 + \left(1+\frac{\ell-1}{N}\right)^N\right) \norm{u - S_{I_{\ell '}}[u]}_{L^\infty(\Gamma, V)}.
\end{equation*}
Now, apply the reliability of the error estimator proved in \cite[Proposition 4.3]{guignard2018posteriori} to obtain
\begin{equation*}
\norm{u - S_{I_{\ell '}}[u]}_{L^\infty(\Gamma, V)}  \lesssim \sum_{\bm{i}\in \mathcal{M}_{I_{\ell '}}} \zeta_{\bm{i}, I_{\ell '}}.
\end{equation*}
Recalling the definition of $A_{n, {I_{\ell '}}}$ and $\mathcal{P}_{n, I_{\ell '}}$ for $n\in 1,\ldots, N$ given in Section \ref{sec:rks_workless_profit}, we have
\begin{equation*}
\sum_{\bm{i}\in \mathcal{M}_{I_{\ell '}}} \zeta_{\bm{i}, I_{\ell '}} = 
\sum_{n=1}^N \sum_{\bm{i}\in A_{n, I_{\ell '}}} \zeta_{\bm{i}, I_{\ell '}} = 
\sum_{n=1}^N \mathcal{P}_{n, I_{\ell '}}
\leq N \mathcal{P}_{n(\ell'), I_{\ell '}}.
\end{equation*}
The profit $\mathcal{P}_{n(\ell'), I_{\ell '}}$ can now be bounded as a function of $r_{n(\ell'), \ell'}$ as we did in Proposition \ref{prop:limit_profit_without_work}
\begin{align*}
\mathcal{P}_{n(\ell'), I_{\ell '}} 
=  \sum_{\bm{j}\in A_{n(\ell'), I_{\ell '}}} \zeta_{\bm{j}, I_{\ell '}} 
\leq  \sum_{\bm{j}\in A_{n(\ell'), I_{\ell '}}} 
\left(\prod_{{k}=1}^N j_k \right)^2 e^{-\frac{\sigma}{2} \vert m(\bm{j} -\bm{1})\vert} 
\lesssim   r_{n(\ell'), \ell'}^2 e^{-\frac{\sigma}{2} m(r_{n(\ell'), \ell'})},
\end{align*}
where in the first inequality we have applied the estimate (\ref{estim_zeta_i_workless_profit}) on $\zeta_{\bm{j}, I_{\ell '}}$
and in the second we have exploited the fact that, for $\bm{j}\in A_{n(\ell'), I_{\ell '}}$, $j_{n(\ell')} = r_{n(\ell'), \ell'} + 1$.
Recalling that $1+\frac{\ell}{N} \leq r_{n(\ell'), \ell'} \leq \ell{+1}$, we obtain
\[\mathcal{P}_{n(\ell'), I_{\ell '}} \lesssim \ell^2 e^{-\frac{\sigma}{2} m(1+\frac{\ell}{N})}. \]
\end{proof}

Let us now prove the analogous result for the algorithm driven by profit with work.
\begin{theorem} \label{th:rate_conv_profit_work}
Consider Algorithm \ref{algo:SC_adaptive_algo} with profit with work defined in (\ref{def:profit_work}). 
Denote by $I_\ell$ the downward-closed multi-index sets chosen by the algorithm at step $\ell>0$ and
by $S_{I_\ell}[u]$ the corresponding sparse grid stochastic collocation approximation of the analytic function $u:\Gamma\rightarrow V$.
Then,
\begin{equation}\label{estim_error_SGSC_interpol_with_rate_profit}
\norm{u - S_{I_\ell}[u]}_{L^\infty(\Gamma, V)}  \lesssim
\ell^5
\left(\frac{\ell}{N}\right)^{4N}
2^{\ell\left(1-\frac{1}{N}\right)}
e^{-\frac{\sigma}{2} m\left(\ell^{\frac{1}{N}}\right)} 
\qquad \forall \ell>0.
\end{equation}
\end{theorem}
\begin{proof}
For brevity, we write $\zeta_{\bm{i}}$, $A_{\bm{i}}$ and $\mathcal{P}_{\bm{i}}$ instead of $\zeta_{\bm{i}, I}$, $A_{\bm{i}, I}$ and $\mathcal{P}_{\bm{i}, I}$ respectively.
Fix $\ell>0$ and consider $\bar{r} \coloneqq \max_{\bm{i}\in I_\ell} \vert \bm{i}\vert_{\ell ^\infty}$ 
and $\bar{n}\in 1,\ldots, N$ such that, for some $\bm{i}\in I_\ell$, $i_{\bar{n}} = \bar{r}$.
Observe that $\# I_\ell\gtrsim \ell$ and hence 
\[\bar{r} \geq \ell^{\frac{1}{N}}.\]
Consider {the last step $\ell'$ in which $I_\ell$ has been extended in direction $\bar{n}$, i.e.,}
\begin{equation} \label{def:l'_maximun_component}
\ell' \coloneqq \max \left\{\ell' \in 1,\ldots, \ell : \langle \bm{i}_{\ell '}, \bm{e}_{\bar{n}}\rangle = \bar{r} \textrm{ and } \bm{i}_{\ell '} - \bm{e}_{\bar{n}}\in I_{\ell '}\right\}.
\end{equation}
Applying estimate (\ref{estim_norm_S_I_work_profit}) from Remark~\ref{rk:bound_oeprator_norm_S_I}, we can bound 
\begin{equation} \label{bound_error_step_l_with_erro_step_l'}
\norm{u - S_{I_\ell}[u]}_{L^\infty(\Gamma, V)}  
\leq \left( 1 + (\ell-1) \left(1+\frac{\ell-1}{N}\right)^{2N}\right)
\norm{u - S_{I_{\ell '}}[u]}_{L^\infty(\Gamma, V)}.
\end{equation}
In \cite[Proposition 4.3]{guignard2018posteriori}, the reliability of the error estimator is proved
\begin{equation*}
\norm{u - S_{I_{\ell '}}[u]}_{L^\infty(\Gamma, V)}  \lesssim \sum_{\bm{i}\in \mathcal{M}_{I_{\ell '}}} \zeta_{\bm{i}}.
\end{equation*}
Recalling the definition of $\mu_{I_{\ell '}}$, the set of maximal elements in $\mathcal{M}_{I_{\ell '}}$ (Definition \ref{def:maximal_points_in_margin}), the margin can be represented (but in general not partitioned) as
\begin{align}\label{eq:partition}
\mathcal{M}_{I_{\ell '}} = \bigcup_{\bm{j}\in \mu_{I_{\ell '}}} A_{\bm{j}}.
\end{align}
Thus, we can estimate
\begin{align*}
\sum_{\bm{i}\in \mathcal{M}_{I_{\ell '}}} \zeta_{\bm{i}} 
\leq & \sum_{\bm{j}\in \mu_{I_{\ell '}}} \sum_{\bm{i}\in A_{\bm{j}}} \zeta_{\bm{i}} 
=  \sum_{\bm{j}\in \mu_{I_{\ell '}}} \frac{\sum_{\bm{i}\in A_{\bm{j}}} \zeta_{\bm{i}}}{\sum_{\bm{i}\in A_{\bm{j}}} W_{\bm{i}}} \sum_{\bm{i}\in A_{\bm{j}}} W_{\bm{i}} 
=  \sum_{\bm{j}\in \mu_{I_{\ell '}}} \mathcal{P}_{\bm{j}} \sum_{\bm{i}\in A_{\bm{j}}} W_{\bm{i}} \\
\leq & \mathcal{P}_{\bm{i}_{\ell '}} \sum_{\bm{j}\in \mu_{I_{\ell '}}} \sum_{\bm{i}\in A_{\bm{j}}} W_{\bm{i}} 
=  \left( \sum_{\bm{i}\in A_{\bm{i}_{\ell '}}} \zeta_{\bm{i}} \right)
\frac{1}{\sum_{\bm{i}\in A_{\bm{i}_{\ell '}}} W_{\bm{i}}} 
\left(\sum_{\bm{j}\in \mu_{I_{\ell '}}} \sum_{\bm{i}\in A_{\bm{j}}} W_{\bm{i}} \right),
\end{align*}
where in the second inequality we have used the fact that $\mathcal{P}_{\bm{i}_{\ell '}} \geq \mathcal{P}_{\bm{j}}$ for any $\bm{j}\in \mathcal{M}_{I_{\ell '}}$.
Let us now estimate each of the three factors separately.
\begin{itemize}
\item $\sum_{\bm{i}\in A_{\bm{i}_{\ell '}}} \zeta_{\bm{i}}$: As in the proof of Theorem~\ref{th:rate_conv_workless_profit} (using the estimate from Proposition \ref{prop:estim_ptwise_estimator} instead of the one in~\eqref{estim_zeta_i_workless_profit}) we obtain with $\ell^{\frac{1}{N}} \leq \bar{r} \leq {\ell+1}$ that 
\begin{equation}\label{estim_sum_ptwise_err_estim}
\sum_{\bm{i}\in A_{\bm{i}_{\ell '}}} \zeta_{\bm{i}} \lesssim N \ell^2 e^{-\frac{\sigma}{2} m\left(\ell^{\frac{1}{N}}\right)}.
\end{equation}
\item $\sum_{\bm{i}\in A_{\bm{i}_{\ell '}}} W_{\bm{i}}$: There holds
\begin{equation}\label{estim_reciprocal_work_i_l}
\sum_{\bm{i}\in A_{\bm{i}_{\ell '}}} W_{\bm{i}} 
\geq W_{\bm{i}_{\ell '}} 
\geq m(\langle \bm{i}_{\ell '}, \bm{e}_{\bar{n}}\rangle) - m(\langle \bm{i}_{\ell '}, \bm{e}_{\bar{n}}\rangle-1)
\geq 2^{\bar{r}-2} \geq 2^{\frac{\ell}{N}-2}
\end{equation}
\item $\sum_{\bm{j}\in \mu_{I_{\ell '}}} \sum_{\bm{i}\in A_{\bm{j}}} W_{\bm{i}}$: We observe
\begin{equation*}
\sum_{\bm{j}\in \mu_{I_{\ell '}}} \sum_{\bm{i}\in A_{\bm{j}}} W_{\bm{i}} = 
\sum_{\bm{i}\in \mathcal{M}_{I_{\ell '}}} \#\left\{\bm{j}\in \mu_{I_{\ell '}} : \bm{i}\in A_{\bm{j}} \right\} W_{\bm{i}}.
\end{equation*}
Thus, being  $\# \left\{\bm{j}\in \mu_{I_{\ell '}} : \bm{i}\in A_{\bm{j}} \right\} \leq \# \mathcal{M}_{I_{\ell '}} $, we can estimate
\begin{equation} \label{estim_double_sum_works}
\sum_{\bm{j}\in \mu_{I_{\ell '}}} \sum_{\bm{i}\in A_{\bm{j}}} W_{\bm{i}} 
\leq \#\mathcal{M}_{I_{\ell '}} \sum_{\bm{i}\in \mathcal{M}_{I_{\ell '}}} W_{\bm{i}} 
\leq \left( \#\mathcal{M}_{I_{\ell '}} \right)^2 \max_{\bm{i}\in \mathcal{M}_{I_{\ell '}}} W_{\bm{i}}.
\end{equation}
An estimate for $\#\mathcal{M}_{I_{\ell '}}$ is given in~\eqref{eq:bound_card_A_i}. 
For the second factor, use the bound on $W_{\bm{i}}$ from (\ref{bounds_work_i}) and the fact that for any $\bm{i}\in \mathcal{M}_{I_\ell}, \vert\bm{i}\vert_1 \leq N+\ell$ to obtain:
\begin{equation} \label{estim_double_sum_works_final}
\sum_{\bm{j}\in \mu_{I_{\ell '}}} \sum_{\bm{i}\in A_{\bm{j}}} W_{\bm{i}} 
\leq \left( N + N(\ell-1)\left(1+\frac{\ell-1}{N}\right)^N \right)^2 2^{\ell}.
\end{equation}
\end{itemize}
Finally, the statement of the theorem is obtained combining (\ref{estim_sum_ptwise_err_estim}), (\ref{estim_reciprocal_work_i_l}) and (\ref{estim_double_sum_works_final}).
\end{proof}
{\begin{remark}
           We note that the convergence rates in Theorems~\ref{th:rate_conv_workless_profit}--\ref{th:rate_conv_profit_work} above compare the error to the number of adaptive steps $\ell$. This is hard to compare to classical a~priori results which bound the error in terms of the number of collocation points (see, e.g., ~\cite{scc5,babuvska2007stochastic}). Due to the adaptive nature of the algorithm we have no knowledge about the shape of $I_\ell$ and hence the number of collocation points $\#\HH_{I_\ell}$. 
           Additionally, we do not assume any a~priori information about the anisotropy of the solution. Hence, the term $\ell^{1/N}$ is the \emph{worst-case} for a fully isotropic solution. We point out that the observed rate of convergence is much better (see Section~\ref{sec:numerics}) and further research is required to explain the performance of the adaptive algorithm.
          \end{remark}}
\section{Convergence of the fully discrete algorithm} \label{sec:conv_fully_discr_algo}
In order to prove the convergence of Algorithm \ref{algo:SCFE}, it is sufficient to prove that 
\begin{itemize}
 \item 
in Algorithm~\ref{algo:refine_FE} (the finite element refinement sub-routine) the finite element error eventually falls below the tolerance prescribed in Line \ref{algo_FE_ref:line:tolerance} { (Alg.~\ref{algo:refine_FE})} and iteratively updated in Line \ref{algo_FE_ref:line:tolerance_inloop} { (Alg.~\ref{algo:refine_FE})} (proved in Section \ref{sec:h_ref_convergence})
\item that  the parametric estimator $\zeta_{SC, I_\ell}$ in Algorithm \ref{algo:SCFE} vanishes (proved in Section \ref{sec:vanishing_param_estim_SCFE}).
\end{itemize}
Indeed, if this is the case, $\eta_{ {\rm FE}, I_\ell}$ will vanish with $\zeta_{SC, I_\ell}$ because of the definition of the finite element refinement tolerance 
and the reliability of the estimator will ensure the convergence of the discrete solution to the analytic one.

In the present section, we will write $\zeta_{SC, I} (\cdot), \zeta_{\bm{i},I}(\cdot)$ to denote the dependence on the function explicitly. The same will be done for the finite element estimator $\eta_{ {\rm FE},I}(\cdot)$. 
For instance, the parametric estimator from Section~\ref{sec:def_adaptive_algo} { (defined in (\ref{eq:def_param_estim}))} can be written as $\zeta_{SC, I}(U)$, if we denote by $U$ the current discrete finite element solution. In the previous section, in which we assumed to be able to sample the analytic solution, we were dealing with $\zeta_{SC, I}(u)$.

The following lemma will be used in the next sections.
\begin{lemma} \label{lemma:estim_difference_param_estimator_between_solutions}
Given a downward-closed multi-index set $I\subset\N^N$, there holds
\[\vert \zeta_{SC, I}(u) - \zeta_{SC, I}(U) \vert \lesssim  {\left(\sum_{\bm{i}\in\mathcal{M}_I} \Lambda_{\bm{i}} \right)}\eta_{ {\rm FE}, I}(U).\]
\end{lemma}
\begin{proof}
The stability bound (\ref{stab_hierarchical_surplus}) for the hierarchical surplus operator implies
\begin{align*}
\vert \zeta_{SC, I}(u) - \zeta_{SC, I}(U) \vert
& \leq \sum_{\bm{i}\in\mathcal{M}_I} \vert \zeta_{\bm{i}, I}(u) - \zeta_{\bm{i}, I}(U) \vert\\
& \leq \sum_{\bm{i}\in\mathcal{M}_I} \norm{\Delta^{m(\bm{i})}\left( a \nabla S_I[u-U]\right)}_{L^\infty(\Gamma, L^2(D))}\\
& \lesssim \left(\sum_{\bm{i}\in\mathcal{M}_I} \Lambda_{\bm{i}}\right) \norm{\nabla S_I[u-U]}_{L^\infty(\Gamma, L^2(D))}.
\end{align*}
Now we only need to bound the last factor with the finite element estimator:
\begin{align*}
\norm{\nabla S_I[u-U]}_{L^\infty(\Gamma, L^2(D))}
& \leq \sum_{\bm{y}\in\mathcal{H}_I} \norm{\left( u(\bm{y})-U_{\bm{y}}\right)  L_{\bm{y}}}_{L^{\infty}(\Gamma, V)} \\
& \leq \sum_{\bm{y}\in\mathcal{H}_I} \norm{\nabla \left( u(\bm{y}) - U_{\bm{y}}\right) }_{L^2(D)} \norm{L_{\bm{y}}}_{L^{\infty}(\Gamma)}.
\end{align*}
The reliability of the residual-based error estimator in each collocation point $\bm{y}$ {concludes} the proof.
\end{proof}
\subsection{Convergence under h-refinement} \label{sec:h_ref_convergence}
The stochastic collocation finite element algorithm (Algorithm~\ref{algo:SCFE}) delegates to Algorithm \ref{algo:refine_FE} the task of refining the finite element solutions corresponding to the collocation points until the finite element a-posteriori estimator {defined in (\ref{eq:def_FE_estim})} falls below a given tolerance.
Recall that Algorithm \ref{algo:refine_FE} is given a multi-index set $I$, or equivalently a sparse grid $\mathcal{H}_I$ consisting of $N_c$ collocation points that will \emph{not} change during its execution. Hence, we will drop the index $I$ in the following.
Moreover, the index $\ell\in\N$ will denote the current iteration of the adaptive loop starting at Line \ref{algo_FE_ref:line:loop} { (Alg.~\ref{algo:refine_FE})} (so $U_{\ell ,\bm{y}}$ and $\eta_{\ell ,\bm{y}}$ will denote respectively the finite element solution and finite element estimator on the collocation point $\bm{y}\in\HH$ at iteration $\ell$ and with respect to the mesh $\TT_{\ell,\bm{y}}$).\\

{From the theory of the classical h-adaptive finite element algorithm, we have the following contraction property (see, e.g., \cite{Cascon_2008, stevenson07, axioms}) for all $\bm{y}\in\HH$: 
\begin{equation} \label{eq:estred}
\sum_{T\in\TT_{\ell+1,\bm{y}}\setminus\TT_{\ell,\bm{y}}}\eta_{\ell+1, \bm{y},T}^2  \leq q \sum_{T\in\TT_{\ell,\bm{y}}\setminus\TT_{\ell+1,\bm{y}}} \eta_{\ell,\bm{y},T}^2+ C\norm{U_{\ell+1,\bm{y}}-U_{\ell,\bm{y}}}_{V}^2
\end{equation}
as well as
\begin{align}\label{eq:stability}
 \Big(\sum_{T\in\TT_{\ell+1,\bm{y}}\cap\TT_{\ell,\bm{y}}}\eta_{\ell+1, \bm{y},T}^2 \Big)^{1/2} \leq \Big(\sum_{T\in\TT_{\ell,\bm{y}}\cap\TT_{\ell+1,\bm{y}}} \eta_{\ell,\bm{y},T}^2\Big)^{1/2}+ C^{1/2}\norm{U_{\ell+1,\bm{y}}-U_{\ell,\bm{y}}}_{V}
\end{align}
for $0<q<1$ and $C>0$ independent of $\ell$ but depending on the shape-regularity of the mesh and the regularity assumptions on the coefficient $a(\bm{y},\cdot)$ on $\TT_{\rm init}$. Since we use newest-vertex-bisection for mesh refinement, the shape regularity depends only on $\TT_{\rm init}$.}

{
As in the deterministic setting, D\"orfler marking together with~\eqref{eq:estred}--\eqref{eq:stability} can be used to prove a contraction property of the estimator (see also~\cite{bespalov2019convergence} for a similar argument with a slightly different marking strategy).}

{\begin{proposition}\label{prop:linconv}
Given an arbitrary downward closed index set $I\subseteq \N^\N$, Algorithm~\ref{algo:refine_FE} satisfies 
\begin{align}\label{eq:linconv}
 \sum_{\bm{y}\in\HH}\eta_{\ell+k,\bm{y}}^2\leq C_{\rm lin} q_{\rm lin}^k\sum_{\bm{y}\in\HH}\eta_{\ell,\bm{y}}^2
\end{align}
for all $\ell,k\in\N$ and some uniform constants $0<q_{\rm lin}<1$, $C_{\rm lin}>0$. In particular, we have:
\begin{align*}
\lim_{\ell \rightarrow \infty} \norm{S_I[u] - S_I[U_\ell]}_{L^\infty(\Gamma, V)} = 0 = \lim_{\ell \rightarrow \infty} \eta_{ {\rm FE}, I}(U_\ell).
\end{align*}
\end{proposition}
\begin{proof}
We show with~\eqref{eq:estred}--\eqref{eq:stability} that all $\delta>0$ satisfy (recall the definition of $\bigsqcup$ from Line~\ref{algo_FE_ref:line:dorfler_marking} (Alg.~\ref{algo:refine_FE}))
\begin{align*}
\sum_{\bm{y}\in\HH}\eta_{\ell+1,\bm{y}}^2 &= 
 \sum_{(\bm{y},T_{\bm{y}})\in \bigsqcup_{\bm{y}\in \HH} \TT_{\ell+1,\bm{y}}\setminus\TT_{\ell,\bm{y}}} \eta_{\ell+1,\bm{y},T_{\bm{y}}}^2
 +\sum_{(\bm{y},T_{\bm{y}})\in \bigsqcup_{\bm{y}\in \HH} \TT_{\ell+1,\bm{y}}\cap\TT_{\ell,\bm{y}}} \eta_{\ell+1,\bm{y},T_{\bm{y}}}^2\\
 &\leq 
 q\sum_{(\bm{y},T_{\bm{y}})\in \bigsqcup_{\bm{y}\in \HH} \TT_{\ell,\bm{y}}\setminus\TT_{\ell+1,\bm{y}}} \eta_{\ell,\bm{y},T_{\bm{y}}}^2
 +(1+\delta)\sum_{(\bm{y},T_{\bm{y}})\in \bigsqcup_{\bm{y}\in \HH} \TT_{\ell,\bm{y}}\cap\TT_{\ell+1,\bm{y}}} \eta_{\ell,\bm{y},T_{\bm{y}}}^2\\
 &\qquad + C(2+\delta^{-1})\sum_{\bm{y}\in\HH}\norm{U_{\ell+1,\bm{y}}-U_{\ell,\bm{y}}}_V^2\\
 &\leq 
 (q-1)\sum_{(\bm{y},T_{\bm{y}})\in \bigsqcup_{\bm{y}\in \HH} \TT_{\ell,\bm{y}}\setminus\TT_{\ell+1,\bm{y}}} \eta_{\ell,\bm{y},T_{\bm{y}}}^2
 +(1+\delta)\sum_{(\bm{y},T_{\bm{y}})\in \bigsqcup_{\bm{y}\in \HH} \TT_{\ell,\bm{y}}} \eta_{\ell+1,\bm{y},T_{\bm{y}}}^2\\
 &\qquad + C(2+\delta^{-1})\sum_{\bm{y}\in\HH}\norm{U_{\ell+1,\bm{y}}-U_{\ell,\bm{y}}}_V^2. 
\end{align*}
The D\"orfler marking  from Algorithm~\ref{algo:refine_FE} ensures $\mathcal{K}\subseteq \bigsqcup_{\bm{y}\in \HH} \TT_{\ell,\bm{y}}\setminus\TT_{\ell+1,\bm{y}}$ and hence
\begin{align*}
 (q-1)\sum_{(\bm{y},T_{\bm{y}})\in \bigsqcup_{\bm{y}\in \HH} \TT_{\ell,\bm{y}}\setminus\TT_{\ell+1,\bm{y}}} \eta_{\ell,\bm{y},T_{\bm{y}}}^2\leq \theta(q-1)\sum_{(\bm{y},T_{\bm{y}})\in \bigsqcup_{\bm{y}\in \HH} \TT_{\ell,\bm{y}}} \eta_{\ell,\bm{y},T_{\bm{y}}}^2.
\end{align*}
Altogether, we obtain for $\kappa:= 1+\delta - \theta(1-q)$ and $\widetilde C:= C(2+\delta^{-1})$ that
\begin{align*}
 \sum_{\bm{y}\in\HH}\eta_{\ell+1,\bm{y}}^2 &\leq \kappa\sum_{\bm{y}\in\HH}\eta_{\ell,\bm{y}}^2 + \widetilde C \sum_{\bm{y}\in\HH}\norm{U_{\ell+1,\bm{y}}-U_{\ell,\bm{y}}}_V^2.
\end{align*}
With the Galerkin orthogonality 
\begin{align*}
 \sum_{\bm{y}\in\HH}&\norm{a(\bm{y})^{1/2}\nabla(U_{\ell+1,\bm{y}}-U_{\ell,\bm{y}})}_{L^2(D)}^2 \\
 &= \sum_{\bm{y}\in\HH}\Big(\norm{a(\bm{y})^{1/2}\nabla(u(\bm{y})-U_{\ell,\bm{y}})}_{L^2(D)}^2-\norm{a(\bm{y})^{1/2}\nabla(u(\bm{y})-U_{\ell+1,\bm{y}})}_{L^2(D)}^2\Big)
\end{align*}
we may follow~\cite[Section~4]{axioms} verbatim in order to prove~\eqref{eq:linconv}. Since $\#\HH$ is fixed, we have $\sum_{\bm{y}\in\HH}\eta_{\ell,\bm{y}}^2\simeq \eta_{ {\rm FE}, I}(U_\ell)^2$ and reliability proves $\lim_{\ell\to\infty}\eta_{ {\rm FE},I}(U_\ell)=\lim_{\ell\to\infty}\norm{S_I[u]-S_I[U_\ell]}_{L^\infty(\Gamma,V)}=0$. This concludes the statement.    
\end{proof}
\begin{remark}
The previous proposition implies that Algorithm \ref{algo:refine_FE} terminates.  In particular, the algorithm will eventually satisfy the condition $\eta_{ {\rm FE}, I}(U_\ell) < {\rm {Tol_{\ell}}}$, where ${\rm {Tol_{\ell}}}\coloneqq \alpha \frac{1}{{\sum_{\bm{i}\in\mathcal{M}_I}\Lambda_{\bm{i}}}} \zeta_{SC, I}(U_\ell)$.
Indeed, due to Lemma \ref{lemma:estim_difference_param_estimator_between_solutions} we have that, as $\left( \eta_{ {\rm FE}, I_\ell}(U_\ell)\right)_\ell$ vanishes, 
$\zeta_{SC, I}(U_\ell)$ converges to $\zeta_{SC, I}(u) \geq \epsilon >0$, therefore $\lim_{\ell \rightarrow \infty} {\rm Tol_{\ell}} = \alpha \frac{1}{{\sum_{\bm{i}\in\mathcal{M}_I}\Lambda_{\bm{i}}}} \zeta_{SC, I}(u) >0$. Note that the convergence proof uses an $\ell_2$-type estimator instead of an $\ell_1$-type as in $\eta_{ {\rm FE},I}$. In this regard, the $\ell_2$-type might seem more natural and we refer to Section~\ref{sec:others} for further discussion. 
\end{remark}}

{\begin{theorem}\label{thm:optimality}
 Given an arbitrary downward closed index set $I\subseteq \N^\N$, Algorithm~\ref{algo:refine_FE} converges with the optimal rate in the following sense: Let $\mathbb{T}$ denote the set of all meshes which can be obtained from $\TT_{\rm init}$ by iterated newest-vertex-bisection with mesh closure. Let $s>0$ such that
 \begin{align}\label{eq:optrate}
  \sup_{N\in\N} \inf_{\TT_{\bm{y}}\in \mathbb{T}\atop \sum_{\bm{y}\in \HH}\#\TT_{\bm{y}}\leq N} \Big(\sum_{\bm{y}\in\HH}\norm{u(\bm{y})-U_{\TT_{\bm{y}}}}_V^2 +\norm{h_{\TT_{\bm{y}}} (1-\Pi_{\TT_{\bm{y}}})f}_{L^2(D)}^2\Big)^{1/2}N^s<\infty,
 \end{align}
 where $h_{\TT}$ denotes the local mesh-size function and $\Pi_{\TT}$ is the $L^2(D)$-orthogonal projection onto $\TT$-elementwise constant functions. Then, there holds
 \begin{align*}
  \sup_{\ell\in\N} \norm{S_I[u]-S_I[U_\ell]}_{L^\infty(\Gamma,V)}\Big(\sum_{\bm{y}\in \HH}\#\TT_{\ell,\bm{y}}\Big)^s<\infty.
 \end{align*}
\end{theorem}}
{
\begin{proof}
First note that standard upper/lower bounds   for the residual error estimator together with the regularity assumptions on $a(\bm{y},\cdot)$ show $\eta_{\bm{y}}\simeq\sqrt{\norm{u(\bm{y})-U_{\TT_{\bm{y}}}}_V^2 +\norm{h_{\TT_{\bm{y}}} (1-\Pi_{\TT_{\bm{y}}})f}_{L^2(D)}^2}$ and hence~\eqref{eq:optrate} is equivalent to
\begin{align*}
 \sup_{N\in\N} \inf_{\TT_{\bm{y}}\in \mathbb{T}\atop \sum_{\bm{y}\in \HH}\#\TT_{\bm{y}}\leq N} \Big(\sum_{\bm{y}\in\HH}\eta_{\bm{y}}^2\Big)^{1/2}N^s<\infty.
\end{align*}
 With the error norm $||| u- U_\ell|||:=\sqrt{\sum_{\bm{y}\in\HH}\norm{u(\bm{y})-U_{\TT_{\ell,\bm{y}}}}_V^2 }$ and~\eqref{eq:estred}--\eqref{eq:stability}, the estimator $\eta_{{\rm FE},I}$ satisfies (A1) and (A2) from~\cite[Section~3]{axioms}. From the classical theory of $h$-adaptivity~\cite{Cascon_2008}, we immediately obtain discrete reliability~(A3) in the sense
 \begin{align*}
  |||U_{\ell+k}-U_\ell|||^2 = \sum_{\bm{y}\in\HH} \norm{U_{\ell+k,\bm{y}}-U_{\ell,\bm{y}}}_V^2 \leq C_{\rm drel}\sum_{\bm{y}\in\HH}\sum_{T\in \omega(\TT_{\ell,\bm{y}}\setminus\TT_{\ell+k,\bm{y}})}\eta_{\ell,\bm{y},T}^2,
 \end{align*}
 where $\omega(\cdot)$ denotes the set of elements with non-empty intersection with $(\cdot)$.
With these ingredients and the linear convergence from Proposition~\ref{prop:linconv},~\cite[Proposition~4.12 \& Proposition~4.15]{axioms} show optimal convergence of the error estimator
\begin{align*}
  \sup_{\ell\in\N} \sqrt{\sum_{\bm{y}\in\HH}\eta_{\ell,\bm{y}}^2}\Big(\sum_{\bm{y}\in \HH}\#\TT_{\ell,\bm{y}}\Big)^s<\infty.
 \end{align*}
 With constants depending only on the size of $I$, the quantity $\sqrt{\sum_{\bm{y}\in\HH}\eta_{\ell,\bm{y}}^2}$ is equivalent to $\eta_{{\rm FE},I}$ and hence reliability concludes the proof.
\end{proof}
}

\subsection{Proof of convergence of the fully discrete algorithm} \label{sec:vanishing_param_estim_SCFE}
The tolerance for finite element refinement was defined in Algorithm \ref{algo:refine_FE} as: 
\begin{equation} \label{def:tol_FE_refinement_in_SGSC}
\textrm{Tol} = \textrm{Tol}(I, \zeta_{\bm{i}, I}(U), \alpha) \coloneqq \alpha \frac{1}{{\sum_{\bm{i}\in\mathcal{M}_{I}}\Lambda_{\bm{i}}}} \zeta_{SC, I}(U).
\end{equation}
where $\alpha\in(0,1)$, $\Lambda_{\bm{i}}$ was defined in (\ref{def:bound_stab_const}) and $\zeta_{SC, I}(U)$ is the parametric a-posteriori error estimator.
This choice is motivated by the following estimate: For fixed downward closed $I\subset \N^N$, Lemma~\ref{lemma:estim_difference_param_estimator_between_solutions} shows
\begin{equation*} 
\zeta_{SC, I}(U) \leq \zeta_{SC, I}(u) + {\left(\sum_{\bm{i}\in\mathcal{M}_{I}}\Lambda_{\bm{i}}\right)}\eta_{ {\rm FE}, I}(U) 
\leq \zeta_{SC, I}(u) + \alpha \zeta_{SC, I}(U),
\end{equation*}
and hence
\begin{equation}\label{bound_ptwise_estimator_from_FE_refinement}
\zeta_{SC, I}(U)  \leq \frac{1}{1-\alpha} \zeta_{SC, I}(u).
\end{equation}
In the context of the adaptive algorithm, this implies that $\left(\zeta_{SC, I_\ell}(U_\ell)\right)_\ell$ is uniformly bounded since $\left(\zeta_{SC, I_\ell}(u)\right)_\ell$ is. This last fact was proved in Remark \ref{rk:boundedness_param_estim} using the estimate on the pointwise error estimator from Proposition \ref{prop:estim_ptwise_estimator}.

\begin{lemma} \label{lemma:vanishing_profit_SCFE}
Algorithm~\ref{algo:SCFE} 
with either workless profit {(and $0<\alpha<1$ sufficiently small)} or profit with work {(and arbitrary $0<\alpha<1$)} 
and the tolerance~\eqref{def:tol_FE_refinement_in_SGSC} satisfies
$\lim_{\ell \rightarrow\infty} \mathcal{P}_{\bm{i}_\ell, I_\ell} = 0$.
\end{lemma}
\begin{proof}
We consider the two definitions of profit separately:

{\em Profit with work}: 
$ \mathcal{P}_{\bm{i}, I} \coloneqq \frac{\sum_{\bm{j}\in A_{\bm{i}, I}} \zeta_{\bm{j}, I}(U)}{\sum_{\bm{j}\in A_{\bm{i}, I}} W_{\bm{j}}}.$\\
The uniform boundedness of the parametric a-posteriori error estimator, together with the fact that the work over $A_{\bm{i}_\ell, I_\ell}$ diverges, gives
\[\mathcal{P}_{\bm{i}_\ell, I_\ell} \leq \frac{\zeta_{SC, I_\ell}(U_\ell)}{\sum_{\bm{j}\in A_{\bm{i}_\ell, I_\ell}} W_{\bm{j}}} 
\lesssim \frac{1}{\sum_{\bm{j}\in A_{\bm{i}_\ell, I_\ell}} W_{\bm{j}}}\rightarrow 0.\]

{\em Workless profit}: 
$ \mathcal{P}_{\bm{i}, I} \coloneqq \sum_{\bm{j}\in A_{\bm{i}, I}} \zeta_{\bm{j}, I}(U).$
We recall from~\eqref{eq:partition} that, for the profit-maximizer $\bm{i}_\ell\in \mathcal{M}_{I_\ell}$, $\mathcal{P}_{\bm{i}_\ell, I_\ell} \geq \frac{1}{N} \zeta_{SC, I_\ell} (U)$. Thus, Lemma~\ref{lemma:estim_difference_param_estimator_between_solutions} shows
\begin{align*}
\mathcal{P}_{\bm{i}_\ell, I_\ell} 
& \leq \sum_{\bm{j}\in A_{\bm{i}_\ell, I_\ell}} \zeta_{\bm{j}, I_\ell}(u) 
+ \alpha \frac{{\sum_{\bm{j}\in A_{\bm{i}_\ell, I_\ell}} \Lambda_{\bm{j}}}}{{\sum_{\bm{j}\in \mathcal{M}_{I_\ell}} \Lambda_{\bm{j}}}} \zeta_{SC, I_\ell}(U_\ell)\\
& \leq \sum_{\bm{j}\in A_{\bm{i}_\ell, I_\ell}} \zeta_{\bm{j}, I_\ell}(u)
+ \alpha \frac{{\sum_{\bm{j}\in A_{\bm{i}_\ell, I_\ell}} \Lambda_{\bm{j}}}}{{\sum_{\bm{j}\in \mathcal{M}_{I_\ell}} \Lambda_{\bm{j}}}} N \mathcal{P}_{\bm{i}_\ell, I_\ell}\\
& \leq \sum_{\bm{j}\in A_{\bm{i}_\ell, I_\ell}} \zeta_{\bm{j}, I_\ell}(u)
+ \alpha N \mathcal{P}_{\bm{i}_\ell, I_\ell},
\end{align*}
so
\begin{equation*}
\mathcal{P}_{\bm{i}_\ell, I} \leq \frac{1}{1-\alpha N}\sum_{\bm{j}\in A_{\bm{i}_\ell, I_\ell}} \zeta_{\bm{j}, I_\ell}(u)\rightarrow 0 \qquad \textrm{as } l\rightarrow \infty.
\end{equation*}
Observe that this introduces the constraint on $\alpha$ with respect to the number of dimensions: $\alpha < N^{-1}$.
This constraint can be improved by replacing the crude estimate
\[\frac{\sum_{\bm{j}\in A_{\bm{i}_\ell, I_\ell}} \Lambda_{\bm{j}}}{\sum_{\bm{j}\in \mathcal{M}_{I_\ell}} \Lambda_{\bm{j}}} \leq 1,\]
with the better bound
\[\alpha \leq \left(  \frac{\max_{n\in 1,\ldots, N} {\sum_{\bm{j}\in A_{\bm{i}_{\ell-1} +\bm{e}_n, I_\ell}} \Lambda_{\bm{j}}}}{{\sum_{\bm{j}\in \mathcal{M}_{I_\ell}} \Lambda_{\bm{j}}}} N \right)^{-1}.\]
This concludes the proof.
\end{proof}

We can finally prove that the error estimator vanishes with a technique similar to that used in Theorem \ref{th:convergence_param_only_algo} for the parametric algorithm.
\begin{theorem}\label{thm:SCFEconv}
{Algorithm~\ref{algo:SCFE} 
with either workless profit (and $0<\alpha<1$ sufficiently small) or profit with work (and arbitrary $0<\alpha<1$)
and the tolerance~\eqref{def:tol_FE_refinement_in_SGSC} satisfies the following:
The sequence of parametric a-posteriori error estimators $\left( \zeta_{SC, I_\ell}(U_\ell) \right)_\ell$ vanishes
\[\lim_{\ell \rightarrow \infty} \zeta_{SC, I_\ell}(U_\ell) = 0.\]
Thus, also the finite element error estimator vanishes
\[\lim_{\ell \rightarrow \infty} \eta_{ {\rm FE}, I_\ell}(U_\ell) = 0,\]
and the reliability of the a-posteriori error estimator implies error convergence 
\[\lim_{\ell \rightarrow \infty} \norm{u-S_{I_\ell}[U_\ell]}_{L^{\infty}(\Gamma, V)} = 0.\]}
\end{theorem}
\begin{proof}
The a-posteriori error estimator can be expressed as
\[\zeta_{SC, I_\ell} (U_\ell) = \sum_{\bm{i}\in\N^N} \zeta_{\bm{i}, I_\ell}(U_\ell) \mathbbm{1}_{\mathcal{M}_{I_\ell}}(\bm{i}).\]
Since the sequence $\left( \zeta_{SC, I_\ell}(U_\ell) \right)_\ell$ is uniformly bounded~\eqref{bound_ptwise_estimator_from_FE_refinement}, it is sufficient to prove that 
$\left(\zeta_{\bm{i}, I_\ell}(U_\ell) \mathbbm{1}_{\mathcal{M}_{I_\ell}}\right)_\ell$ vanishes for any fixed $\bm{i}\in\N^N$.
We can distinguish three cases:
\begin{itemize}
\item if $\bm{i}$ is eventually added to $I_\ell$, then $\mathbbm{1}_{\mathcal{M}_{I_\ell}}(\bm{i})$ is eventually zero;
\item if $\bm{i}$ is never added to the margin $\mathcal{M}_{I_\ell}$, then $\zeta_{\bm{i}, I_\ell}(U_\ell)$ is constantly zero;
\item finally, if it exists $\bar{\ell}\in\N$ such that for all $\ell>\bar{\ell}$, $\bm{i}\in\mathcal{M}_{I_\ell}$, then $\lim_{\ell \rightarrow \infty} \zeta_{\bm{i}, I_\ell}(U_\ell) = 0.$ Indeed, because of Lemma \ref{lemma:vanishing_profit_SCFE}, $\lim_{\ell \rightarrow \infty} \mathcal{P}_{\bm{i}, I_\ell} = 0$ (for both workless profit and profit with work), thus $\left(\zeta_{\bm{i}, I_\ell}(U_\ell)\right)_\ell$ vanishes as in Proposition \ref{prop:multi_index_in_margin_has_0_ptwise_estim}.
\end{itemize}
This concludes the proof.
\end{proof}

\subsection{{Other versions of the finite element estimator}}\label{sec:others}

In the previous section we followed~\cite{guignard2018posteriori} to derive the estimator via
\begin{align*}
S_I \left[ \int_D fv- a \nabla S_I[U] \cdot \nabla v \right] 
& = \sum_{\bm{y}\in \mathcal{H}_I} \left[ \int_D fv- a(\bm{y}) \nabla S_I[U](\bm{y}) \cdot \nabla v  \right] L_{\bm{y}}\\
& \leq C \sum_{\bm{y}\in \mathcal{H}_I} \eta_{\bm{y}} \vert L_{\bm{y}}\vert \norm{\nabla v}_{L^2(D)}.
\end{align*}
Choosing $v = u - S_I[U]$ and taking the $L^{\infty}(\Gamma)$ norm leads to the estimator we used above, i.e., 
\begin{align*}
\norm{\sum_{\bm{y}\in \mathcal{H}_I} \eta_{\bm{y}} \vert L_{\bm{y}}\vert}_{L^{\infty}(\Gamma)} \leq
\sum_{\bm{y}\in \mathcal{H}_I} \eta_{\bm{y}} \norm{L_{\bm{y}}}_{L^{\infty}(\Gamma)} = \eta_{ {\rm FE}, I}(U). 
\end{align*}
Using the H\"older estimates with other combinations of $(p,q)\in\{ (2,2),(\infty,1)\}$, we obtain
\begin{align*}
\norm{\sum_{\bm{y}\in \mathcal{H}_I} \eta_{\bm{y}} \vert L_{\bm{y}}\vert}_{L^{\infty}(\Gamma)} 
& \leq \norm{\left(\sum_{\bm{y}\in \mathcal{H}_I} \eta_{\bm{y}}^2\right)^{\frac{1}{2}} \left(\sum_{\bm{y}\in \mathcal{H}_I} \vert L_{\bm{y}}\vert^2\right)^{\frac{1}{2}}}_{L^{\infty}(\Gamma)} 
= \eta_{p,I} \Lambda_{q,I},
\end{align*}
where
\begin{align*}
\eta_{p, I} \coloneqq 
\begin{cases} \left(\sum_{\bm{y}\in \mathcal{H}_I} \eta_{\bm{y}}^2\right)^{\frac{1}{2}} &p=2,\\
                        \max_{\bm{y}\in \mathcal{H}_I} \eta_{\bm{y}} &p=\infty,
\end{cases}
\quad\text{and}\quad
\Lambda_{q,I} \coloneqq \begin{cases}
\norm{\left(\sum_{\bm{y}\in \mathcal{H}_I} \vert L_{\bm{y}}\vert^2\right)}_{L^{\infty}(\Gamma)}^{\frac{1}{2}} & q=2,\\
\norm{\sum_{\bm{y}\in \mathcal{H}_I} \vert L_{\bm{y}}\vert}_{L^{\infty}(\Gamma)} &q=1.
\end{cases}
\end{align*}
The perturbation result from Lemma~\ref{lemma:estim_difference_param_estimator_between_solutions} can be analogously modified to obtain:
\begin{align*}
\vert \zeta_{SC, I}(u) - \zeta_{SC, I}(U) \vert
& \lesssim \left(\sum_{\bm{i}\in\mathcal{M}_I} \Lambda_{\bm{i}}\right) \eta_{p,I} \Lambda_{q,I}
\end{align*}
From these results, the sufficient condition in~\eqref{def:tol_FE_refinement_in_SGSC} for convergence becomes respectively
\begin{align*}
\eta_{p,I}(U) \leq \alpha \left( \Lambda_{q,I} \sum_{\bm{i}\in\mathcal{M}_I} \Lambda_{\bm{i}} \right)^{-1} \zeta_{SC, I}(U).
\end{align*}
With these ingredients, all the other results of the previous sections hold for the variants of the finite element estimator discussed above.
\subsection{{Convergence of a single mesh version of the fully discrete algorithm}}\label{sec:singlemesh}
We also consider SCFE with the same \emph{adaptively} refined mesh in all collocation points. The idea is that, if the set of singularities of the solution $u$ is small, one single adaptive mesh can resolve all of them simultaneously and thus substantially reduce the computational effort. We employ the following estimator from~\cite[Remark 4.4]{guignard2018posteriori} for the finite element part
\begin{align*}
\eta_{ {\rm FE}, I} (U) &\coloneqq \left(\sum_{T\in\TT}\eta^2_{T}(U)\right)^{1/2},\qquad 
\eta_{T}(U) \coloneqq \norm{\eta_{T}(\ \cdot\ ; U)}_{L^{\infty}(\Gamma)},\\
\eta^2_{T}(\bm{y}; U) &\coloneqq h_T^2 \norm{ S_I \left[f+\nabla\cdot(a\nabla U) \right](\bm{y})}_{L^2(T)}^2
+ \sum_{e\subset \partial T} h_T \norm{S_I \left[\left[ a\nabla U\cdot \bm{n}_e \right]_{\bm{n}_e}\right](\bm{y})}_{L^2(e)}^2.
\end{align*}

\bigskip

Since we use a single mesh for all collocation points $\bm{y}\in \HH_I$,
we replace $\TT_{\bm{y}}$ in Algorithm~\ref{algo:refine_FE} by $\TT$. We change the D\"orfler marking in Line~\ref{algo_FE_ref:line:dorfler_marking} (Alg.~\ref{algo:refine_FE}) to: Find  minimal $\mathcal{K}\subseteq \TT$ such that
\begin{align*}
 \sum_{T\in\mathcal{K}} \eta_T(U)^2 \geq \theta \eta_{{\rm FE},I}^2.
\end{align*}
Moreover, we replace the refinement loop in Line~\ref{algo_FE_reg:NVB} (Alg.~\ref{algo:refine_FE}) by a single refinement of the mesh $\TT$ with marked elements $\mathcal{K}$.

\bigskip

Due to the fact that $U\colon \Gamma \to \mathcal{S}_0^1(\mathcal{T})$
admits a holomorphic extension to $\Sigma(\Gamma,\bm{\tau})$ just as does $u$ (the same arguments work also for the discrete approximation),
the convergence analysis of the parametric enrichment algorithm remains unchanged (Section~\ref{sec:paramconv}), we now have to show convergence of the adaptive finite element subroutine. With this, we may analogously employ the results of Section~\ref{sec:vanishing_param_estim_SCFE} to obtain convergence of the full algorithm. Note that we can not directly transfer the proof of Proposition~\ref{prop:linconv} as the definition of $\eta_{{\rm FE},I}$ in this section mixes $L^2$-norms and $L^\infty$-norms. 

In this setting, the multi-index set $I\subset \N^N$ is fixed. 
We denote by $\TT_{\ell}$ the finite element mesh at step $\ell>0$ (the same for every collocation point). 
$U_{\ell}$ represents the discrete solution at step $\ell$ and $U_{\ell, \bm{y}}\in \mathcal{S}_0^1(\TT_{\ell})$ its value on a collocation point $\bm{y}\in \HH_I$.
We simplify the notation for the estimator as
$\eta_{\ell} \coloneqq \eta_{ {\rm FE}, I}(U_{\ell})$, 
$\eta_{\ell}(U)\coloneqq \left(\sum_{T\in\TT_{\ell}} \eta_T(U)^2\right)^{1/2}$.\\
We first give a perturbation estimate localized on one element $T$ of a mesh $\TT$, analogously to~\cite[Proposition 3.3]{Cascon_2008}. 
\begin{lemma}\label{lem:perturb}
Consider a shape-regular mesh $\TT$ obtained by NVB from a mesh $\TT_{\rm init}$.
There holds for $U,W \in C^0(\Gamma, \mathcal{S}_0^1(\TT))$ that
\begin{align}
\eta_T(U) \leq \eta_T(W) + C\norm{S_I}_{\mathcal{L}(L^\infty(\Gamma,L^2(D)))} \max_{\bm{y}\in\HH}\norm{\nabla (U(\bm{y})-W(\bm{y}))}_{L^2(\omega (T))} \qquad \forall T\in\TT,
\end{align}
where $\omega(T)$ is the union of the elements sharing an edge with $T$, $C>0$ depends only on $a$ and $\TT_{\rm init}$.
\end{lemma}
\begin{proof}
For any fixed $\bm{y}\in\Gamma$, the linearity of $S_I$ and the triangle inequality yield
\begin{align*}
\eta_T(\bm{y}; U) \leq \eta_T(\bm{y}; W)
+ h_T \norm{S_I\left[\nabla\cdot \left(a\nabla (U-W)\right) \right](\bm{y})}_{L^2(T)} \\
+ h_T^{1/2}\sum_{e \subset \partial T} \norm{S_I\left[ [ a\nabla (U-W)\cdot \bm{n}_e ]_{\bm{n}_e}\right](\bm{y})}_{L^2(e)}.
\end{align*}
With the operator norm of $S_I$, we obtain
\begin{align*}
\norm{S_I\left[\nabla\cdot \left(a\nabla (U-W)\right) \right](\bm{y})}_{L^2(T)}
&\leq \norm{S_I}_{\mathcal{L}(L^\infty(\Gamma,L^2(D)))} \max_{\bm{y}\in\HH} \norm{\nabla a(\bm{y})}_{L^\infty(T))} \norm{\nabla (U-W)(\bm{y})}_{L^2(T)}.
\end{align*}
Analogously, for the jump terms $[ a\nabla (U-W)\cdot \bm{n}_e ]_{\bm{n}_e}$ with $e\subset \partial T$, we obtain, following the same steps as in \cite[Proposition 3.3]{Cascon_2008},
\begin{align*}
\sum_{e\subset \partial T} \norm{S_I\left[ [ a\nabla (U-W)\cdot \bm{n}_e ]_{\bm{n}_e}\right](\bm{y})}_{L^2(e)}
\lesssim \norm{S_I}_{\mathcal{L}(L^\infty(\Gamma,L^2(D)))} \max_{\bm{y}\in\HH} \norm{a(\bm{y})}_{L^\infty(\omega(T)))} \norm{\nabla (U-W)(\bm{y})}_{L^2(\omega(T))}.
\end{align*}
This concludes the proof.
\end{proof}
\begin{proposition}
The sequence of finite element estimators $\eta_{\ell}$ obtained from the single mesh adaptive algorithm satisfies 
\begin{align*}
\lim_{\ell\rightarrow \infty} \eta_{\ell} = 0.
\end{align*}
\end{proposition}
\begin{proof}
With the perturbation estimate from Lemma~\ref{lem:perturb}, we may follow~\cite[Section~4.3]{axioms} to show estimator reduction 
\begin{align} \label{eq:perturb_total_estim}
\eta_{\ell+1}^2 \leq q\eta_{\ell}^2 + C^2\norm{S_I}^2 \sum_{T\in\TT_\ell} \max_{\bm{y}\in\HH}\norm{\nabla (U_{\ell+1}(\bm{y})-U_\ell(\bm{y}))}_{L^2(\omega (T))}^2
\end{align}
for some universal $0<q<1$ and all $\ell\in\N$.

To show that the second term in (\ref{eq:perturb_total_estim}) vanishes, we first observe that $(U_{\ell}(\bm{y}))_{\ell\in\N}$ converges in $V$ for all $\bm{y}\in\HH$.
Indeed, for any fixed $\bm{y}\in\HH$, the nestedness of the finite element spaces $V_{\ell}$ guarantees the existence of $U_{\infty}(\bm{y})\in \overline{\bigcup_{\ell} V_{\ell}}\subset V$ such that $\lim_{\ell\rightarrow \infty} \norm{U_{\infty}(\bm{y}) - U_{\ell}(\bm{y})}_{V}=0$ by C\'ea's lemma (see, e.g.,~\cite[Section~3.6]{axioms}).
This implies that
\begin{align}\label{eq:vanishing_limit_increments}
\lim_{\ell\rightarrow \infty} \norm{\nabla(U_{\ell+1}-U_{\ell})(\bm{y})}_{L^2(D)}= 0
\end{align}	
for all $\bm{y}\in\HH$. Since $\#\HH$ is fixed in the finite element refinement loop of the adaptive algorithm, we have
\begin{align*}
 \sum_{T\in\TT_\ell} \max_{\bm{y}\in\HH}\norm{\nabla (U_{\ell+1}(\bm{y})-U_\ell(\bm{y}))}_{L^2(\omega (T))}^2&\leq \sum_{T\in\TT_\ell} \sum_{\bm{y}\in\HH}\norm{\nabla (U_{\ell+1}(\bm{y})-U_\ell(\bm{y}))}_{L^2(\omega (T))}^2\\
 &\lesssim\sum_{\bm{y}\in\HH}\norm{\nabla (U_{\ell+1}(\bm{y})-U_\ell(\bm{y}))}_{L^2(D)}^2\to 0
\end{align*}
as $\ell\to\infty$.
Passing to the limit superior in~\eqref{eq:perturb_total_estim} shows
$
0 \leq \limsup_{\ell\rightarrow\infty} \eta_{\ell+1}
\leq q\limsup_{\ell\rightarrow\infty} \eta_{\ell}
$
and thus concludes $\lim_{\ell\to\infty}\eta_\ell=0$.
\end{proof}
Altogether, we obtain the convergence result analogously to Theorem~\ref{thm:SCFEconv}.
\begin{theorem}\label{thm:singleSCFEconv}
The single mesh SCFE algorithm discussed in this section satisfies the following:
The sequence of parametric a-posteriori error estimators $\left( \zeta_{SC, I_\ell}(U_\ell) \right)_\ell$ vanishes
\[\lim_{\ell \rightarrow \infty} \zeta_{SC, I_\ell}(U_\ell) = 0.\]
Thus, also the finite element error estimator vanishes
\[\lim_{\ell \rightarrow \infty} \eta_{ {\rm FE}, I_\ell}(U_\ell) = 0,\]
and the reliability of the a-posteriori error estimator implies error convergence 
\[\lim_{\ell \rightarrow \infty} \norm{u-S_{I_\ell}[U_\ell]}_{L^{\infty}(\Gamma, V)} = 0.\]
\end{theorem}

\subsection{Cost of the stochastic collocation algorithms}
\revision{Under the assumption that the pointwise estimators $\zeta_{\boldsymbol{i},I}(U)$ and $\eta_{\boldsymbol{y},T}(U)$ can be computed from the discrete solution in $\mathcal{O}(1)$, each step of the adaptive loop (all algorithms) is linear with respect to the number of degrees of freedom of the current sparse grid and spatial meshes. Indeed, a properly preconditioned iterative solver computes $U$ in linear cost (depending on $a_{\rm min}$ and $a_{\rm max}$). The D\"orfler marking in Algorithm~\ref{algo:refine_FE} requires sorting when done in a naive way, but can be improved to linear cost by binning~\cite{stevenson07} or by a clever variation of the quick-select algorithm~\cite{carl}. Finally, the refinement of the finite element meshes $\TT_{\boldsymbol{y}}$ via newest-vertex-bisection can be done in linear cost~\cite{stevenson08}.}

\revision{
As discussed in Section~\ref{sec:numerics} below, the computation of the $L^\infty(\Gamma)$ and $L^2(D)$-norms for $\zeta_{\boldsymbol{i},I}$ is done via a random sample/Monte-Carlo procedure. This results in constant cost $\mathcal{O}(1)$ and the numerical experiments below show that the approximation error is negligible. A precise convergence analysis of this procedure would be interesting but is beyond the scope of this work.
Each random sample requires the evaluation of the sparse grid interpolant. Theoretically, the cost of the evaluation of the sparse grid interpolation operator is linear in terms of collocation points, after a quadratic set-up cost. Practically, however, the cost of computing the discrete solutions is expected to dominate significantly.}

\section{Numerical experiments}\label{sec:numerics}
The Matlab implementation of Algorithm \ref{algo:SCFE} used to produce the numerical results presented in this section is based on {two Matlab libraries. For sparse grid algorithms, the \emph{Sparse Grids Kit} \cite{back.nobile.eal:comparison} was used. The implementation of the adaptive P1 finite element methods is from the { \emph{p1afem}} Matlab package \cite{funken2011efficient}, which uses Matlab's direct solver for sparse matrices. For further details about parameters and algorithm used within these libraries, the reader is referred to the respective documentations.} The parts of the algorithm that deal with parameter enrichment (e.g. Algorithm \ref{algo:Refine_param_space}) were implemented following the guidelines from \cite{guignard2018posteriori}.

In order to compute the $L^{\infty}(\Gamma)$ norm approximately, we consider {a set $\Theta$ of 500 uniformly distributed random points in $\Gamma$} and approximate, for any $g\in C^0(\Gamma)$, $\norm{g}_{L^{\infty}(\Gamma)} {\approx}\max_{\bm{y}\in\Theta} \vert g(\bm{y})\vert$.
The computation of the $L^2(D)$ norm is carried out with Monte Carlo integration: Given $f\in L^2(D)$, we denote by {$\Pi$ a set of 500 uniformly distributed random points in $D$} and approximate $\norm{f}_{L^2(D)}^2 {\approx} \frac{1}{{\#\Pi}} \sum_{x\in \Pi} f(x)^2$.
{The reason Monte Carlo integration is used is that for a generic $\bm{y}\in\Gamma$ the discrete solution $S_I[U](\bm{y})$ belongs to the finite element space $\mathcal{S}^1_0(\TT)$, where $\TT$ is the coarsest common refinement of the meshes $\TT_{\bm{y}}$. Therefore, in order to compute the exact $L^2(D)$ norm of the function, it would be necessary to compute $\TT$, which would lead to a significant computational overhead. In numerical experiments, we have observed that increasing $\#\Pi$ does not lead to a significant improvement in the approximation of the $L^2(D)$ norm, thus suggesting that the approximation error can be neglected.
In the numerical examples presented in the next sections, we approximate the error between the exact solution $u$ and a discrete solution $S_I[U]$ by $\norm{u - S_{I} [U]}_{L^{\infty}(\Gamma,V)}\approx \norm{u_{\rm approx} - S_I [U]}_{L^{\infty}(\Gamma,V)}$, where $u_{\rm approx}$ is a discrete solution obtained as the last iteration of the single mesh version of SCFE. To approximate the $L^{\infty}(\Gamma,V)$-norm appearing in the error, we use the same method detailed above, just with $\#\Theta = \# \Pi = 5000$.}

{To drive parametric refinement, we employ only profits with work as defined in~\eqref{def:profit_work}. As observed in Section \ref{sec:rks_workless_profit}, workless profits lead to a tensor-product interpolant and thus less interesting results.
For all examples, we consider the finite element estimator is $\eta_{FE,I} = \eta_{2,I} \Lambda_{2,I}$ as defined in Section~\ref{sec:others}.
The D\"orfler parameter for refinement is chosen as $\theta = 0.7$ and, as default mesh $\mathcal{T}_{\rm init}$, a quasi-uniform mesh with 512 triangles and 289 vertices.}

In order to decrease the memory requirements of the program, the finite element refinement tolerance from Section~\ref{sec:others} is modified as follows:
\begin{align}\label{def:big_tol}
\textrm{Tol} \coloneqq \alpha \Lambda_{2,I}^{-1} \zeta_{SC, I},
\end{align}
i.e. we neglect the term depending on the margin of $I$. 
{In the experiments below, we observe that this choice does not compromise convergence. Further investigations will have to be carried out in order to understand whether or not the sufficient condition for convergence can be weakened.
The constant $\alpha$ appearing in~\eqref{def:big_tol} is chosen as $\alpha=0.9$. A value of $\alpha$ close to one shifts the balance between finite element refinement and parameter enrichment  towards the latter one.}

In order to improve the computational efficiency, we use the following shortcut in the implementation of Algorithm \ref{algo:refine_FE}: Instead of re-computing the tolerance $\rm Tol$ at each iteration of the loop, we update it only at the end and, if needed, keep refining the finite element solutions. We alternate these two steps until the finite element estimator falls below the tolerance.

{In the following two sections, we consider a physical domain $D= (0,1)^2$ and denote $x = (x_1, x_2)\in D$. The parametric domain is $\Gamma = [-1,1]^N$ for an integer $N$ representing the number of parametric dimensions of the problem.}

We recall that for a numerical solution $S_{I}[U]$ obtained with SCFE, its number of degrees of freedom is proportional to $M\coloneqq \sum_{\bm{y}\in\HH_I} \#\TT_{\bm{y}}$,
where $\#\TT_{\bm{y}}$ is the number of vertices of the mesh corresponding to the collocation point $\bm{y}$ (or equivalently the dimension of the finite element space $V_{\bm{y}}$ up to boundary conditions).

\subsection{First example: Karhunen–Loève expansion with $\textbf{N=5, 11}$}
We consider a constant forcing term $f(x)\equiv 1$ and the following diffusion coefficient with affine dependence on the parameter $\bm{y}\in\Gamma$:
\begin{align}\label{eq:diffusion_KL}
a(x, \bm{y}) = a_0(x) +\frac{1}{3} \left( a_1(x) y_1 + \sum_{n=2}^N a_n(x) y_n \right),
\end{align}
where $a_0(x) \equiv 1$, $a_1(x) \equiv\left( \frac{\sqrt{\pi} L}{2}\right)^{1/2}$ and, for $n>1$, 
\begin{align*}
\lambda_n = \left( \sqrt{\pi} L\right)^{1/2} \exp \left(-\frac{\left(\left\lfloor\frac{n}{2}\right\rfloor \pi L\right)^2}{8} \right),\\
a_n (x) =
\begin{cases}
\sqrt{\lambda_n} \sin{\left(n \pi x_1\right)}\qquad {\rm if }\ n\ {\rm even}\\
\sqrt{\lambda_n} \cos{\left(n \pi x_1\right)}\qquad {\rm if }\ n\ {\rm odd},
\end{cases}
\end{align*}
where $L\in(0,1)$ is a constant.
Such a diffusion coefficient is the result of the Karhunen–Loève expansion \cite{schwab2006karhunen} of the random field $a(x,\omega)$ with mean $a_0$ and covariance 
$Cov(x, x') = \frac{1}{3^2} \exp \left(-\frac{(x_1 - x_1')^2}{L^2}\right)$, for $x,x'\in D$. The constant $L$ denotes the ``correlation length'' of the stochastic parameter.
We choose $L=0.5$, which implies
\begin{align*}
a_1\approx  6.7\cdot 10^{-1},\;\lambda_1 \approx 6.9\cdot 10^{-1},\;
\lambda_2 \approx 2.7\cdot 10^{-1},\;
\lambda_3 \approx 5.8\cdot 10^{-2},\;
\lambda_4 \approx 6.8\cdot 10^{-3},\;
\lambda_5 \approx 4.2\cdot 10^{-4}.
\end{align*}
In the rest of this section, we truncate the expansion to $N=5$ and $N=11$ terms. The aim is to study how the algorithm performs for different numbers of parametric dimensions $N$ on an anisotropic problem, where the first parameters are more relevant than the last ones.

In Figure \ref{fig:explanation_algo}, we use the problem with $N=5$ parameters to provide the reader with a concrete example of the steps of the algorithm.
On the left, we plot the evolution of the estimators with respect to the number of degrees of freedom. We plot the values of the estimators any time they are computed (not only once per iteration). The algorithm alternates between steps of parameter enrichment and mesh refinement. 
The spikes in the value of the finite element estimator correspond to the parametric enrichment steps, when new collocation points are added to the sparse grid with the initial (coarse) mesh $\TT_{\rm init}$.
When finite element refinement is carried out, the finite element estimator eventually decreases with order $M^{-1/2}$, as has to be expected for lowest order adaptive FEM (see also Theorem~\ref{thm:optimality}).
On the right-hand side of Figure \ref{fig:explanation_algo}, we plot the estimator only once per iteration. As prescribed in (\ref{def:big_tol}), the finite element estimator is bounded from above by the parametric estimator after each finite element refinement loop.
\begin{figure}
\includegraphics[width=0.45\linewidth]{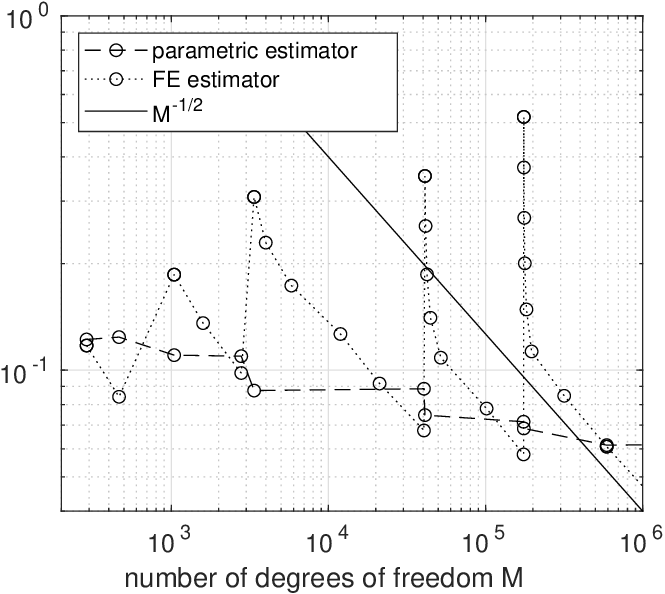}
\includegraphics[width=0.45\linewidth]{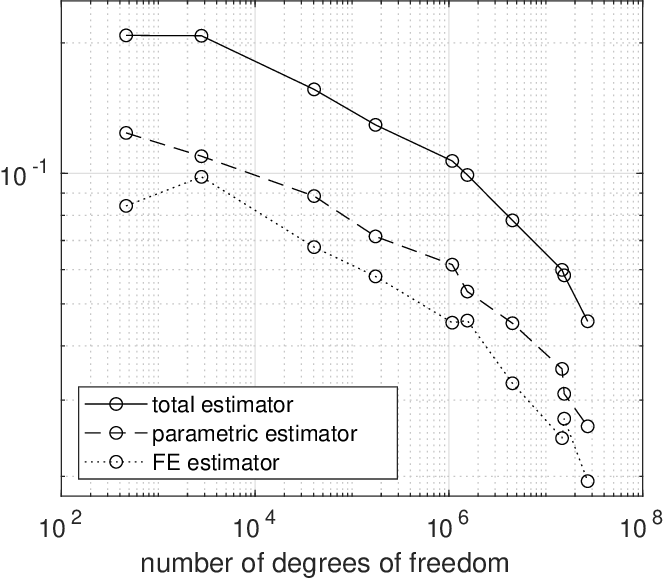}
\caption{First results for SCFE applied to the problem with Karhunen–Loève expansion $N=5$. Left: ``detailed'' evolution of the estimators, i.e. reporting their values any time they are computed during the execution. Right: Total, parametric and finite element estimators at every iteration.}
\label{fig:explanation_algo}
\end{figure}

In Figure \ref{fig:results_estims_errors_KL} we compare the results for $N=5$ and $N=11$.
On the left, the value of total estimator and reference error are plotted as a function of the number of degrees of freedom. The problem with $N=11$ gives larger estimator and reference error. However, the difference is marginal, suggesting that the algorithm successfully detects the anistropy of the problem.
On the right, we plot the effectivity index (ratio between estimator and error). As observed in \cite{guignard2018posteriori}, the number of problem dimensions affects the efficiency of the estimator.
In view of these facts, the algorithm may benefit from an adaptive dimension selection step as the one proposed in \cite[Section 7]{guignard2018posteriori}.
\begin{figure}
\includegraphics[width=0.45\linewidth]{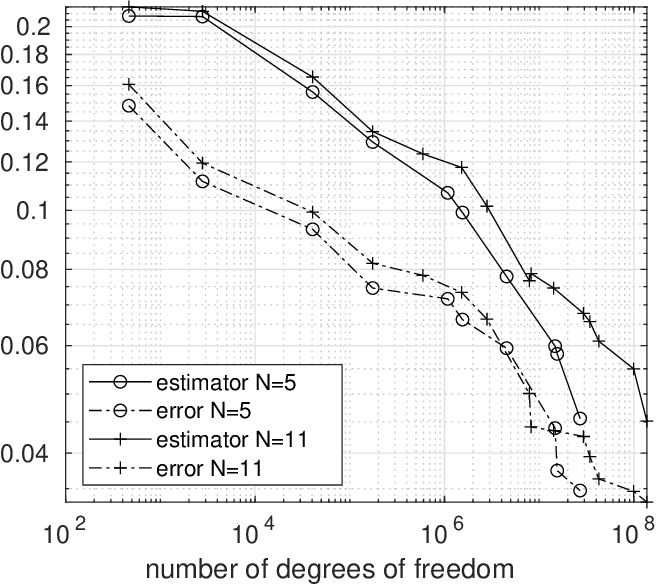}
\includegraphics[width=0.45\linewidth]{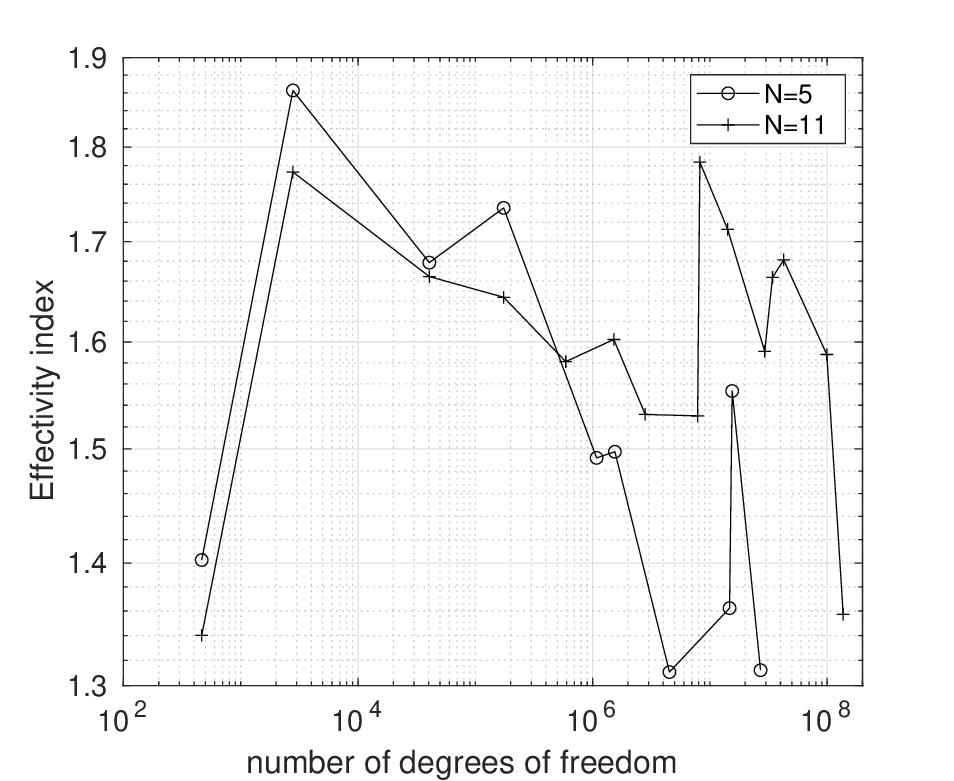}
\caption{Comparing SCFE applied to the problem with Karhunen–Loève expansion for $N=5$ and $N=11$. Left: Total estimator and error. Right: Effectivity index.}
\label{fig:results_estims_errors_KL}
\end{figure}

In Figure \ref{fig:final_mid_set}, we consider the problem with $N=11$ and plot projections of the final multi-index set $I$. The projections are obtained selecting pairs of parametric dimensions $n_1, n_2 \in 1,...,N$ and plotting the 2D set $\left\{ (i_{n_1}, i_{n_2}), \bm{i}\in I\right\}$. Observe how larger values are achieved by the first parametric dimensions, confirming that the algorithm manages to detect the anisotropy of the problem.
\begin{figure}
\includegraphics[width=\linewidth]{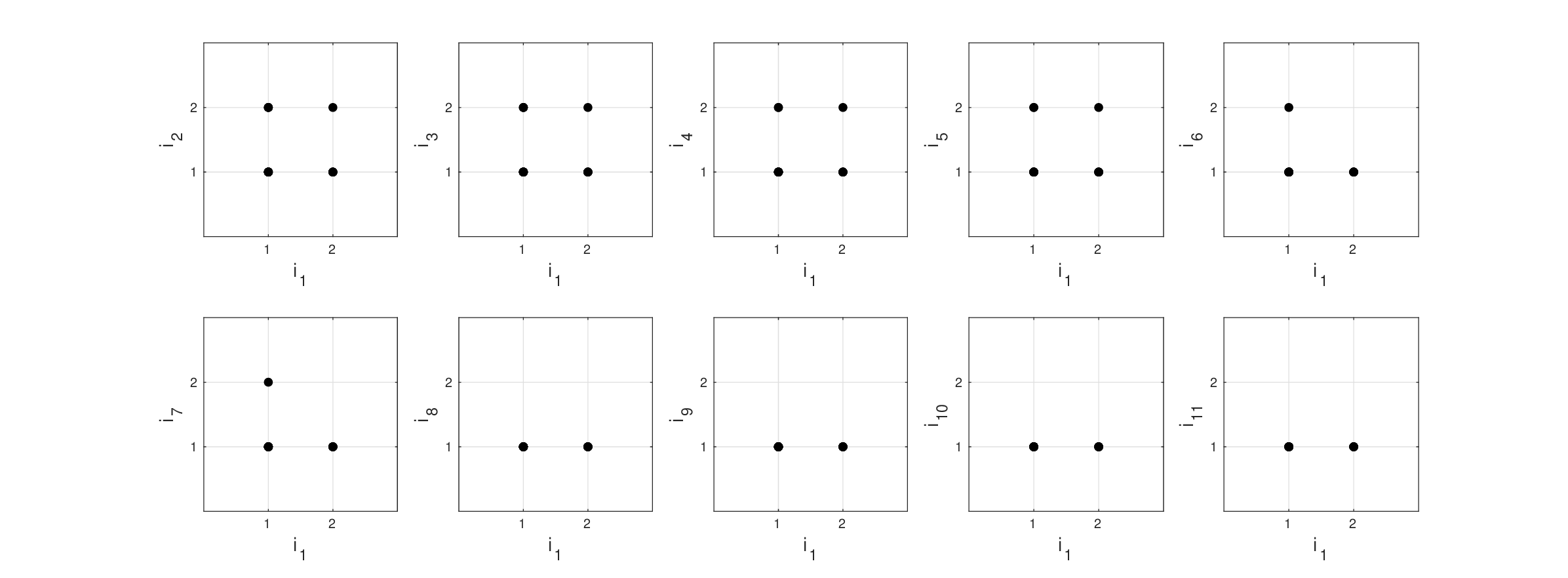}
\caption{Projections of the final multi-index set from SCFE applied to the the problem with Karhunen–Loève expansion for $N=11$.}
\label{fig:final_mid_set}
\end{figure}

\subsection{Second example: Inclusion problem with $\textbf{N=8}$}
We consider an inclusion problem with $N=8$ parameters similar to that in~\cite{guignard2018posteriori}. Within $D$, we identify nine disjoint subdomains $F$ and $\left\{C_n\right\}_{n=1}^8$ depicted in Figure \ref{fig:domain_incl_pb}. 
The diffusion coefficient reads
\begin{align}\label{eq:diffusion_coeff_8d_incl}
a(x, \bm{y}) = a_0(x) + \sum_{n=1}^8 \gamma_n \chi_n y_n \qquad \textrm{with } a_0 \equiv 1.1,
\end{align}
where $\left(\gamma_n\right)_{n=1}^8 = \left(1, 0.8, 0.4, 0.2, 0.1, 0.05, 0.02, 0.01\right)$ are constants used to introduce anisotropy in the problem and $\chi_n$ is the characteristic function of $C_n$, for all $n\in 1,...,8$.
The forcing term reads $f(x) \coloneqq 100 \chi_{F}(x)$, where $\chi_{F}$ is the characteristic function of $F$. 
\begin{figure}
\begin{tikzpicture}[scale=5]
\newcommand\ep{0.0625}

\coordinate (y) at (0,1.2);
\coordinate (x) at (1.2,0);
\draw[axis] (y) node[left]{$x_2$} -- (0,0) --  (x) node[below]{$x_1$};
\foreach \x in {0.2, 0.4, 0.6, 0.8, 1} {
    \draw (\x,0.01) -- (\x,-0.01) node[below] {\x};
}
\foreach \y in {0.2, 0.4, 0.6, 0.8, 1} {
    \draw (0.01, \y) -- (-0.01,\y) node[left] {\y};
}
\filldraw[draw=black,fill=none] (0,0) rectangle (1,1);
\foreach \x in {1,...,3}
    \foreach \y in {1,...,3} 
       \filldraw[draw=black,fill=none] (\x*\ep+\x*0.25-0.25, \y*\ep+\y*0.25-0.25) rectangle (\x*\ep+\x*0.25, \y*\ep+\y*0.25);
\node[] at (0.2, 0.2) {$C_1$};
\node[] at (0.5, 0.2) {$C_2$};
\node[] at (0.8, 0.2) {$C_3$};
\node[] at (0.2, 0.5) {$C_4$};
\node[] at (0.5, 0.5) {$F$};
\node[] at (0.8, 0.5) {$C_5$};
\node[] at (0.2, 0.8) {$C_6$};
\node[] at (0.5, 0.8) {$C_7$};
\node[] at (0.8, 0.8) {$C_8$};
\end{tikzpicture}
\includegraphics[width=0.38\linewidth]{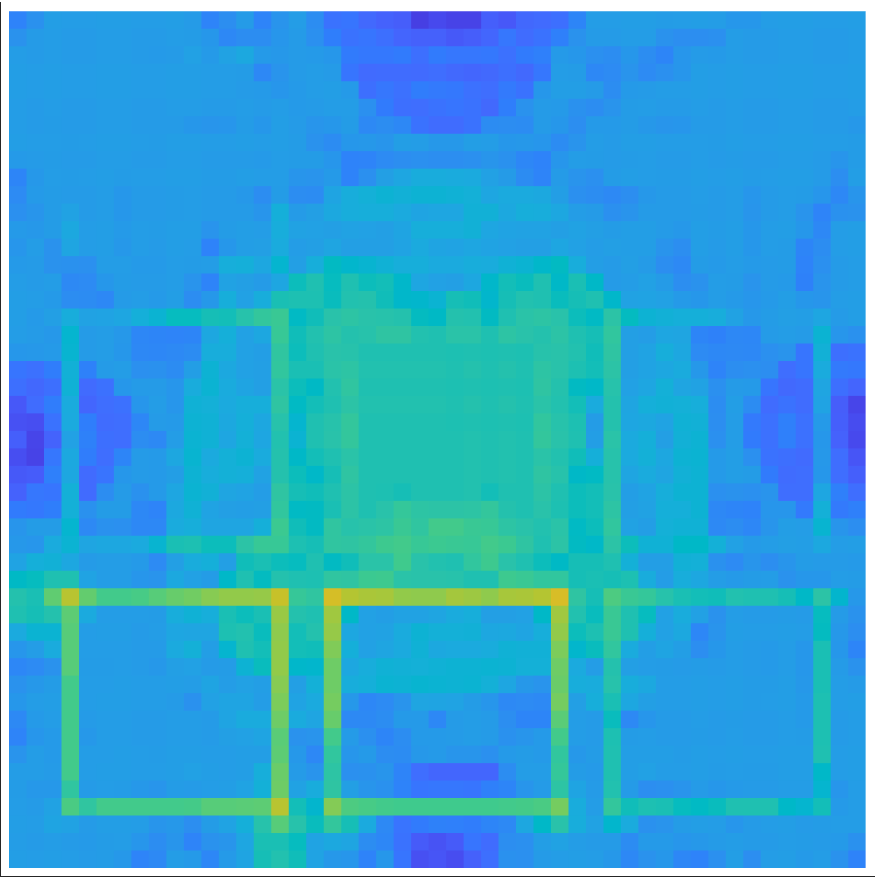}
\includegraphics[width=0.1\linewidth]{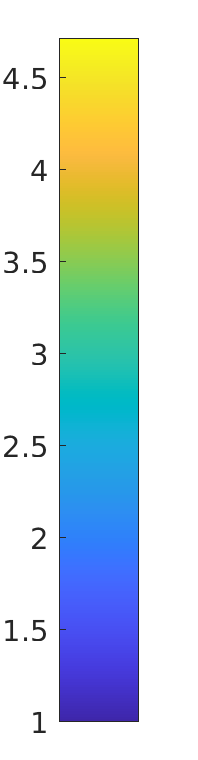}
\caption{Left: Domain for the inclusion problem. Right: Logarithmic density plot of a mesh of SCFE with single adaptive mesh. The colors refer to the number of mesh elements within one pixel of the plot.}
\label{fig:domain_incl_pb}
\end{figure}

In order to highlight the importance of adaptive finite element refinement in space, we present a comparison between the single mesh version of SCFE from Section~\ref{sec:singlemesh}, where the unique mesh is adaptively refined with D\"orfler marking, and an analogous version where only uniform refinement on the whole mesh is allowed.
In Figure~\ref{fig:SCFE_single_mesh} (top left) we report for both algorithms the value of the estimator and reference error. The adaptive version clearly outperforms the one with uniform refinement.
\revision{In Figure~\ref{fig:domain_incl_pb} (right) we show a density plot of a mesh produced by the algorithm with $\approx 2\cdot 10^7$ degrees of freedom.} We see that mesh refinement occurs along the boundary of the inclusions and is more pronounced for the inclusions corresponding to larger anisotropy parameter $\gamma_n$, confirming that the algorithm detects the parametric structure of the problem.
\begin{figure}
\includegraphics[width=0.45\linewidth]{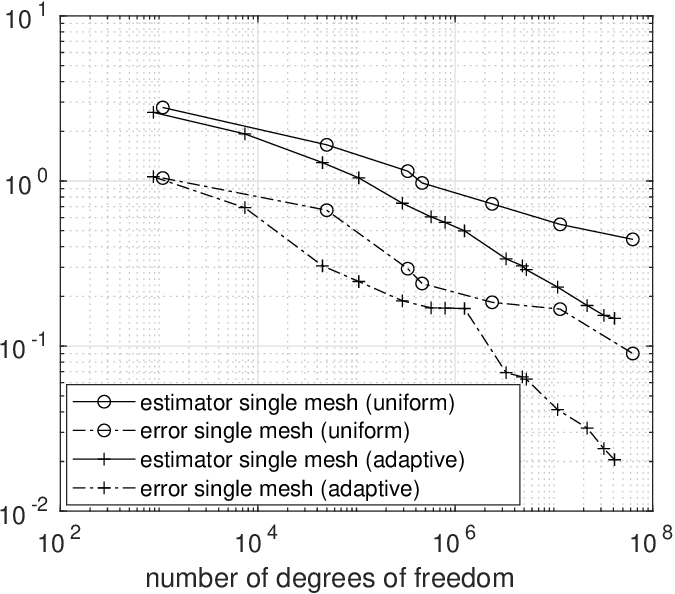} 
\includegraphics[width=0.45\linewidth]{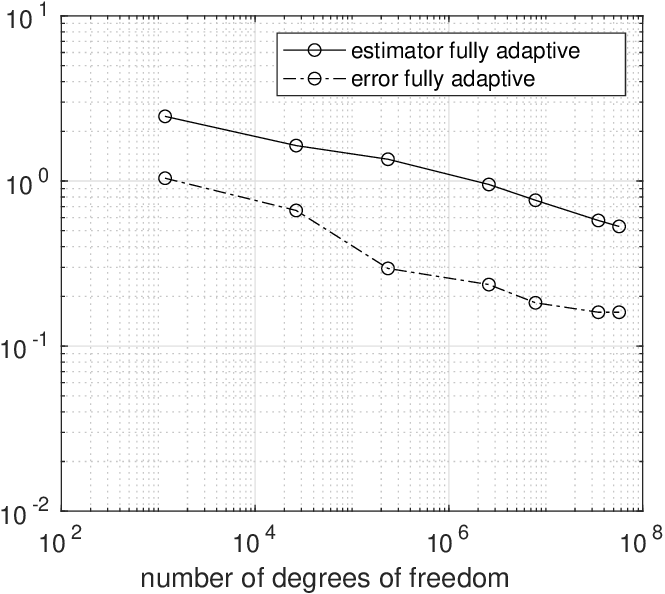}\\
\includegraphics[width=0.45\linewidth]{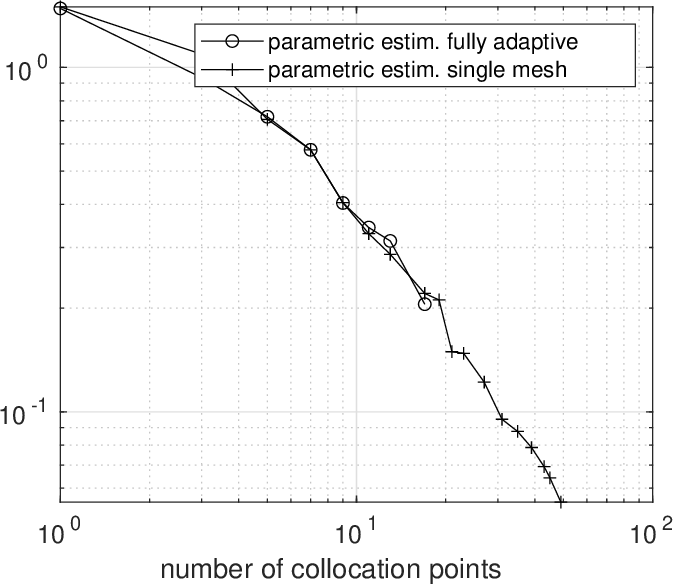}
 \includegraphics[width=0.5\textwidth]{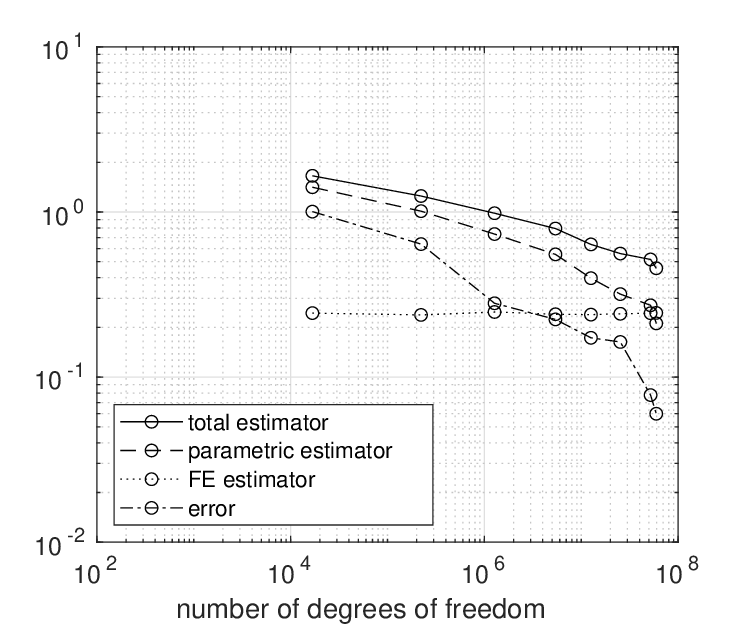}
\caption{Results for SCFE  on the 8D inclusion problem. Top left: Comparison between adaptive and uniform space refinement with the single mesh algorithm from Section~\ref{sec:singlemesh}. Top right: Total estimator and error for the fully adaptive algorithm. Bottom left: Parametric estimator as a function of the number of collocation points, for both the fully adaptive and the single mesh SCFE. Bottom right: the fully adaptive SCFE with fixed value for finite element tolerance ${\rm Tol}$. }
\label{fig:SCFE_single_mesh}
\end{figure}

In Figure \ref{fig:SCFE_single_mesh} (top right) we study the fully adaptive SCFE algorithm and observe a clear performance benefit for the single mesh algorithm from Section~\ref{sec:singlemesh}.
Additional insight is given in the plot on the bottom left of Figure~\ref{fig:SCFE_single_mesh}. Here we show the value of the parametric estimator with respect to the number of collocation point for both the fully adaptive and single mesh versions of SCFE. This shows that the fully adaptive SCFE algorithm seems to overrefine the finite-element meshes. {We suspect that this is due to the fact that in the derivation of $\eta_{{\rm FE},I}$ in Section~\ref{sec:def_adaptive_algo}, one is required to use the triangle inequality and thus sacrifices local information of the sparse grid interpolant. This is not necessary in the single-mesh estimator from Section~\ref{sec:singlemesh}.}

Finally, we ran Algorithm~\ref{algo:SCFE} with fixed tolerance ${\rm Tol}$ (not depending on the parametric estimator) and plot the results in Figure~\ref{fig:SCFE_single_mesh} (bottom right). We observe that the results are very much comparable to the standard fully adaptive SCFE algorithm, except for significant over refinement in the early stages of the computation (see the flat line of the finite element error estimator satisfying the tolerance). In terms of computational effort, the algorithms are nearly identical, as the same spatial refinements are performed, only at different stages of the algorithm.

We also tested the algorithms with respect to the $L^2(\Gamma)$-norm instead of the $L^\infty(\Gamma)$-norm. The necessary changes in the estimators are straightforward, essentially we replace the search for the maximum by a Monte Carlo quadrature. The theoretical results of this manuscript all hold verbatim for the $L^2(\Gamma)$-norm. Figure~\ref{fig:SCFE_8D_incl_L2} shows the results. Again the single mesh algorithm outperforms the fully adaptive algorithm.

\subsubsection{Distribution of computational cost}
\revision{In Figure~\ref{fig:SCFE_8D_incl_L2} (right-hand side) we compare the total number of degrees of freedom and collocation points achieved by the three methods, i.e. the single adaptive mesh algorithm from Section~\ref{sec:singlemesh}, adaptivity in the parameter space but uniform refinement in the spatial domain, and the fully adaptive SCFE algorithm (Algorithm~\ref{algo:SCFE}). 
The adaptive strategy with a single adaptive mesh  performs parametric refinement more often than the other two, leading to a higher number of collocation points and lower average number of degrees of freedom per collocation point. 
In Figure~\ref{fig:meshes_MAM} (compare also Figure~\ref{fig:domain_incl_pb}) we provide logarithmic density plots of the meshes produced by the multiple adaptive mesh algorithm (with $\approx 2\cdot 10^7$ degrees of freedom). We observe that the mesh corresponding to a collocation point is locally refined along the edges of the corresponding inclusion. Furthermore, the intensity of the refinement around a certain inclusion is related to the constant $\gamma_n$ of the diffusion coefficient, confirming that the numerical methods detects the anisotropy of the problem.}


\section{{Conclusion}}
We analyze the adaptive stochastic collocation algorithm from~\cite{guignard2018posteriori} and prove convergence of several different versions of the algorithm:
\begin{itemize}
 \item Convergence of the parametric enrichment algorithm without finite element refinement (Section~\ref{sec:paramconv})
 \item Convergence of the fully adaptive algorithm (Algorithm~\ref{algo:SCFE}) even with optimal convergence of the finite element loop (Theorem~\ref{thm:optimality})
 \item Convergence of a single-mesh variant of Algorithm~\ref{algo:SCFE} (Section~\ref{sec:singlemesh}) proposed in~\cite{guignard2018posteriori}.
\end{itemize}
The numerical examples clearly show the superiority of spatial adaptive refinement combined with parametric enrichment over pure parametric enrichment algorithms.
While the theoretical results are strongest for the fully adaptive algorithm (linear convergence in Proposition~\ref{prop:linconv} for Algorithm~\ref{algo:SCFE}) the single mesh algorithm from Section~\ref{sec:singlemesh} seems to be more efficient. This is underlined  by the numerical experiments in the previous section, which clearly show an advantage of the single mesh version over the fully adaptive version. Based on the theoretical results from Theorem~\ref{thm:optimality} and the experiments, we come to the conclusion that the finite element error estimator of Algorithm~\ref{algo:SCFE} severely over-estimates the total error and hence leads to over-refinement of the finite element meshes. This does not seem to happen for the single-mesh error estimator. We suspect that the application of the triangle inequality in the derivation in Section~\ref{sec:others} is mainly responsible for this over-estimation and further research is required to see whether this can be avoided. 
 \begin{figure}
\includegraphics[width=0.45\linewidth]{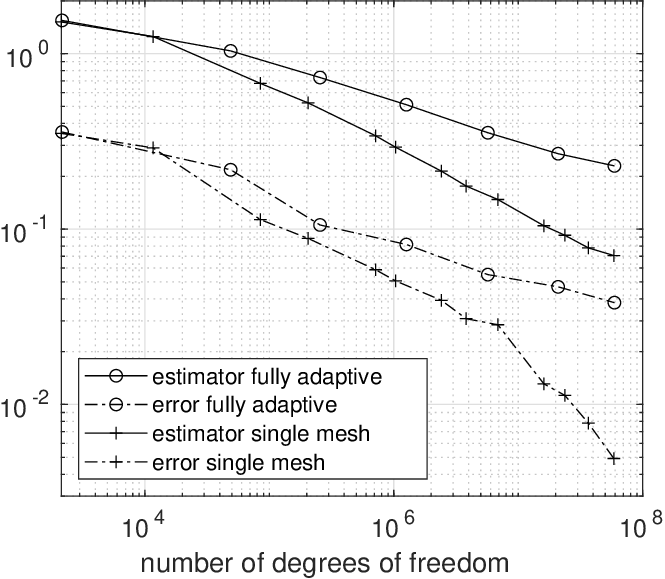} 
\includegraphics[width=0.47\linewidth]{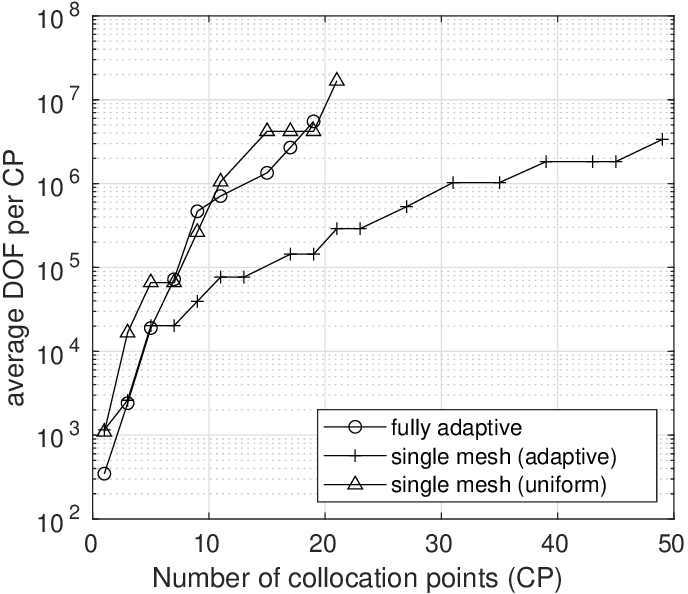}
\caption{Left: Numerical results on the 8D inclusion problem in the $L^2(\Gamma)$-norm. Right: Total estimator and error for SCFE (fully adaptive). Right: \revision{Average number of degrees of freedom (DOF) per collocation point (CP) plotted versus the number of collocation points. Each line corresponds to one of the three proposed algorithms. 
Each marker corresponds to one step of the adaptively refined discrete solution.}}
\label{fig:SCFE_8D_incl_L2}
\end{figure}
\begin{figure}
\begin{tabular}{cccc}
\includegraphics[width=0.22\linewidth]{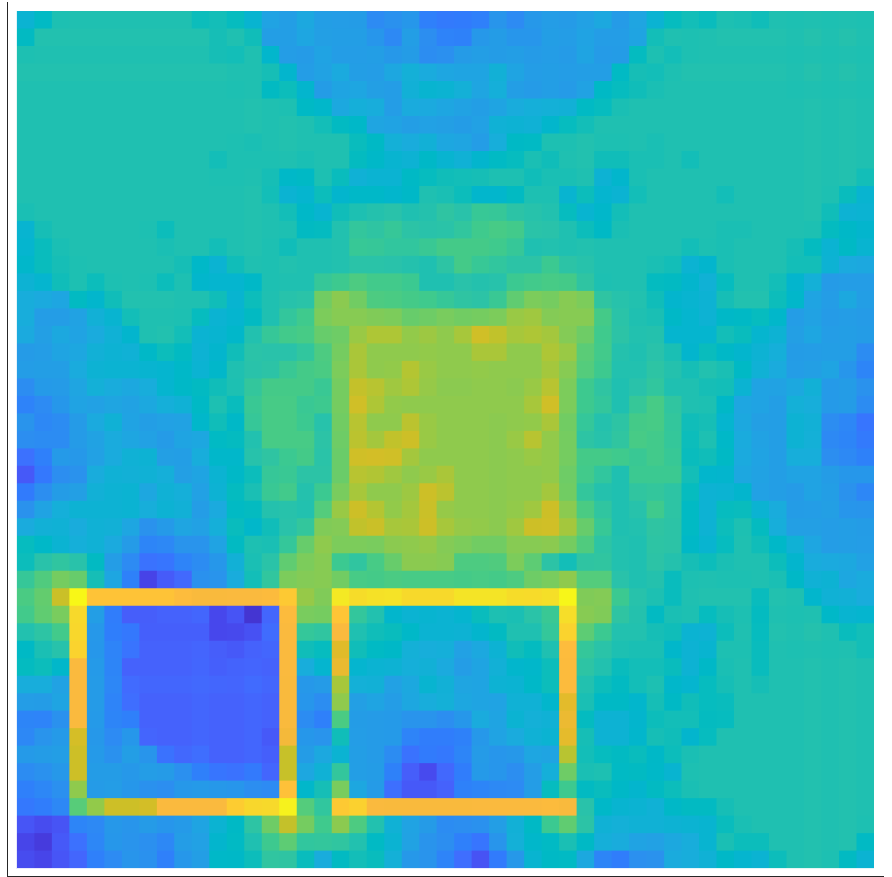} & \includegraphics[width=0.22\linewidth]{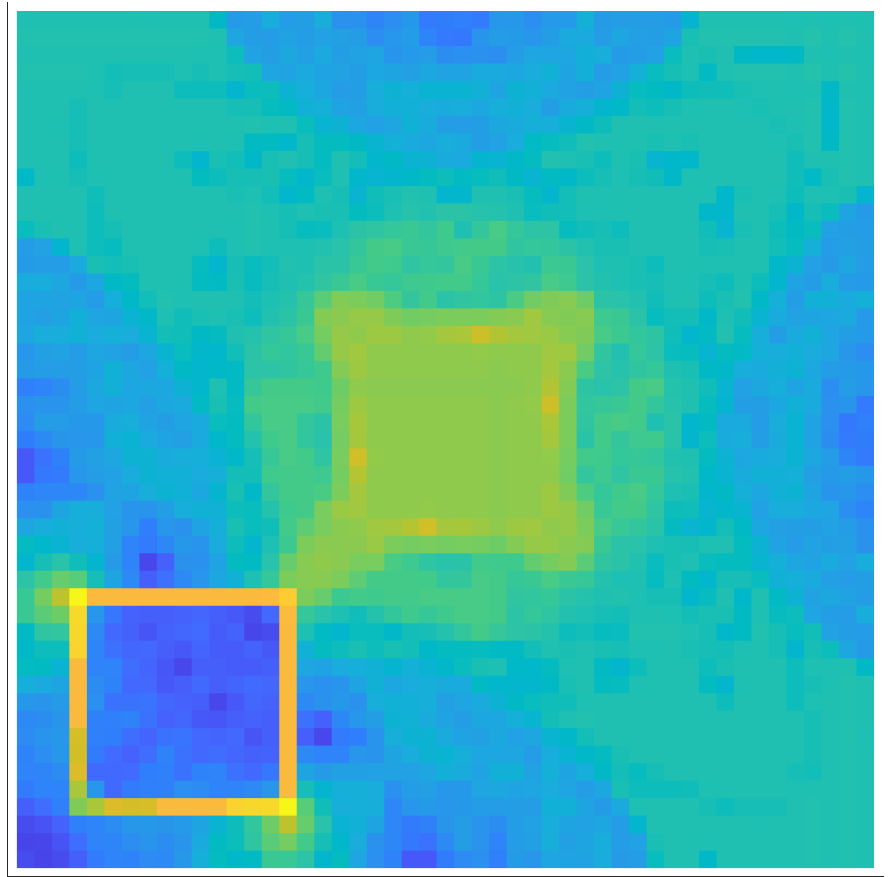} & \includegraphics[width=0.22\linewidth]{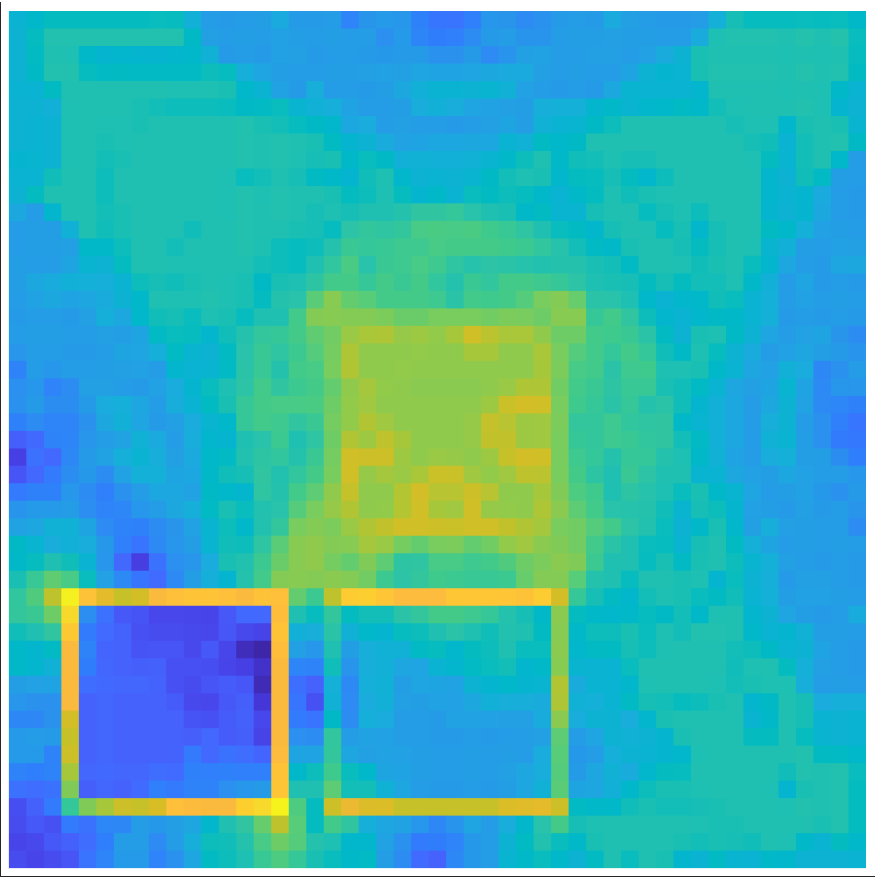} & \includegraphics[width=0.22\linewidth]{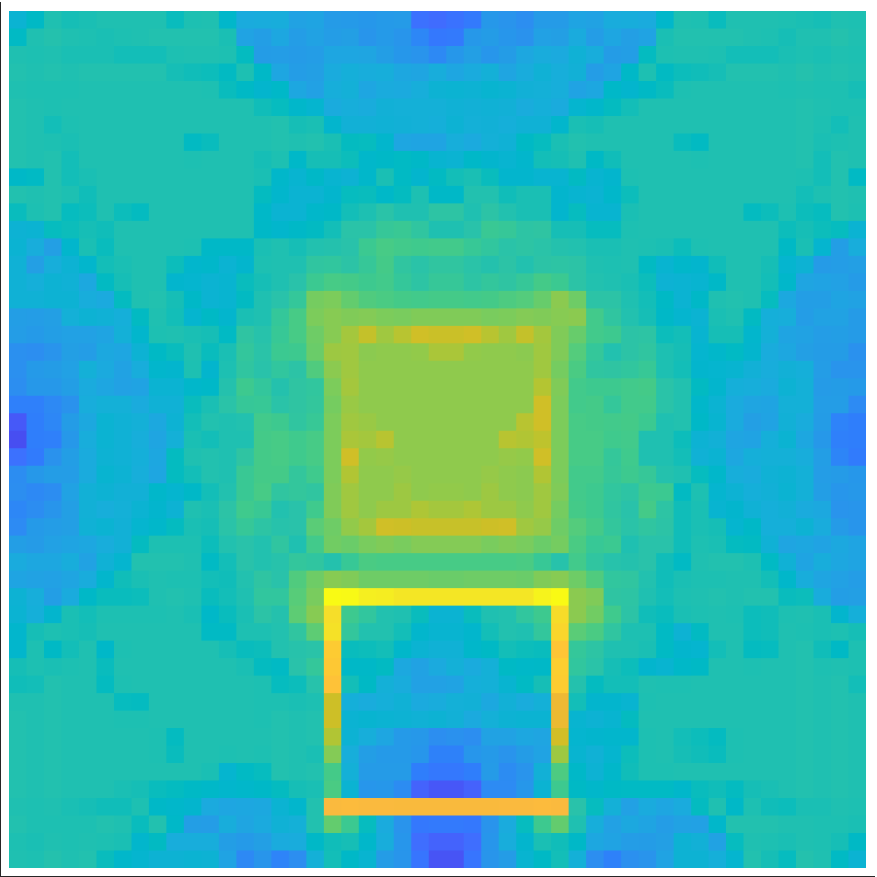} \\
$\bm{y} = (-1,-1)$ & $\bm{y} = (-1)$ & $\bm{y} = (-1,1)$ & $\bm{y} = (0,-1)$\\
\includegraphics[width=0.22\linewidth]{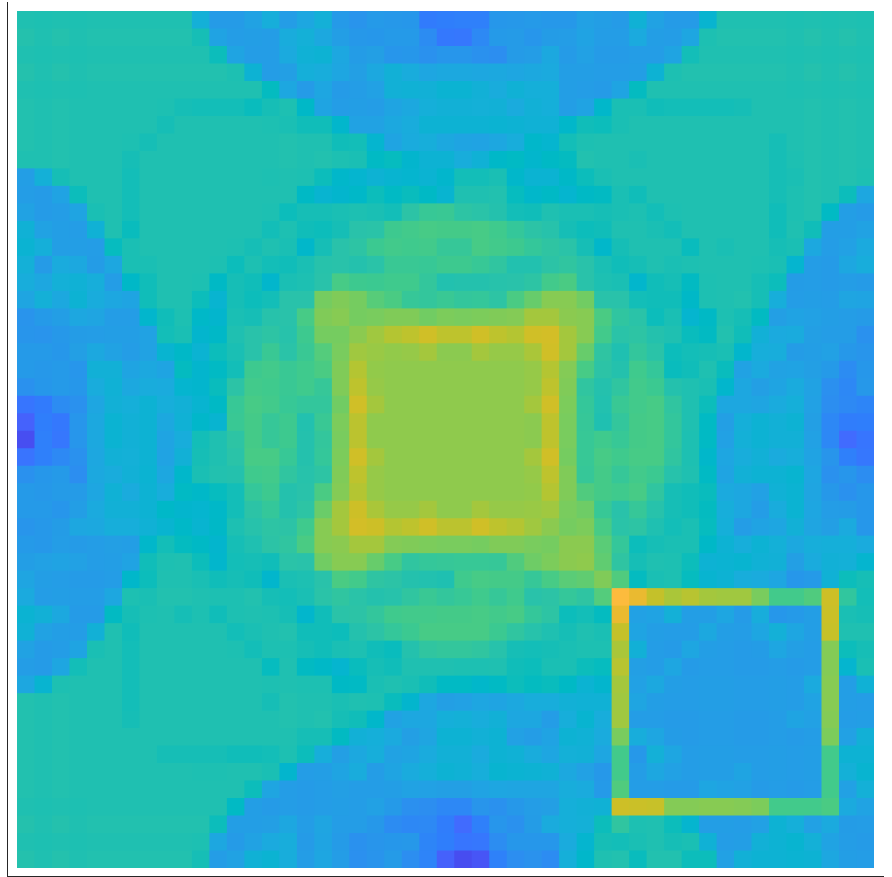} & \includegraphics[width=0.22\linewidth]{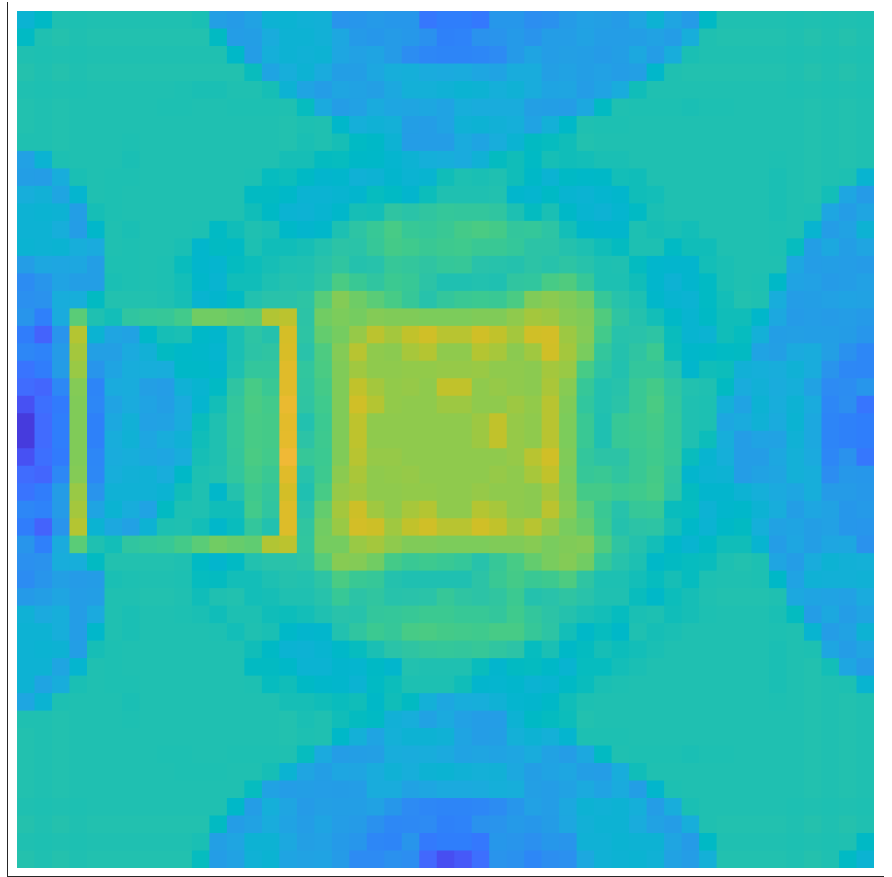} & \includegraphics[width=0.22\linewidth]{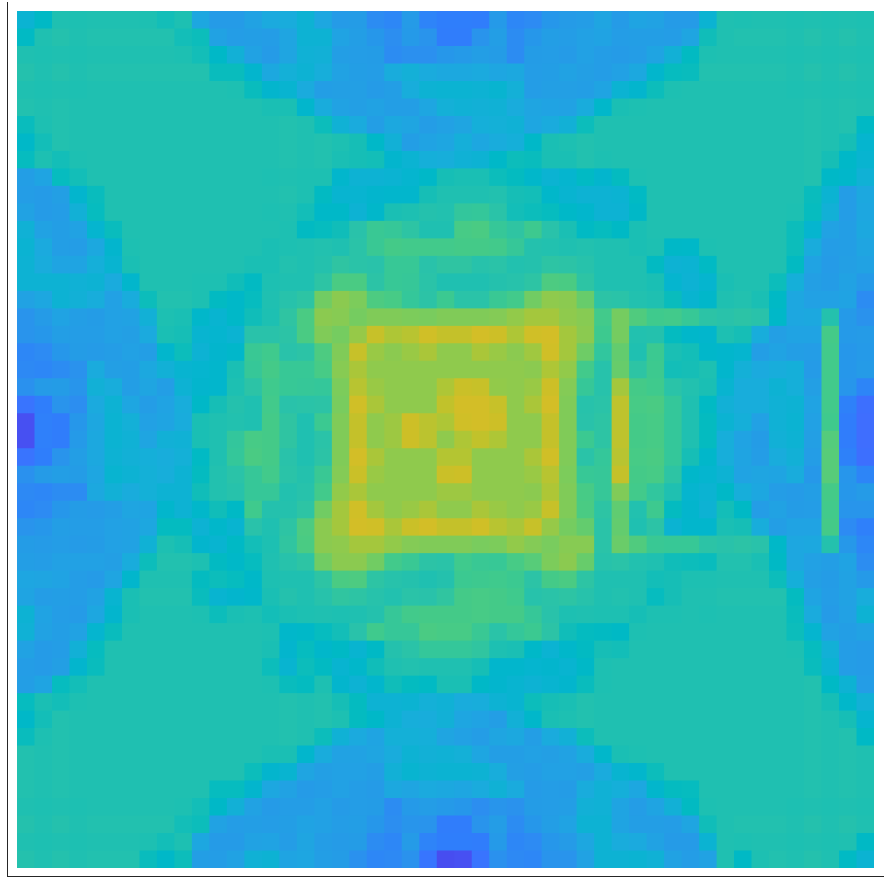} & \includegraphics[width=0.22\linewidth]{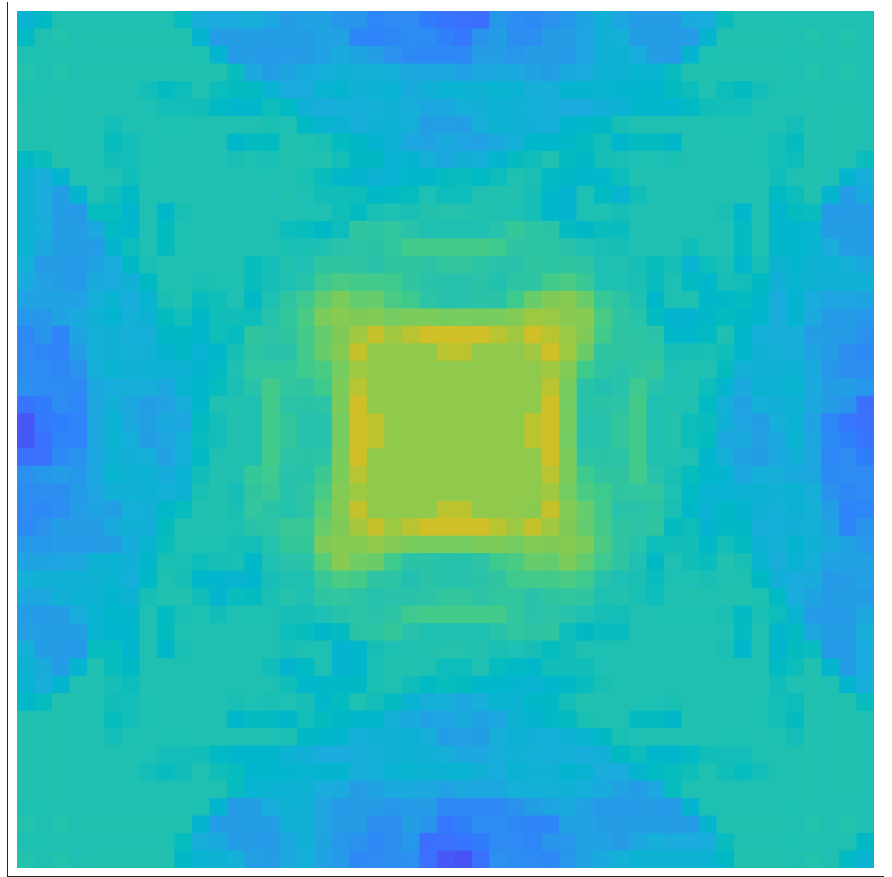}\\
$\bm{y} = (0,0,-1)$ & $\bm{y} = (0,0,0,-1)$ & $\bm{y} = (0,0,0,0,-1)$ & $\bm{y} = \bm{0}$\\
\includegraphics[width=0.22\linewidth]{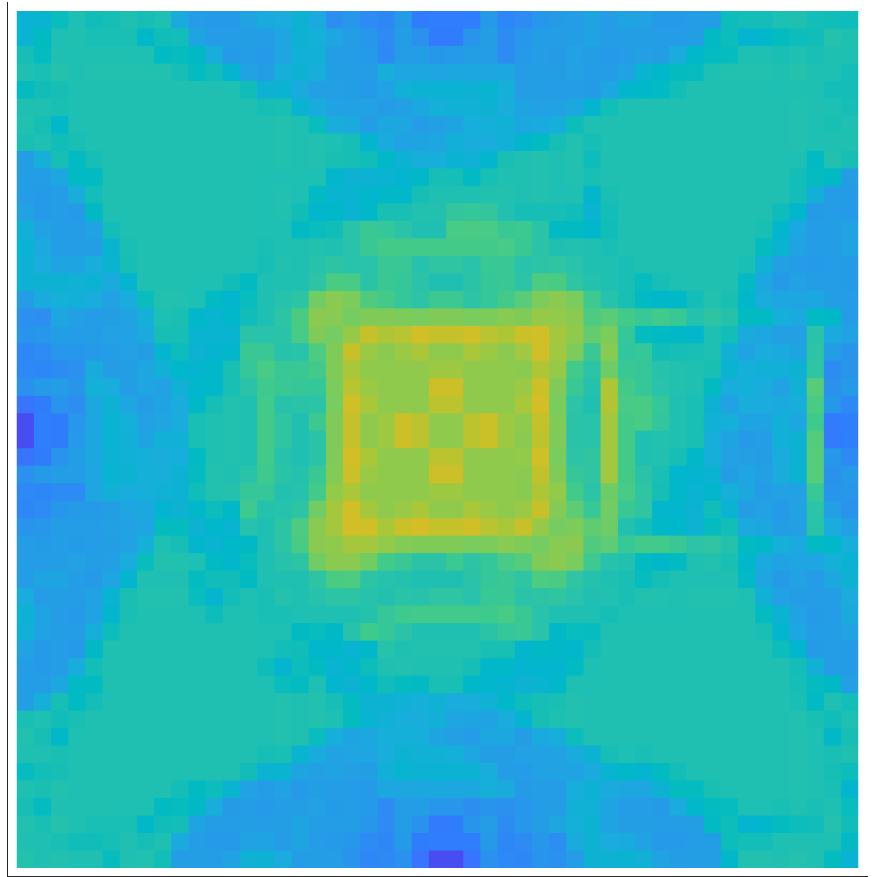} & \includegraphics[width=0.22\linewidth]{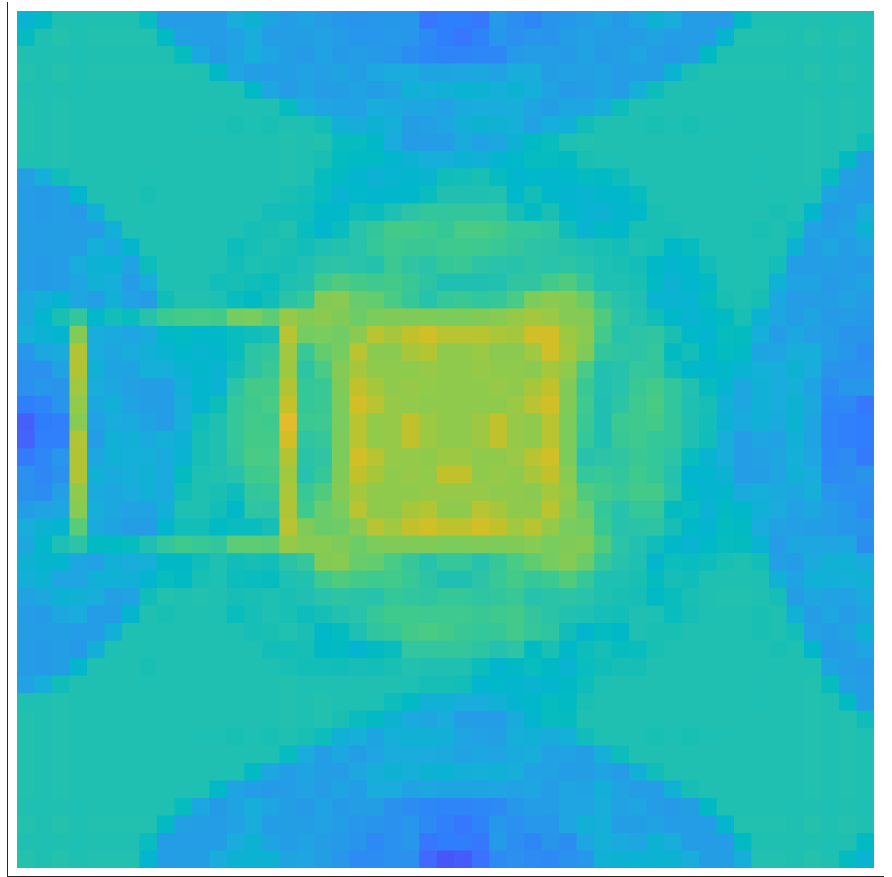} & \includegraphics[width=0.22\linewidth]{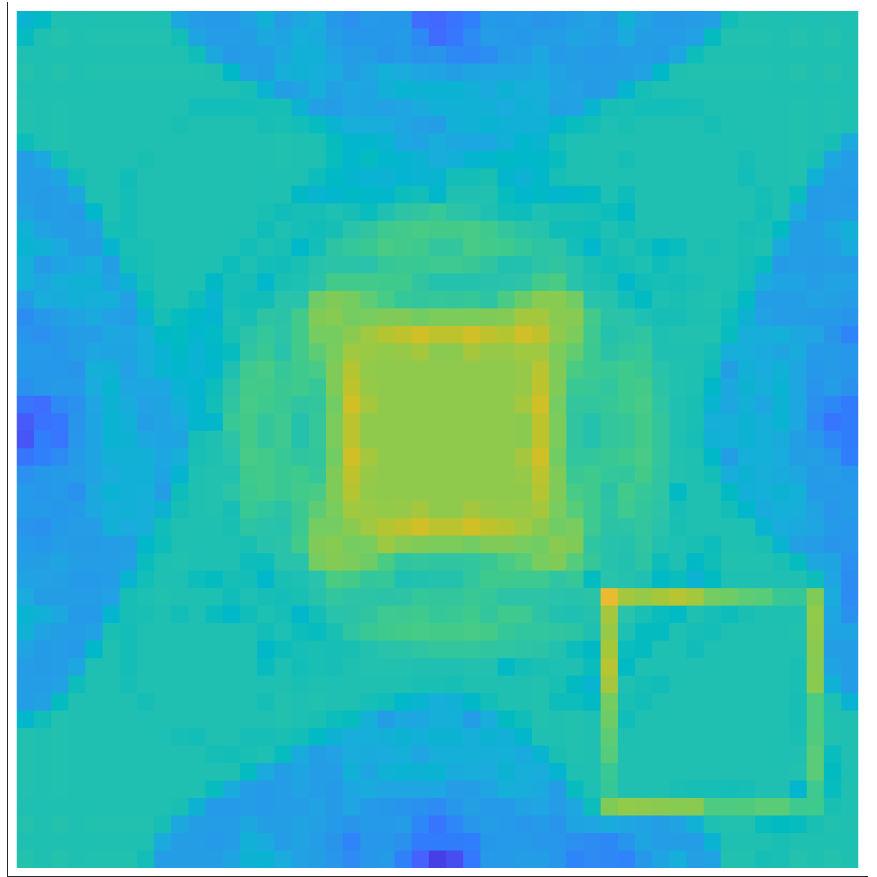} & \includegraphics[width=0.22\linewidth]{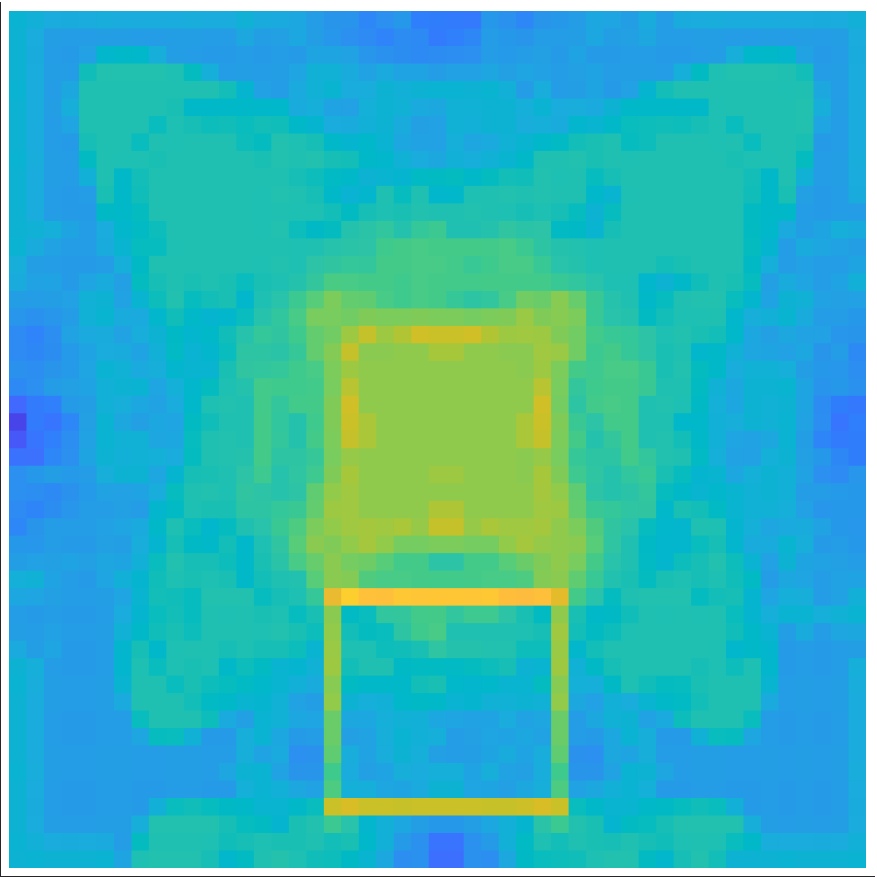}\\
$\bm{y} = (0,0,0,0,1)$ & $\bm{y} = (0,0,0,1)$ & $\bm{y} = (0,0,1)$ & $\bm{y} = (0,1)$\\
\includegraphics[width=0.22\linewidth]{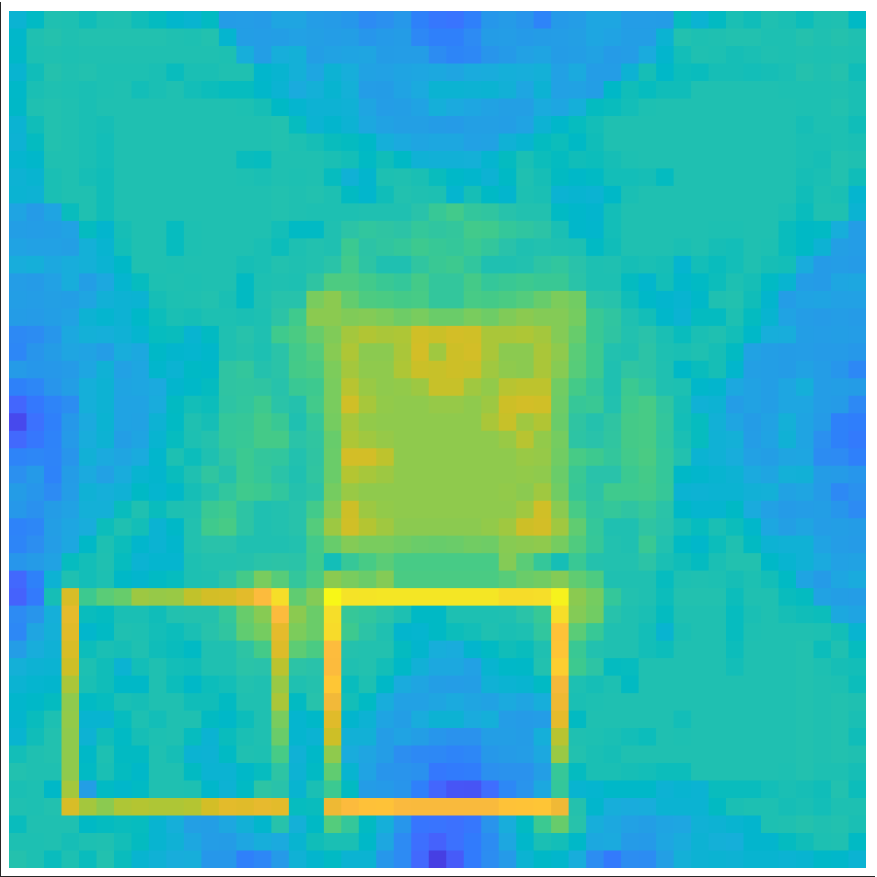} & \includegraphics[width=0.22\linewidth]{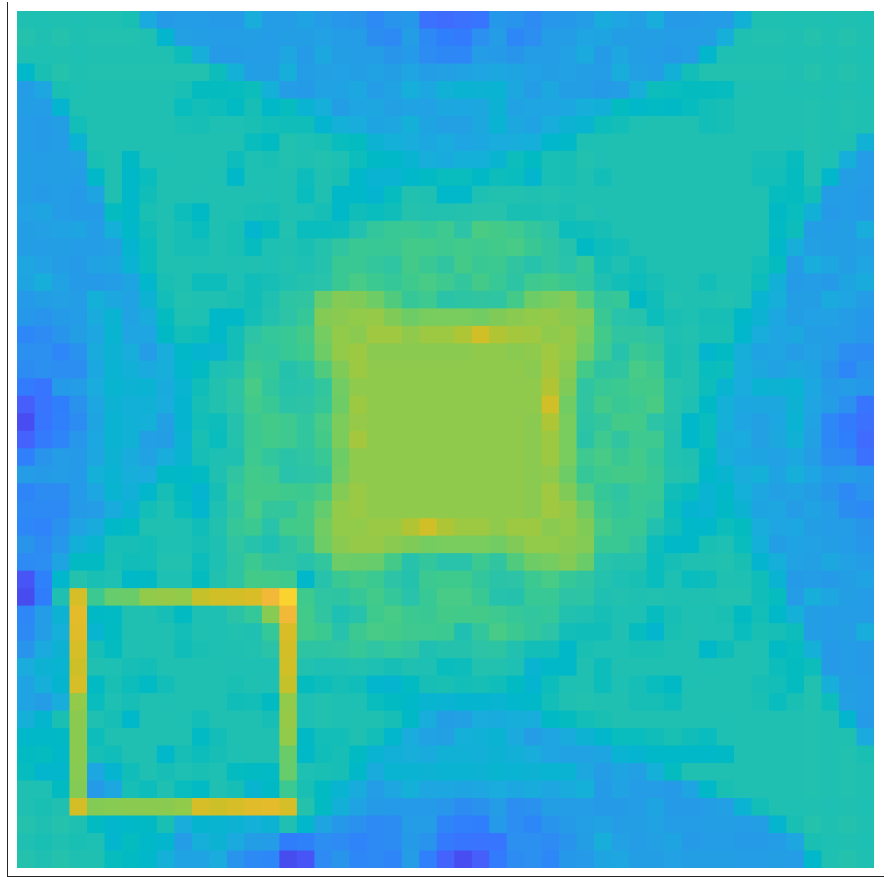} & \includegraphics[width=0.22\linewidth]{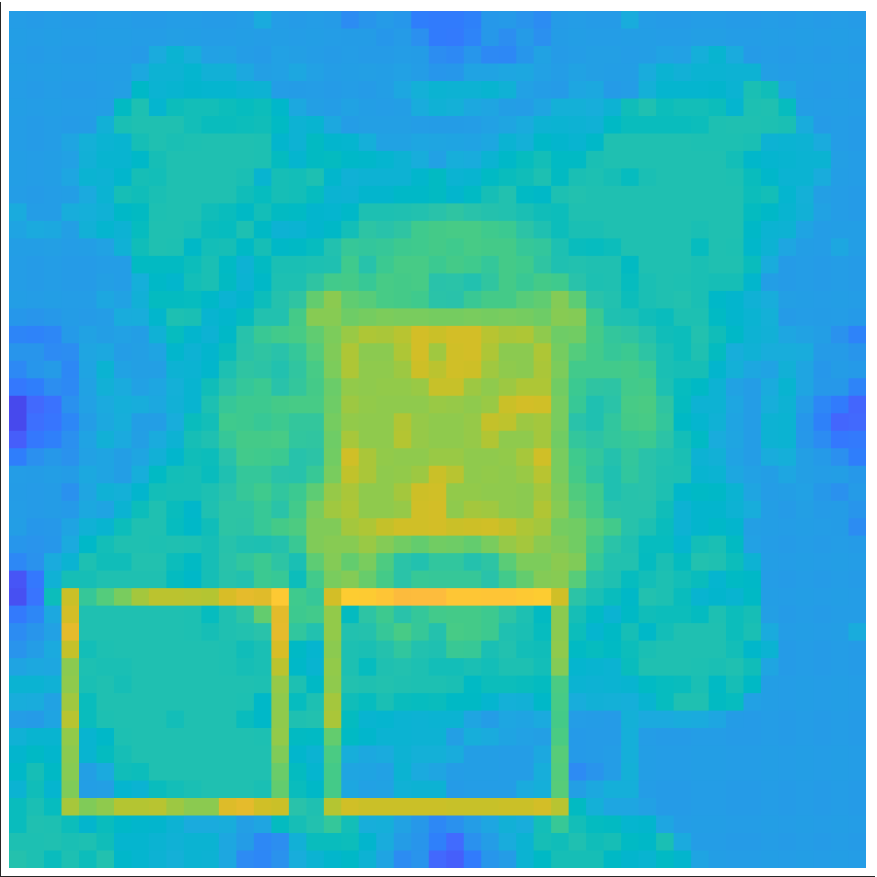}\\
$\bm{y} = (1,-1)$ & $\bm{y} = (1)$ & $\bm{y} = (1,1)$\\
\multicolumn{4}{c}{\includegraphics[width=0.7\linewidth]{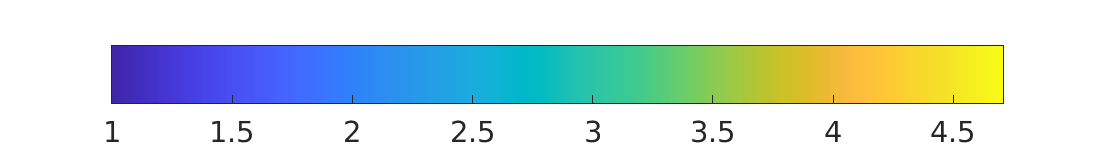}}
\end{tabular}
\caption{\revision{Density plot of the meshes produced by the fully adaptive SCFE algorithm. The corresponding collocation point is indicated below, ignoring the trailing components equal to zero.
The color-bar at the bottom indicates the base-10 logarithm of the density of elements.}}
\label{fig:meshes_MAM}
\end{figure}

\bibliographystyle{plain} 
\bibliography{literature}

\begin{thebibliography}{10}

\bibitem{ap1}
I.~Babu{\v{s}}ka, B.~Andersson, P.~J. Smith, and K.~Levin.
\newblock Damage analysis of fiber composites. {I}. {S}tatistical analysis on fiber scale.
\newblock {\em Comput. Methods Appl. Mech. Engrg.}, 172(1-4):27--77, 1999.

\bibitem{babuvska2007stochastic}
Ivo Babu{\v{s}}ka, Fabio Nobile, and Raul Tempone.
\newblock A stochastic collocation method for elliptic partial differential equations with random input data.
\newblock {\em SIAM Journal on Numerical Analysis}, 45(3):1005--1034, 2007.

\bibitem{back.nobile.eal:comparison}
J.~B\"ack, F.~Nobile, L.~Tamellini, and R.~Tempone.
\newblock Stochastic spectral {G}alerkin and collocation methods for {PDE}s with random coefficients: a numerical comparison.
\newblock In J.S. Hesthaven and E.M. Ronquist, editors, {\em Spectral and High Order Methods for Partial Differential Equations}, volume~76 of {\em Lecture Notes in Computational Science and Engineering}, pages 43--62. Springer, 2011.
\newblock Selected papers from the ICOSAHOM '09 conference, June 22-26, Trondheim, Norway.

\bibitem{barthelmann2000high}
Volker Barthelmann, Erich Novak, and Klaus Ritter.
\newblock High dimensional polynomial interpolation on sparse grids.
\newblock {\em Advances in Computational Mathematics}, 12(4):273--288, 2000.

\bibitem{scc4}
Joakim Beck, Raul Tempone, Fabio Nobile, and Lorenzo Tamellini.
\newblock On the optimal polynomial approximation of stochastic {PDE}s by {G}alerkin and collocation methods.
\newblock {\em Math. Models Methods Appl. Sci.}, 22(9):1250023, 33, 2012.

\bibitem{bespalov2019convergence}
Alex Bespalov, Dirk Praetorius, Leonardo Rocchi, and Michele Ruggeri.
\newblock Convergence of adaptive stochastic galerkin fem.
\newblock {\em SIAM Journal on Numerical Analysis}, 57(5):2359--2382, 2019.

\bibitem{axioms}
C.~Carstensen, M.~Feischl, M.~Page, and D.~Praetorius.
\newblock Axioms of adaptivity.
\newblock {\em Comput. Math. Appl.}, 67(6):1195--1253, 2014.

\bibitem{Cascon_2008}
J~Manuel Cascon, Christian Kreuzer, Ricardo~H Nochetto, and Kunibert~G Siebert.
\newblock Quasi-optimal convergence rate for an adaptive finite element method.
\newblock {\em SIAM Journal on Numerical Analysis}, 46(5):2524--2550, 2008.

\bibitem{chkifa2014high}
Abdellah Chkifa, Albert Cohen, and Christoph Schwab.
\newblock High-dimensional adaptive sparse polynomial interpolation and applications to parametric pdes.
\newblock {\em Foundations of Computational Mathematics}, 14(4):601--633, 2014.

\bibitem{aa1}
Albert Cohen and Ronald DeVore.
\newblock Approximation of high-dimensional parametric {PDE}s.
\newblock {\em Acta Numer.}, 24:1--159, 2015.

\bibitem{aa2}
Albert Cohen, Ronald DeVore, and Christoph Schwab.
\newblock Convergence rates of best {$N$}-term {G}alerkin approximations for a class of elliptic s{PDE}s.
\newblock {\em Found. Comput. Math.}, 10(6):615--646, 2010.

\bibitem{qmcsoa}
J.~Dick, F.~Y. Kuo, Q.~T. Le~Gia, D.~Nuyens, and C.~Schwab.
\newblock Higher order {QMC} {P}etrov-{G}alerkin discretization for affine parametric operator equations with random field inputs.
\newblock {\em SIAM J. Numer. Anal.}, 52(6):2676--2702, 2014.

\bibitem{multiindex}
Josef Dick, Michael Feischl, and Christoph Schwab.
\newblock Improved efficiency of a multi-index {FEM} for computational uncertainty quantification.
\newblock {\em SIAM J. Numer. Anal.}, 57(4):1744--1769, 2019.

\bibitem{ml1}
Josef Dick, Robert~N. Gantner, Quoc~T. Le~Gia, and Christoph Schwab.
\newblock Multilevel higher-order quasi-{M}onte {C}arlo {B}ayesian estimation.
\newblock {\em Math. Models Methods Appl. Sci.}, 27(5):953--995, 2017.

\bibitem{hoqmc}
Josef Dick, Frances~Y. Kuo, Quoc~T. Le~Gia, Dirk Nuyens, and Christoph Schwab.
\newblock Higher order {QMC} {P}etrov-{G}alerkin discretization for affine parametric operator equations with random field inputs.
\newblock {\em SIAM J. Numer. Anal.}, 52(6):2676--2702, 2014.

\bibitem{ml2}
Josef Dick, Frances~Y. Kuo, Quoc~T. Le~Gia, and Christoph Schwab.
\newblock Multilevel higher order {QMC} {P}etrov-{G}alerkin discretization for affine parametric operator equations.
\newblock {\em SIAM J. Numer. Anal.}, 54(4):2541--2568, 2016.

\bibitem{Dolgov15}
Sergey Dolgov, Boris~N. Khoromskij, Alexander Litvinenko, and Hermann~G. Matthies.
\newblock Polynomial chaos expansion of random coefficients and the solution of stochastic partial differential equations in the tensor train format.
\newblock {\em SIAM/ASA Journal on Uncertainty Quantification}, 3(1):1109--1135, 2015.

\bibitem{dzjadyk1983asymptotics}
VK~Dzjadyk and VV~Ivanov.
\newblock On asymptotics and estimates for the uniform norms of the lagrange interpolation polynomials corresponding to the chebyshev nodal points.
\newblock {\em Analysis Mathematica}, 9(2):85--97, 1983.

\bibitem{eigel}
Martin Eigel, Oliver Ernst, Björn Sprungk, and Lorenzo Tamellini.
\newblock On the convergence of adaptive stochastic collocation for elliptic partial differential equations with affine diffusion.
\newblock {\em arXiv:2008.07186}, 2020.

\bibitem{schwabadapt}
Martin Eigel, Claude~Jeffrey Gittelson, Christoph Schwab, and Elmar Zander.
\newblock A convergent adaptive stochastic {G}alerkin finite element method with quasi-optimal spatial meshes.
\newblock {\em ESAIM Math. Model. Numer. Anal.}, 49(5):1367--1398, 2015.

\bibitem{ap2}
I.~Elishakoff, editor.
\newblock {\em Whys and hows in uncertainty modelling}, volume 388 of {\em CISM Courses and Lectures}.
\newblock Springer-Verlag, Vienna, 1999.
\newblock Probability, fuzziness and anti-optimization.

\bibitem{Espig14}
Mike Espig, Wolfgang Hackbusch, Alexander Litvinenko, Hermann~G. Matthies, and Philipp Wähnert.
\newblock Efficient low-rank approximation of the stochastic {G}alerkin matrix in tensor formats.
\newblock {\em Computers \& Mathematics with Applications}, 67(4):818 -- 829, 2014.
\newblock High-order Finite Element Approximation for Partial Differential Equations.

\bibitem{h2randfield}
Michael Feischl, Frances~Y. Kuo, and Ian~H. Sloan.
\newblock Fast random field generation with {$H$}-matrices.
\newblock {\em Numer. Math.}, 140(3):639--676, 2018.

\bibitem{sgc2}
Philipp Frauenfelder, Christoph Schwab, and Radu~Alexandru Todor.
\newblock Finite elements for elliptic problems with stochastic coefficients.
\newblock {\em Comput. Methods Appl. Mech. Engrg.}, 194(2-5):205--228, 2005.

\bibitem{mars}
Jerome~H. Friedman.
\newblock Multivariate adaptive regression splines.
\newblock {\em Ann. Statist.}, 19(1):1--141, 1991.
\newblock With discussion and a rejoinder by the author.

\bibitem{funken2011efficient}
Stefan Funken, Dirk Praetorius, and Philipp Wissgott.
\newblock Efficient implementation of adaptive p1-fem in matlab.
\newblock {\em Computational Methods in Applied Mathematics}, 11(4):460--490, 2011.

\bibitem{Gerstner03}
Thomas Gerstner and Michael Griebel.
\newblock Dimension-adaptive tensor-product quadrature.
\newblock {\em Computing}, 71(1):65--87, September 2003.

\bibitem{sgc3}
Roger Ghanem.
\newblock Ingredients for a general purpose stochastic finite elements implementation.
\newblock {\em Comput. Methods Appl. Mech. Engrg.}, 168(1-4):19--34, 1999.

\bibitem{sgc1}
Roger~G. Ghanem and Pol~D. Spanos.
\newblock {\em Stochastic finite elements: a spectral approach}.
\newblock Springer-Verlag, New York, 1991.

\bibitem{Giraldi14}
Loïc Giraldi, Alexander Litvinenko, Dishi Liu, Hermann~G. Matthies, and Anthony Nouy.
\newblock To be or not to be intrusive? the solution of parametric and stochastic equations---the "plain vanilla" galerkin case.
\newblock {\em SIAM Journal on Scientific Computing}, 36(6):A2720--A2744, 2014.

\bibitem{flow}
I.G. Graham, F.Y. Kuo, D.~Nuyens, R.~Scheichl, and I.H. Sloan.
\newblock Quasi-{M}onte {C}arlo methods for elliptic {PDEs} with random coefficients and applications.
\newblock {\em Journal of Computational Physics}, 230(10):3668 -- 3694, 2011.

\bibitem{guignard2018posteriori}
Diane Guignard and Fabio Nobile.
\newblock A posteriori error estimation for the stochastic collocation finite element method.
\newblock {\em SIAM Journal on Numerical Analysis}, 56(5):3121--3143, 2018.

\bibitem{additive}
T.~J. Hastie and R.~J. Tibshirani.
\newblock {\em Generalized additive models}, volume~43 of {\em Monographs on Statistics and Applied Probability}.
\newblock Chapman and Hall, Ltd., London, 1990.

\bibitem{schwabrand}
Lukas Herrmann, Kristin Kirchner, and Christoph Schwab.
\newblock Multilevel approximation of {G}aussian random fields: fast simulation.
\newblock {\em Math. Models Methods Appl. Sci.}, 30(1):181--223, 2020.

\bibitem{kolmogorov}
A.~N. Kolmogorov.
\newblock On the representation of continuous functions of several variables by superpositions of continuous functions of a smaller number of variables.
\newblock {\em Amer. Math. Soc. Transl. (2)}, 17:369--373, 1961.

\bibitem{scheichl}
J.~Lang, R.~Scheichl, and D.~Silvester.
\newblock A fully adaptive multilevel stochastic collocation strategy for solving elliptic {PDE}s with random data.
\newblock {\em J. Comput. Phys.}, 419:109692, 17, 2020.

\bibitem{scc2}
Lionel Mathelin, M.~Yousuff Hussaini, and Thomas~A. Zang.
\newblock Stochastic approaches to uncertainty quantification in {CFD} simulations.
\newblock {\em Numer. Algorithms}, 38(1-3):209--236, 2005.

\bibitem{sgc4}
Hermann~G. Matthies and Andreas Keese.
\newblock Galerkin methods for linear and nonlinear elliptic stochastic partial differential equations.
\newblock {\em Comput. Methods Appl. Mech. Engrg.}, 194(12-16):1295--1331, 2005.

\bibitem{scc5}
F.~Nobile, R.~Tempone, and C.~G. Webster.
\newblock An anisotropic sparse grid stochastic collocation method for partial differential equations with random input data.
\newblock {\em SIAM J. Numer. Anal.}, 46(5):2411--2442, 2008.

\bibitem{nobile2008sparse}
Fabio Nobile, Ra{\'u}l Tempone, and Clayton~G Webster.
\newblock A sparse grid stochastic collocation method for partial differential equations with random input data.
\newblock {\em SIAM Journal on Numerical Analysis}, 46(5):2309--2345, 2008.

\bibitem{carl}
Carl-Martin Pfeiler and Dirk Praetorius.
\newblock D\"{o}rfler marking with minimal cardinality is a linear complexity problem.
\newblock {\em Math. Comp.}, 89(326):2735--2752, 2020.

\bibitem{schwab2006karhunen}
Christoph Schwab and Radu~Alexandru Todor.
\newblock Karhunen--{L}o{\`e}ve approximation of random fields by generalized fast multipole methods.
\newblock {\em Journal of Computational Physics}, 217(1):100--122, 2006.

\bibitem{stevenson07}
Rob Stevenson.
\newblock Optimality of a standard adaptive finite element method.
\newblock {\em Found. Comput. Math.}, 7(2):245--269, 2007.

\bibitem{stevenson08}
Rob Stevenson.
\newblock The completion of locally refined simplicial partitions created by bisection.
\newblock {\em Math. Comp.}, 77(261):227--241, 2008.

\bibitem{teckentrup}
A.~L. Teckentrup, P.~Jantsch, C.~G. Webster, and M.~Gunzburger.
\newblock A multilevel stochastic collocation method for partial differential equations with random input data.
\newblock {\em SIAM/ASA J. Uncertain. Quantif.}, 3(1):1046--1074, 2015.

\bibitem{anova1}
Grace Wahba.
\newblock {\em Spline models for observational data}, volume~59 of {\em CBMS-NSF Regional Conference Series in Applied Mathematics}.
\newblock Society for Industrial and Applied Mathematics (SIAM), Philadelphia, PA, 1990.

\bibitem{wasilkowski1995explicit}
Grzegorz~W Wasilkowski and Henryk Wozniakowski.
\newblock Explicit cost bounds of algorithms for multivariate tensor product problems.
\newblock {\em Journal of Complexity}, 11(1):1--56, 1995.

\bibitem{scc3}
Dongbin Xiu and Jan~S. Hesthaven.
\newblock High-order collocation methods for differential equations with random inputs.
\newblock {\em SIAM J. Sci. Comput.}, 27(3):1118--1139, 2005.

\bibitem{sgc5}
Dongbin Xiu and George~Em Karniadakis.
\newblock Modeling uncertainty in steady state diffusion problems via generalized polynomial chaos.
\newblock {\em Comput. Methods Appl. Mech. Engrg.}, 191(43):4927--4948, 2002.

\bibitem{anova2}
Rong-Xian Yue and Fred~J. Hickernell.
\newblock Designs for smoothing spline {ANOVA} models.
\newblock {\em Metrika}, 55(3):161--176, 2002.

\end{thebibliography}
\end{document}